\documentclass{amsart}

\usepackage{amsmath}
\usepackage{amssymb}
\usepackage{comment}
\usepackage{hyperref}

\newtheorem{theorem}{Theorem}[section]
\newtheorem{lemma}[theorem]{Lemma}
\newtheorem{corollary}[theorem]{Corollary}
\newtheorem{proposition}[theorem]{Proposition}

\theoremstyle{definition}
\newtheorem{definition}[theorem]{Definition}

\theoremstyle{remark}
\newtheorem{remark}[theorem]{Remark}

\newcommand\R{\mathbb{R}}
\newcommand\T{\mathbb{T}}
\newcommand\C{\mathbb{C}}
\newcommand\area{\mathrm{Area}}
\newcommand\dist{\mathrm{dist}}
\newcommand\supp{\mathrm{supp}}
\newcommand\tr{\mathrm{tr}}

\numberwithin{equation}{section}

\begin{document}

\title[Symplectic non-squeezing for KdV on $\mathbb R$]{Symplectic non-squeezing for the KdV flow on the line}

\author{Maria Ntekoume}
\address{Department of Mathematics, University of California, Los Angeles, CA 90095, USA}

\email{mntekoume@math.ucla.edu}





\begin{abstract}
We show symplectic non-squeezing for the KdV equation on the line $\R$. This is achieved via finite-dimensional approximation. Our choice of finite-dimensional Hamiltonian system that effectively approximates the KdV flow is inspired by the recent breakthrough in the well-posedness theory of KdV in low regularity spaces, relying on its completely integrable structure (\cite{KV18}). The employment of our methods also provides us with a new concise proof of symplectic non-squeezing for the same equation on the circle $\T$, recovering the result of \cite{CKSTT}. 
\end{abstract}

\maketitle

\section{Introduction}

We consider the real-valued Korteweg--de Vries equation on the line:
\begin{equation} \label{KdV} \tag{KdV}
    \frac{d}{dt} q = -q'''+ 6q q'.
\end{equation}
This equation was introduced over a hundred years ago in \cite{KdV} to describe the evolution of long waves in shallow channels of water. In particular, it sought to mathematically explain the observation of solitary waves. Its physical significance, along with its unique mathematical features, have captured the interest of researchers ever since. Firstly, the KdV equation is a Hamiltonian evolution. It is also one of the most prominent examples of a completely integrable system.

The setting in which \eqref{KdV} can be viewed to be Hamiltonian is that of a symplectic Hilbert space. Consider a Hilbert space $\mathbb H$ and $\omega$ a symplectic form on $\mathbb{H}$, that is, a nondegenerate, antisymmetric form $\omega:\mathbb{H} \times \mathbb{H}\to \mathbb{C}$. A Hamiltonian function $H: \mathbb{H}\to \mathbb{R}$ gives rise to dynamics of the form
\begin{align*}
    \dot q= X_H(q).
\end{align*}
Here $X_H$ is the vector field given by
\begin{align*}
    \omega(v,X_H(q))=\frac{d}{d\varepsilon}|_{\varepsilon=0} H(q+\varepsilon v).
\end{align*}

Indeed, the KdV equation is, formally at least, the Hamiltonian evolution in the symplectic space $$\dot H^{-\frac 12}(\mathbb R):=\left\{q: \int_{\mathbb R}q(x) dx=0, \int_{\mathbb R} \frac{|\hat q(\xi)|^2}{|\xi|} d\xi<\infty  \right\}$$
associated to the Hamiltonian
$$H_{KdV}(q)=\int_{\mathbb R} \frac 1 2 (q')^2 +q^3 dx.$$
Considering the symplectic form
\begin{align*}
    \omega_{-\frac 12}(u,v)= \int_{\mathbb R} u \partial_x ^{-1} v dx
\end{align*} 
in the Hilbert space $\dot H^{-\frac 12}(\mathbb R)$, simple calculations show that
\begin{align*}
    \frac{d}{d\varepsilon}|_{\varepsilon=0} H_{KdV}(u+\varepsilon v)&= \int_{\mathbb R} u'v'+3u^2 v dx=\int_{\mathbb R} (-u''+3u^2) v dx\\
    &= \int_{\mathbb R} \partial_x^{-1}(-u'''+6u u') v dx \\
    &= \omega_{-\frac 12} (v, -u'''+6u u').
\end{align*}

In the context of Hamiltonian mechanics, Liouville's theorem asserts that Hamiltonian flows preserve phase-space volume. This implies that in order for a Hamiltonian evolution to flow one region of space to another, the volume of the latter must exceed that of the former. The natural question that emerges is whether preservation of volume is the only obstruction for the existence of a symplectomorphism between two regions. Although this is the case in one (complex) dimension, the situation is much different in higher dimensions. 

\begin{theorem} [Gromov, \cite{Gromov}]
\label{gromov}
Fix $z \in \C^n$, $l\in \C^n$ with unit length, $\alpha \in \C$, and $0<r<R<\infty$. Let $B(z, R)$ denote the ball of radius $R$ centered at $z$ and suppose $\phi: B(z, R) \to \C^n$ is a smooth symplectomorphism (with respect to the standard structure). Then there exists $\zeta \in B(z, R)$ so that
$$|\langle l, \phi(\zeta)\rangle -\alpha|>r.$$
\end{theorem} 

In other words, a smooth symplectomorphism cannot map a ball wholly inside a cylinder of smaller radius, despite the fact that the volume of the ball is finite and the volume of the cylinder is infinite. In particular, this indicates that, although symplectomorphisms preserve volume, it is far more restrictive to be a symplectic transformation than to be volume-preserving. Symplectic non-squeezing can also be viewed as a classical analogue of the uncertainty principle; if a collection of particles initially spread out all over a ball, then one cannot squeeze the collection into a statistical state in which the momentum and position in some direction spread out less than initially.

The main goal of this paper is to show that the analogue of Gromov's theorem holds for the infinite-dimensional dynamics associated to KdV on the line. 

\begin{theorem}
\label{sns KdV R}
Let $z\in \dot H^{-\frac 1 2}(\mathbb R)$, $l\in H^{\frac 1 2} (\mathbb R)$ with $\|l\|_{\dot H^{\frac 1 2}(\mathbb R)}=1$, $\alpha\in \mathbb C$, $0<r<R<\infty$, and $T>0$. Then there exists $q_0 \in \{ q\in \dot H^{-\frac 1 2}(\mathbb R) : \|q-z\|_{\dot H^{-\frac 1 2}(\mathbb R)}< R\}$ such that the solution $q$ to \eqref{KdV} with initial data $q(0)=q_{0}$ satisfies 
$$|\langle l, q(T)\rangle -\alpha|>r.$$
\end{theorem} 

As an offshoot of our methods, we also obtain a much simpler proof for the known corresponding result on the circle.

\begin{theorem}
\label{sns KdV T}
Let $z\in \dot H^{-\frac 12}(\mathbb T)$, $l\in  H^{\frac 12} (\mathbb T)$ with $\|l\|_{\dot H^{\frac 1 2}(\mathbb T)}=1$, $\alpha\in \mathbb C$, $0<r<R<\infty$, and $T>0$. Then there exists $q_0 \in \{ q\in \dot H^{-\frac 12}(\mathbb T) : \|q-z\|_{\dot H^{-\frac 1 2}(\mathbb T)}< R\}$ such that the solution $q$ to \eqref{KdV} with initial data $q(0)=q_0$ satisfies 
$$|\langle l, q(T)\rangle -\alpha|>r.$$
\end{theorem}

While it is tempting to believe that any property that holds for any finite-dimensional Hamiltonian system must carry over to infinite-dimensional systems, no such universal result is known. Although for infinite-dimensional PDE finite-time blow-up can occur despite exact conservation of norm, there is no analogue of this phenomenon in finite dimensions; blow-up for ODEs occurs only via norm blow-up. Another key distinction is that in $\mathbb{C}^n$ the weak and norm topologies coincide. This is not true in infinite-dimensional Hilbert spaces. This is one of the intriguing aspects of non-squeezing in this setting, as it commingles these two topologies: the initial data is measured in norm, but the solution at time $T$ is examined from the standpoint of the weak topology.

Although no universal analogue of Theorem \ref{gromov} is known in infinite dimensions, symplectic non-squeezing has been proved for several Hamiltonian PDE. The proof of these results always depends heavily on the model in question. The study of symplectic non-squeezing for Hamiltonian PDE was initiated by Kuksin. In \cite{Ku95} he proved an analogue of Gromov's theorem for flows that consist of a linear operator and a compact smooth operator and provided several examples that fall into this framework. This is made possible by using finite-dimensional approximation to extend the Hofer-Zehnder capacity to infinite dimensional Hilbert spaces and thus obtain an infinite-dimensional symplectic invariant. Subsequently, Kuksin's result was used to prove symplectic non-squeezing for several other models (\cite{Bou95}, \cite{Rou10}).

Bourgain later proved symplectic non-squeezing for a flow that is not amenable to Kuksin's treatment, namely the defocusing cubic NLS on the one-dimensional torus (\cite{Bou94}). Finite-dimensional approximation is once again at the core of his method; he introduces a (sharp) Fourier cutoff on the nonlinearity to approximate the full equation by a finite-dimensional system and applies Gromov's result to this approximate flow. 

Similar methods were employed in \cite{CKSTT} to obtain symplectic non-squeezing for the KdV flow on the circle. One of the main difficulties of this problem is that, although the KdV equation is subcritical in $\dot H^{-\frac 1 2}(\mathbb T)$, it does represent the end-point regularity for strong notions of well-posedness on the torus. The authors show that projecting the equation crudely as in the argument of Bourgain does not provide a good approximation and they adopt a smooth frequency truncation instead. Their method utilizes the Miura transform to pass to mKdV and take advantage of its better smoothing properties.

A simpler proof for the symplectic non-squeezing result for KdV was discovered by Hong and Kwon (\cite{HKwon}). Instead of the Miura map, they use a normal form transformation to obtain the result for the KdV as well as for the coupled KdV flow on the torus.

Higher-order KdV-type equations were treated in \cite{HKwak}. The authors use an argument based on Bourgain's sharp projection approximation, exploiting the better modulation effect from the non-resonant interaction that these equations enjoy compared to the KdV flow in order to avoid the use of the Miura map or the need for a smooth frequency truncation. Bourgain's method was also followed in \cite{Kwak} to prove that the forth-order cubic NLS on the one-dimensional torus has the symplectic non-squeezing property.

In \cite{M} Mendelson obtained symplectic non-squeezing results for the cubic nonlinear Klein--Gordon equation on the three-torus. This is a critical result in the sense that the regularity needed to define the symplectic form coincides with the scaling critical regularity of the equation. The proof relies again on a finite-dimensional approximation, using a combination of deterministic and probabilistic techniques.

More recently, non-squeezing was considered in infinite volume. This significant advance was made in the work of Killip, Visan, and Zhang (\cite{KVZ19}, \cite{KVZ16}) where they proved symplectic non-squeezing for the cubic NLS on $\mathbb R$ and $\mathbb R^2$. Note that the latter is also the first unconditional critical result. Their method relies on approximating the full flow by higher and higher frequency-truncated versions of the equation posed on larger and larger tori. For instance, the finite-dimensional systems used to approximate the cubic NLS on the line were
\begin{align*}
    i\partial u+\Delta u= P_{\leq N_n}\left(|P_{\leq N_n}u|^2 P_{\leq N_n}u\right)\qquad\text{on}\,\,\mathbb R/ L_n \mathbb Z
\end{align*}
with $N_n, L_n\to\infty$ as $n\to \infty$.

This novel strategy for proving symplectic non-squeezing results in the infinite volume setting has already proved to be extremely effective in several Hamiltonian PDE. For instance, in \cite{Miao} these techniques were implemented for certain mass subcritical forth-order Schr\"odinger equations on $\mathbb R$, while in \cite{Yang} Yang adapted this method to address all mass subcritical Hartree equations on $\mathbb R^d$ for $d\geq 2$. 

While the innovative techniques in \cite{CKSTT, HKwon} were successful in proving non-squeezing on the torus, the challenging question of establishing a similar result on the line was left open. The torus assumption is essential in both arguments, suggesting the need for new ideas in the infinite volume setting. In our treatment these new ideas stem from the completely integrable structure of the KdV equation.

The modern theory of complete integrability has its roots in the discovery by Gardner, Greene, Kruskal, and Miura in \cite{GGKM} of a method for solving the initial-value problem of the KdV equation with rapidly decaying initial data. This was achieved by drawing the connection between this problem and the spectral and scattering theory of the Schr\"odinger operator 
$$L(t)=-\partial_x^2 + q(x,t).$$ A more elegant expression of this connection and the integrable structure of KdV was provided soon afterwards by Lax in the form of the Lax pair
\begin{align*}
    L(t)&= -\partial_x^2 +q(t),\\
    P(t)&= -4 \partial_x^3+3 \partial_x q(t)+3 q(t)\partial_x;
\end{align*}
one can easily verify that $q$ is a solution to \eqref{KdV} if and only if
\begin{align*}
    \frac{d}{dt}L(t)= [P(t),L(t)].
\end{align*}

Since $P(t)$ is anti-self-adjoint, at every time $t$ the unitary operator $U(t)$ given by
\begin{align*}
    \frac{d}{dt}U(t)=P(t)U(t),\qquad U(0)=\text{Id}
\end{align*}
satisfies 
\begin{align*}
    L(t)=U(t)L(0)U(t)^\ast.
\end{align*}
This suggests that \eqref{KdV} preserves the spectral properties of $L$. We therefore recover an infinite sequence of conservation laws for the KdV, which can be expessed in the form of polynomials of $q$ and its derivatives, first discovered in \cite{MGK68} by Gardner, Kruskal, and Miura. For example, the first three of these conserved quantities are 
\begin{align*}
    M(q)=\int q(x)dx,\quad P(q)=\int \frac 12 q(x)^2 dx,\quad H_{KdV}(q)=\int \frac 12 q'(x)^2 +q(x)^3 dx
\end{align*}
and describe the conservation of mass, momentum, and energy, respectively. A powerful
consequence of these conservation laws is that, for every non-negative integer $m$, the $H^m$ norm of smooth solutions to \eqref{KdV} is bounded uniformly in time in terms of the corresponding norm of their initial data.

Even more significant for our work are the new low regularity conservation laws recently discovered independently both by Koch and Tataru \cite{KT}, as well as by Killip, Visan, and Zhang \cite{KVZ17}. Following the exposition of the latter, the authors observed that the perturbation determinant
\begin{align*}
    \det\left( \frac{-\partial_x^2 +q(t)+\kappa^2}{-\partial_x^2 +\kappa^2} \right)=\det\left( 1+\sqrt{R_0(\kappa)}q \sqrt{R_0(\kappa)}\right),\qquad R_0(\kappa)=(-\partial_x^2 +\kappa^2)^{-1}
\end{align*}
can be used to encapsulate the preservation of the spectral properties of the Lax operator $L$ by \eqref{KdV}.
After further renormalizations, the authors proposed the quantities
\begin{align*}
    \alpha(\kappa;q):= -\log \det\left(1+\sqrt{R_0(\kappa)}q \sqrt{R_0(\kappa)}\right) +\frac{1}{2\kappa}M(q)
\end{align*}
and proved that they are indeed conserved under the KdV flow for all large enough $\kappa\geq 1$. One of the main reasons behind this renormalization is that it allows us to express the conserved quantities in the series expansion
\begin{align*}
    \alpha(\kappa;q)= \sum_{l=2} ^\infty \frac{(-1)^l}{l} \tr\left\{(\sqrt{R_0}q\sqrt{R_0})^l\right\}.
\end{align*}
As it turns out, $\alpha(\kappa;q)$ captures the $H^{-1}$ norm of the part of $q$ that lives at frequencies $|\xi|\gtrsim \kappa$.

Furthermore, this new discovery was the catalyst for progress in the well-posedness theory of \eqref{KdV} in low regularity spaces. 
Building upon their work in \cite{KVZ17}, Killip and Visan established in \cite{KV18}  well-posedness in the optimal regularity space $H^{-1}$ both on the line and the circle, in the sense that the solution map extends uniquely from Schwartz space to a jointly continuous map
$\Phi : \mathbb R \times  H^{-1}\to H^{-1}$ (see Theorem \ref{KdV wp}). In the case of the torus this result had already been obtained by Kappeler and Topalov \cite{KapTop}. It is worth noting that it is sharp, as it has been shown in \cite{Mol1, Mol2} that continuous dependence on the initial data cannot hold in $H^s$ for $s<-1$ in both geometries, thus precluding well-posedness below $H^{-1}$.

Their method relies on considering the evolutions induced by the Hamiltonians
\begin{align*}
    H_\kappa(q):= -16 \kappa^5 \alpha(\kappa;q)+ 4\kappa^2 P(q).
\end{align*}
Observing that $\alpha(\kappa;q)-\frac{1}{2\kappa}M(q)$ acts as a generating function for the polynomial conserved quantities, the asymptotic expansion
\begin{align*}
    \alpha(\kappa;q)=\frac{1}{4\kappa^3}P(q)- \frac{1}{16 \kappa^5}H_{KdV}+O(\kappa^{-7})
\end{align*}
suggests that the Hamiltonian $H_\kappa$ can serve as a good approximation to $H_{KdV}$. They prove that the family of Poisson commuting flows associated with the Hamiltonians $H_\kappa$ are globally well-posed in $H^{-1}$, commute with the KdV flow, and the flow induced by the difference $H_{KdV} - H_{\kappa}$ is close to the identity in the $H^{-1}$ metric for $\kappa$ large on bounded time intervals, thus making it possible to obtain well-posedness for KdV by a limiting process.

Inspired by the innovations introduced in \cite{KV18} that we described above, we consider a new approach to the study of the symplectic behavior of \eqref{KdV}. In order to investigate whether the KdV flow can be accurately approximated by finite-dimensional models, we employ the approximation of \eqref{KdV} by the commuting flows induced by the Hamiltonians $H_\kappa$. This can be seen as a truncation in the Hamiltonian. Furthermore, the proof of Theorem \ref{sns KdV R} utilizes a finite-dimensional approximation by truncation in frequency as well as in space, similar to the one in \cite{KVZ19,KVZ16}. The approximation result we get allows us to transfer results for finite-dimensional systems, such as non-squeezing, to the PDE setting.

Firstly, we focus our efforts on proving a symplectic non-squeezing result for the $H_\kappa$ flows via approximation by finite-dimensional systems. This is a much easier task than proving this property for \eqref{KdV}; it turns out that these equations have a Lipschitz nonlinearity and enjoy finite speed of propagation. More specifically, for every fixed sufficiently large $\kappa$ we consider the frequency-truncated Hamiltonians
\begin{align*}
H_\kappa^{n}(q)=-16\kappa^5 \alpha(\kappa;P^{L_n}_{m_n<\dots \leq M_n} q)+4\kappa^2 P(q)
\end{align*}
and the induced Hamiltonian flows on the torus $\mathbb R/L_n \mathbb Z$ for appropriate parameters $m_n\to 0$, $M_n\to\infty$, and $L_n\to\infty$ as $n\to\infty$. As finite-dimensional systems, they demonstrate symplectic non-squeezing. If $u_n$ are the witnesses to non-squeezing for the $H_\kappa^n$ flows given by Gromov's theorem, we wish to extract a `limit' and show that it is, indeed, a witness to non-squeezing for the $H_\kappa$ flow. Unlike what was done in \cite{KVZ16} and \cite{KVZ19}, the removal of the frequency truncation and the extension from the large circle to the line happen simultaneously.

Of course, this limiting process is very delicate. We are working with a sequence $u_n(0)$ of periodic initial data with different periods, so there is no appropriate space that contains all the elements, making it impossible to talk about a limit. Our remedy is to find compactly supported initial data $q_n(0)\in \dot H^{-\frac 12}(\mathbb R)$, with potentially larger and larger support containing the origin, that are `close' to a copy of one period of $u_n(0)$. This is achieved by viewing $u_n(0)$ as a function on the circle of circumference $L_n$, locating a small subinterval where $u_n(0)$ is `not too large', and `cutting' there in a smooth fashion to `unwrap' on the line and obtain a new sequence, $q_n(0)$. 

Despite sounding deceivingly simple, this process is one of the most vexing parts of the proof. First of all, the cut must be done in a way that establishes not only that $q_n(0)$ is `close' to $u_n(0)$, but also that the evolution of $q_n(0)$ remains a good approximation to $u_n(t)$ on the time interval $[0,T]$. To that end, one needs to avoid cutting in a location where there is a bubble of concentration for $u_n(0)$. It is also crucial to cut along a sufficiently large subinterval in order to ensure that the periodic solution will not have a chance to wrap around the torus and reinforce itself. On a technical level, the main difficulties stem from the fact that we have to work in the fractional Sobolev space $\dot H^{-\frac 1 2}$ which is non-local. For instance, even for a smooth cutoff $\chi\in C_c^\infty(\mathbb R)$ it is possible that $\chi u_n(0)$ is not in $\dot H^{-\frac 12}(\mathbb R)$. 

The difficulty of locating an appropriate cutting interval informs our choice of finite-dimensional approximating systems, which differs from the preceding works we discussed. The removal of low frequencies is not an arbitrary decision, but the decisive and indispensable technical step that allows the process described above to run smoothly (see the proof of Lemma \ref{bump}). One may worry that this choice could affect the well-posedness of the induced flows or obstruct the convergence to the full equation, especially since we are working in negative regularity spaces. Nevertheless, these concerns turn out to be unwarranted, thanks to the presence of a derivative in the nonlinearity of \eqref{eq:Hk}. 
Moreover, after carrying out the delicate `cutting' and `unwrapping' process carefully, we ensure not only that $q_n(0)$ is `close' to $u_n(0)$, but also that the flow on the line with initial data $q_n(0)$ stays `close' to $u_n$.

The new sequence of initial data $q_n(0)$ is bounded in $\dot H^{-\frac 12}(\mathbb R)$, permitting us to get a weak limit. This indicates that we need to understand the behavior of the $H_\kappa$ flow in the weak topology. In particular, we prove that for a sequence of initial data that is bounded in $\dot H^{-\frac 12}(\mathbb R)$, passing to a subsequence the corrresponding solutions to the $H_\kappa$ flow converge weakly in $H^{-1}(\mathbb R)$ to a solution to the same equation uniformly on any compact time interval containing $0$. The weak limit extracted through this process is expected to be the desired witness to non-squeezing for the $H_\kappa$ flow on the line.

Establishing the non-squeezing property for the $H_\kappa$ flows constitutes an essential step towards proving our main theorem. As we discussed earlier, these flows provide a good approximation for KdV. For a bounded in $\dot H^{-\frac 1 2}(\mathbb R)$ sequence of initial data, the difference of the corresponding solutions to the KdV and the $H_\kappa$ flow converges to $0$ in $H^{-1}$ uniformly on a compact time interval $[0,T]$ and for all elements of the sequence as $\kappa\to\infty$ (see Lemma \ref{Q3}). Consequently, we can rely on this approximation to prove that KdV inherits this property.

The key ingredient to do so is prove a weak well-posendess result for \eqref{KdV} in $H^{-1}$. Indeed, we show that, if $q_n(t)$ is a bounded in $\dot H^{-\frac 12}(\mathbb R)$ sequence of Schwartz solutions to \eqref{KdV} and $q_n(0)$ converges weakly in $H^{-1}$ to $q(0)$, then $q_n(T)$ converges weakly in $H^{-1}$ to the solution to \eqref{KdV} with initial data $q(0)$ for all $T>0$. 

Understanding the behavior of \eqref{KdV} in the weak topology is an interesting question on its own, and one that has already received attention. In particular, in \cite{CK} Cui and Kenig proved a weak continuity result in the Sobolev space $H^{-\frac 34}(\mathbb R)$, but the problem in lower regularity spaces remained unanswered. These authors were motivated by potential applications of this property to finite time blow-up and asymptotic stability of solitary waves. It also becomes apparent that studying the behavior of weak limits under the flow is fundamental for symplectic non-squeezing. We would add that the weak topology better represents what one may achieve in actual experiments: One cannot wholly suppress high-frequncy noise, nor employ an infinite domain. One can only make finitely many measurements.

There are a few simplifications that are possible in the torus setting. First of all, our finite-dimensional approximation for the $H_\kappa$ flow can comprise of only a frequency truncation of the full equation, as we are working on finite volume to begin with. Moreover, the witnesses to non-squeezing we obtain by Gromov's result for these finite-dimensional models all lie in the same space. Finally, the finite volume setting allows us to avoid working in the weak topology. 

Besides providing a simpler proof for the result of \cite{CKSTT} and allowing us to deal with the infinite volume setting, our methods can take us even further. As we mentioned earlier, one of the most crucial results in our work is that the KdV flow respects weak $H^{-1}$ limits on $\dot H^{-\frac 12}$ bounded subsets. We can however get this result for a larger class of subsets of $H^{-1}$. By proving this for all bounded equicontinuous subsets of $H^{-1}$, we gain access to the following stronger results, where the center of the ball is merely in $H^{-1}$. 

\begin{theorem}
\label{sns KdV R ell}
Let $z\in H^{-1}(\mathbb R)$, $l\in H^{1} (\mathbb R)$ with $\|l\|_{\dot H^{\frac 1 2}(\mathbb R)}=1$, $\alpha\in \mathbb C$, $0<r<R<\infty$, and $T>0$. Then there exists $q_0 \in \{ q\in H^{-1}(\mathbb R) : \|q-z\|_{\dot H^{-\frac 1 2}(\mathbb R)}< R\}$ such that the solution $q$ to \eqref{KdV} with initial data $q(0)=q_{0}$ satisfies 
$$|\langle l, q(T)\rangle -\alpha|>r.$$
\end{theorem} 
 
\begin{theorem}
\label{sns KdV T ell}
Let $z\in H^{-1}(\mathbb T)$, $l\in  H^{1} (\mathbb T)$ with $\|l\|_{\dot H^{\frac 1 2}(\mathbb T)}=1$, $\alpha\in \mathbb C$, $0<r<R<\infty$, and $T>0$. Then there exists $q_0 \in \{ q\in H^{-1}(\mathbb T) : \|q-z\|_{\dot H^{-\frac 1 2}(\mathbb T)}< R\}$ such that the solution $q$ to \eqref{KdV} with initial data $q(0)=q_0$ satisfies 
$$|\langle l, q(T)\rangle -\alpha|>r.$$
\end{theorem}

Note that, even in the compact setting, this is a new result.

Looking back at Gromov's Theorem, it asserts that a smooth symplectomorphism cannot map a ball of radius $R$ into a cylinder with cross-section $r<R$. The next natural question is to ask what happens if we consider $r=R$. Moreover, one cannot help but wonder if it is necessary to restrict our interest to cylinders with a circular cross-section. As it turns out, the following statements are equivalent to Gromov's Theorem (Remark 1.4, \cite{KVZ16}):
\begin{enumerate}
    \item A symplectomorphism cannot map a closed ball into an open cylinder of the same radius.
    \item A symplectomorphism cannot map a ball of radius $R$ into a cylinder whose cross-section has area less than $\pi R^2$.
\end{enumerate}
The robustness of our methods can also be manifested by allowing us to obtain analogous alternate formulations of symplectic non-squeezing for \eqref{KdV}, both on the circle and the line.


\subsection*{Acknowledgements} I would like to thank my advisors, Rowan Killip and Monica Visan, for suggesting this problem and for their invaluable support and guidance. 
This work was supported in part by NSF grants DMS-1600942 (Rowan Killip), DMS-1500707 and DMS-1763074 (Monica Visan).

\section{Preliminaries}

We adopt the following convention for the Fourier transform and the inverse Fourier transform:
\begin{align*}
    \hat f(\xi)&=\int_{\R} e^{-2\pi i\xi x} f(x)dx,\qquad \xi\in \mathbb R\\
    f(x)&= \int_{\R} e^{2\pi i\xi x} \hat f(\xi)d\xi
\end{align*}
for functions on the line and
\begin{align*}
    \hat f(k)&=\frac{1}{\ell} \int_0^\ell e^{-2 \pi i k x} f(x)dx, \qquad k\in \frac {1}{\ell}\mathbb Z\\
    f(x)&= \sum_{k\in \frac {1}{\ell}\mathbb Z} e^{2\pi i k x} \hat f(k)
\end{align*}
for functions on the circle $\T_\ell=\R / (\ell \mathbb Z)$. Plancherel's theorem asserts that
\begin{align*}
    \|f\|_{L^2(\mathbb R)} ^2&:= \int_{\mathbb R} |f(x)|^2 dx = \int_{\mathbb R} |\hat f(\xi)|^2 d\xi,\\
    \|f\|_{L^2(\mathbb T_\ell)} ^2&:= \int_0^\ell |f(x)|^2 dx = \ell \sum_{k\in \frac {1}{\ell}\mathbb Z} |\hat f(k)|^2.
\end{align*}
The above give rise to the definition of the $H^s$ and $\dot H^s$ Sobolev norms for $s\in\mathbb{R}\setminus\{0\}$
$$\|f\|^2_{H^s(\R)}=\int_{\R} |\hat f(\xi)|^2 (1+|\xi|^2)^s d\xi,\qquad \|f\|^2_{\dot H^s(\R)}=\int_{\R} |\hat f(\xi)|^2 |\xi|^{2s} d\xi$$
and 
$$\|f\|^2_{H^s(\T_\ell)}=\ell \sum_{k\in \frac {1}{\ell}\mathbb Z} |\hat f(k)|^2 (1+|k|^2)^s,\qquad \|f\|^2_{\dot H^s(\T_\ell)}=\ell \sum_{k\in \frac {1}{\ell}\mathbb Z\setminus \{0\}} |\hat f(k)|^2 |k|^{2s}.$$
In light of these definitions, the Sobolev spaces are given by
\begin{align*}
    H^s(\mathbb R):= \{f: \|f\|_{H^s(\mathbb R)}<\infty\}, \qquad  H^s(\mathbb T_\ell):= \{f: \|f\|_{H^s(\mathbb T_\ell)}<\infty\}
\end{align*}    
for all $s\in\mathbb R\setminus\{0\}$,
\begin{align*}
    \dot H^s(\mathbb R):= \{f: \|f\|_{\dot H^s(\mathbb R)}<\infty\}, \qquad  \dot H^s(\mathbb T_\ell):= \{f: \|f\|_{\dot H^s(\mathbb T_\ell)}<\infty\}
\end{align*}    
for $s>0$,
\begin{align*}
    \dot H^{s}(\mathbb R):= \{f: \|f\|_{\dot H^{s}(\mathbb R)}<\infty\}, \qquad
    \dot H^{s}(\mathbb T_\ell):= \{f: \hat f(0)=0,\, \|f\|_{\dot H^{s}(\mathbb T_\ell)}<\infty\}
\end{align*}
for $s<0$.

We use the standard Littlewood-Paley operators $P_{\leq N}$, $P_{>N}$, $P_{N<\dots\leq M}$ for functions on the line given by
\begin{align*}
    \widehat{ P_{\leq N}f}(\xi)&:= m(\frac{\xi}{N}) \hat f(\xi),\\
    \widehat{ P_{> N}f}(\xi)&:= \big(1-m(\frac{\xi}{N})\big) \hat f(\xi),\\
    P_{N<\dots\leq M}f &:= P_{\leq M}f-P_{\leq N}f
\end{align*}
and $P^\ell_{\leq N}$, $P^\ell_{>N}$, $P^\ell_{N<\dots\leq M}$ for functions on the circle $\mathbb T_\ell=\mathbb R/(\ell \mathbb Z)$ given by
\begin{align*}
    \widehat{ P^\ell_{\leq N}f}(k)&:= m(\frac{k}{N}) \hat f(k),\\
    \widehat{ P^\ell_{> N}f}(k)&:= \big(1-m(\frac{k}{N})\big) \hat f(k),\\
    P^\ell_{N<\dots\leq M}f &:= P^\ell_{\leq M}f-P^\ell_{\leq N}f
\end{align*}
for $N,M\in 2^\mathbb Z$, where $m\in C_c^\infty(\R)$ is a radial bump function supported in $[-2,2]$ and equal to 1 on $[-1,1]$. Like all Fourier multipliers, the Littlewood-Paley operators commute with differential operators as well as with other Fourier multipliers. They also obey the following estimates.

\begin{lemma}
[Bernstein estimates]
For $s\in\mathbb R$, $\sigma>0$, $N\in 2^{\mathbb Z}$

\begin{align*}
    \|P_{\leq N}f\|_{H^s(\mathbb R)}+\|P_{> N}f\|_{H^s(\mathbb R)}&\lesssim \|f\|_{H^s(\mathbb R)},\\
    \|P_{\leq N}f\|_{\dot H^s(\mathbb R)}+\|P_{> N}f\|_{\dot H^s(\mathbb R)}&\lesssim \|f\|_{\dot H^s(\mathbb R)},\\
    \||P_{\leq N}f\|_{\dot H^s(\mathbb R)}&\lesssim N^\sigma \|P_{\leq N}f\|_{\dot H^{s-\sigma}(\mathbb R)},\\
    \|P_{>N}f\|_{\dot H^s(\mathbb R)}&\lesssim N^{-\sigma} \| P_{> N}f\|_{\dot H^{s+\sigma}(\mathbb R)}.
\end{align*}
The analogous estimates also hold on the circle.
\end{lemma}

\subsection{Diagonal Green's function} \label{sec;green}

In this section we summarize some properties of the Green's function associated with the Schr\"odinger operator
\begin{align*}
    L=-\partial_x^2 +q
\end{align*}
both on the line and on the circle. We follow the exposition of \cite{KV18}, where proofs of the statements in this section can be found. 

We begin with the line setting. For $\delta>0$ small, we denote
\begin{align*}
    B_\delta:= \{q\in H^{-1}(\mathbb R): \|q\|_{H^{-1}(\mathbb R)}\leq\delta\}.
\end{align*}

\begin{proposition}
Given $q \in H^{-1}(\mathbb R)$, there exists a unique self-adjoint operator $L$ associated to the quadratic form 
\begin{align*}
    \psi \mapsto \int |\psi '(x)|^2 +q(x) |\psi(x)|^2 dx
\end{align*}
with domain $H^1(\mathbb R)$. It is semi-bounded. Moreover, for $\delta \leq \frac 12$ and $q \in B_\delta$, the resolvent is given by the norm-convergent series
\begin{align*}
   R(\kappa):= (L+\kappa^2)^{-1} = \sum_{l=0}^\infty (-1)^l \sqrt{R_0} (\sqrt{R_0} q \sqrt{R_0})^l \sqrt{R_0}
\end{align*}
for all $\kappa \geq 1$, where $R_0$ denotes the resolvent in the case $q=0$
\begin{align*}
    R_0(\kappa)=(-\partial_x^2 +\kappa^2)^{-1}.
\end{align*}
\end{proposition}

It is worth noting that in the heart of the proof of the above Proposition is the crucial estimate
\begin{align} \label{op norm}
    \|\sqrt{R_0}q\sqrt{R_0}\|_{\text{op}}^2\leq \|\sqrt{R_0}q\sqrt{R_0}\|_{\mathfrak I_2}^2 = \frac{1}{\kappa} \int \frac{|\hat q(\xi)|^2}{\xi^2+4\kappa^2}d\xi.
\end{align}

\begin{proposition}[Diffeomorphism property] \label{diffeo R}
There exists $\delta > 0$ so that the following are true for all $\kappa \geq 1$: For each $q \in B_\delta$, the resolvent $R$ admits a continuous integral kernel $G(x, y; \kappa; q)$; thus, we may unambiguously define the diagonal Green's function
\begin{align*}
    g(x; \kappa; q) := G(x, x; \kappa; q),
\end{align*}
which is given by the uniformly convergent series
\begin{align}\label{series}
     g(x;\kappa; q)= \frac{1}{2\kappa} +\sum_{l=1}^\infty (-1)^l \left\langle \sqrt{R_0} \delta_x, (\sqrt{R_0} q \sqrt{R_0})^l \sqrt{R_0} \delta_x\right\rangle
\end{align}
where inner products are taken in $L^2(\mathbb R)$.
Moreover, the mapping
\begin{align*}
    q\mapsto g-\frac{1}{2\kappa}
\end{align*}
is a (real analytic) diffeomorphism of $B_\delta$ into $H^1 (\mathbb R)$. In particular, for all $q, \tilde q \in B_\delta$
\begin{align}
    \|g(q)-g(\tilde q)\|_{H^1(\mathbb R)}&\lesssim \|q-\tilde q\|_{H^{-1}(\mathbb R)},\\
    \left\|g(q)-\frac{1}{2\kappa}\right\|_{H^1(\mathbb R)}&\lesssim \|q\|_{H^{-1}(\mathbb R)}.
\end{align}
The implicit constants do not depend on $\kappa$.
\end{proposition}

On the circle $\mathbb T_\ell=\mathbb R/(\ell\mathbb Z)$, following the exposition of \cite{KV18}, we choose to consider the Lax operator $L= -\partial_x^2 +q$ acting on $L^2(\mathbb R)$ with periodic coefficients rather than on $L^2(\mathbb T_\ell)$. Here we impose working with potentials in the small ball
\begin{align*}
    B_{\delta,\kappa}^\ell:= \{q\in H^{-1}(\mathbb T_\ell): \kappa^{-\frac 12}\|q\|_{H^{-1}(\mathbb T_\ell)}\leq\delta\}
\end{align*}
for $\delta>0$ small.

\begin{proposition}
\label{diffeo T}
Given $q \in H^{-1}(\mathbb T_\ell)$, there is a unique self-adjoint operator $L$ acting on $L^2(\mathbb R)$ associated to the semi-bounded quadratic form
\begin{align*}
    \psi \mapsto \int |\psi '(x)|^2 +q(x) |\psi(x)|^2 dx.
\end{align*}
There exists $\delta > 0$, so that for all $\kappa\geq 1$ the following are true: For each $q\in B_{\delta,\kappa}^\ell$ the resolvent $R := (L + \kappa^2)^{-1}$ admits a continuous integral kernel $G(x, y; \kappa; q)$ given by the uniformly convergent series
\begin{align*}
    G(x, y; \kappa; q) = \frac{1}{2\kappa} e^{-\kappa|x-y|} +\sum_{l=1}^{\infty} (-1)^l \left\langle \sqrt{R_0} \delta_x, (\sqrt{R_0} q \sqrt{R_0})^l \sqrt{R_0} \delta_y \right\rangle.
\end{align*}
Thus, we may unambiguously define the diagonal Green's function
\begin{align*}
    g(x; \kappa; q) := G(x, x; \kappa; q).
\end{align*}
Moreover, the mapping
\begin{align*}
    q\mapsto g-\frac{1}{2\kappa}
\end{align*}
is a (real analytic) diffeomorphism of $B_{\delta,\kappa}^\ell$ into $H^1 (\mathbb T_\ell)$. In particular, for all $q, \tilde q \in B_{\delta,\kappa}^\ell$
\begin{align}
    \|g(q)-g(\tilde q)\|_{H^1(\mathbb T_\ell)}&\lesssim \|q-\tilde q\|_{H^{-1}(\mathbb T_\ell)},\\
    \|g(q)\|_{H^1(\mathbb T_\ell)}&\lesssim \|q\|_{H^{-1}(\mathbb T_\ell)}.
\end{align}
The implicit constants do not depend on $\kappa$.
\end{proposition}

Comparing the series expansions of $\alpha(\kappa;q)$ and $g(\kappa;q)$ one can observe that
\begin{align*}
    \frac{\delta \alpha}{\delta q}= \frac {1}{2\kappa}- g(q).
\end{align*}
This allows us to write down the Hamiltonian evolutions arising from the Hamiltonians $H_\kappa$ and $H_\kappa^n$ as 
\begin{equation} \label{eq:Hk} \tag{$H_\kappa$}
    \frac{d}{dt} q = 4\kappa^2 q'+ 16\kappa^5 g'(q),
\end{equation}
\begin{equation} \tag{$H_\kappa^{n}$} \label{eq:Hkn}
    \frac{d}{dt} q = 4\kappa^2 q'+ 16\kappa^5 P^{L_n}_{m_n<\dots \leq M_n}g'(P^{L_n}_{m_n<\dots \leq M_n} q)
\end{equation}
respectively.

\subsection{Equicontinuity}

\begin{definition}
A subset $Q$ of $H^{-1}(\mathbb R)$ is \textit{equicontinuous} if $q(x+h)\to q(x)$ in $H^{-1}$ as $h\to 0$, uniformly for all $q\in Q$.
\end{definition}

\begin{lemma} \label{Q1}
A bounded subset $Q\subset H^{-1}(\mathbb R)$ is equicontinuous in $ H^{-1}(\mathbb R)$ if and only if 
\begin{align}
    \label{Q eqcts}
    \int_{|\xi|\geq \lambda} \frac{1}{1+\xi^2}|\hat{q}(\xi)|^2 d\xi \to 0 \qquad\text{as}\,\,\lambda\to\infty,\,\,\text{uniformly for}\,\,q\in Q.
\end{align}
\end{lemma}

\begin{corollary} \label{Qc}
1. A bounded subset $Q\subset \dot H^{-\frac 1 2}(\mathbb R)$ is equicontinuous in $ H^{-1}(\mathbb R)$.\\
2. Let $q_n\in H^{-1}(\mathbb R)$, $n\in \mathbb N$, $q\in H^{-1}(\mathbb R)$ such that $q_n\to q$ in $H^{-1}(\mathbb R)$ as $n\to \infty$. Then $Q=\{q_n: n\in \mathbb N\}$ is equicontinuous in $H^{-1}(\mathbb R)$.\\
3. If $Q_1$, $Q_2$ are equicontinuous in $H^{-1}(\mathbb R)$ and $Q\subset Q_1+Q_2$, then $Q$ is also equicontinuous in $H^{-1}(\mathbb R)$.\\
4. A subset $Q\subset \{q\in H^{-1}(\mathbb R): \|q-z\|_{\dot H^{-\frac 1 2}(\mathbb R)}\leq A\}$ for some $z\in H^{-1}(\mathbb R)$, $A>0$ is equicontinuous in $ H^{-1}(\mathbb R)$.
\end{corollary}

\begin{proof}
1. Suppose $Q\subset \{q\in H^{-1}(\mathbb R): \|q\|_{\dot H^{-\frac 1 2}(\mathbb R)}\leq A\}$.
An easy computation shows that
\begin{align*}
    \int_{|\xi|\geq \lambda} \frac{1}{1+\xi^2}|\hat{q}(\xi)|^2 d\xi\leq \frac{1}{\lambda} \|q\|_{\dot H^{-\frac 12}(\mathbb R)}^2\leq \frac{1}{\lambda} A^2,
\end{align*}
hence \eqref{Q eqcts} holds.

2. We have that $q_n\to q$ in $H^{-1}(\mathbb R)$ as $n\to\infty$. Let $\varepsilon>0$. We can find $\lambda_n>0$, $n\in\mathbb N$, and $\lambda_0>0$ so that
\begin{align*}
    \int_{|\xi|\geq \lambda} \frac{1}{1+\xi^2}|\hat{q}_n(\xi)|^2 d\xi &<\varepsilon\quad\text{for all}\,\, \lambda\geq \lambda_n,\\
    \int_{|\xi|\geq \lambda} \frac{1}{1+\xi^2}|\hat{q}(\xi)|^2 d\xi &<\varepsilon\quad\text{for all}\,\, \lambda\geq \lambda_0.
\end{align*}
Moreover, there exists $n_0\in \mathbb N$ such that $\|q_n-q\|_{H^{-1}(\mathbb R)}^2<\varepsilon$ for all $n\geq n_0$. Hence for all $n\geq n_0$ we get that
\begin{align*}
    \int_{|\xi|\geq \lambda} \frac{1}{1+\xi^2}|\hat{q}_n(\xi)|^2 d\xi &\lesssim \int_{|\xi|\geq \lambda} \frac{1}{1+\xi^2}|\hat{q}(\xi)|^2 d\xi + \int_{|\xi|\geq \lambda} \frac{1}{1+\xi^2}|\widehat{q_n-q}(\xi)|^2 d\xi \\
    &\lesssim \int_{|\xi|\geq \lambda} \frac{1}{1+\xi^2}|\hat{q}(\xi)|^2 d\xi +\varepsilon\\
    &\lesssim \varepsilon
\end{align*}
for $\lambda>\lambda_0$. We conclude that for $\lambda>\max\{\lambda_0, \lambda_1, \dots, \lambda_{n_0}\}$ 
\begin{align*}
    \int_{|\xi|\geq \lambda} \frac{1}{1+\xi^2}|\hat{q}_n(\xi)|^2 d\xi \lesssim \varepsilon.
\end{align*}

3. Let $\varepsilon>0$. There exist $\lambda_1, \lambda_2>0$ so that
\begin{align*}
    \int_{|\xi|\geq \lambda} \frac{1}{1+\xi^2}|\hat{q}(\xi)|^2 d\xi &<\varepsilon\quad\text{for all}\,\,q\in Q_1,\,\, \lambda\geq \lambda_1,\\
    \int_{|\xi|\geq \lambda} \frac{1}{1+\xi^2}|\hat{q}(\xi)|^2 d\xi &<\varepsilon\quad\text{for all}\,\,q\in Q_2,\,\, \lambda\geq \lambda_2.
\end{align*}
Then for every $q= q_1+ q_2$ for some $q_1\in Q_1$, $q_2\in Q_2$ we have that
\begin{align*}
    \int_{|\xi|\geq \lambda} \frac{1}{1+\xi^2}|\hat{q}(\xi)|^2 d\xi &\lesssim \int_{|\xi|\geq \lambda} \frac{1}{1+\xi^2}|\hat{q}_1(\xi)|^2 d\xi + \int_{|\xi|\geq \lambda} \frac{1}{1+\xi^2}|\hat{q}_2(\xi)|^2 d\xi\lesssim \varepsilon
\end{align*}
for all $\lambda>\max\{\lambda_1, \lambda_2\}$.

4. We observe that $Q\subset \{f\in H^{-1}(\mathbb R): \|f\|_{\dot H^{-\frac 1 2}(\mathbb R)}\leq A\}+z$. Parts (2) and (3) assert that $Q$ is equicontinuous in $H^{-1}(\mathbb R)$.
\end{proof}

\section{Well-posedness of the truncated systems}\label{sec;wp}

As we laid out earlier, our strategy for proving Theorem \ref{sns KdV R} is to show that certain truncated systems have this property and use limiting arguments to transfer it to \eqref{KdV}. The first step is truncation in the Hamiltonian, which gives us the family of equations \eqref{eq:Hk}. Next, we perform a truncation in frequency and space, yielding equations \eqref{eq:Hkn}.

Naturally, before asking whether these systems obey symplectic non-squeezing, we need to consider whether they are even well-posed in the spaces we are interested in. Note that in both cases the nonlinearity only makes sense for small in $H^{-1}$ solutions. As it turns out, this is the major enemy as far as global well-posedness is concerned. In the case of \eqref{eq:Hk} we can circumvent this difficulty, but for \eqref{eq:Hkn} we compromise with existence of solutions up to times $T_n$ with $T_n\to\infty$ as $n\to\infty$ instead of global solutions.

\begin{theorem} [Global well-posedness of \ref{eq:Hk}]
\label{Hk wp}
Let $L>0$ and $\kappa \geq 1$ be fixed.\\
The equation \ref{eq:Hk} is globally well-posed for initial data in $B_{\delta_0} \subset H^{-1}(\mathbb R)$ and $ B_{\delta_0, \kappa}^L \subset H^{- 1}(\mathbb T_L)$ for $\delta_0>0$ small enough, independent of $\kappa$ and $L$. For each such initial data $u_0\in H^{-1}$ the solution $u$ obeys  
\begin{align} \label{H-1 conservation}
  \|u(t)\|_{H^{-1}}\lesssim \|u_0\|_{H^{- 1}}\quad\text{for all}\quad t\geq 0.  
\end{align}
Moreover, if in addition $u_0\in \dot H^{-\frac 1 2}$ then $u(t)\in \dot H^{-\frac 12}$ and 
\begin{align}  \label{H-1/2 conservation}
  \|u(t)\|_{\dot H^{-\frac 1 2}}\lesssim \|u_0\|_{\dot H^{-\frac 1 2}}e^{ct}\quad\text{for all}\quad t\geq 0.  
\end{align}
\end{theorem}

\begin{proof}
For both statements, the proof is the same for the line and for the circle.

Local well-posedness of \eqref{eq:Hk} in $H^{-1}$ and $\dot H^{-\frac 1 2}$ follows easily by rewriting the equation in integral form and observing that the nonlinearity is Lipschitz. By the diffeomorphism property of $g$,
\begin{align}
\label{dif1}
    \|g'(u)-g'(v)\|_{H^{-1}}\lesssim \|g(u)-g(v)\|_{H^1}\lesssim \|u-v\|_{H^{-1}}
\end{align}
and
\begin{align}
\label{dif2}
    \|g'(u)-g'(v)\|_{\dot H^{-\frac 1 2}}\leq \|g(u)-g(v)\|_{H^1}\lesssim \|u-v\|_{H^{-1}}\lesssim\|u-v\|_{\dot H^{-\frac 1 2}},
\end{align}
so a Picard iteration argument establishes the local well-posedness in both spaces. Then global well-posedness in $H^{-1}$ follows by \eqref{H-1 conservation} which is  consequence of the conservation of $\alpha(\varkappa)$ for all $\varkappa\geq 1$. For more details, see Proposition 3.2 in \cite{KV18}. 

Next, using the Duhamel formula we get that
\begin{align*}
    \|u(t)\|_{\dot H^{-\frac 1 2}}&\lesssim \|u_0\|_{\dot H^{-\frac 1 2}}+\int_0^t \|g'(u(s))\|_{\dot H^{-\frac 1 2}} ds\\
    &\lesssim \|u_0\|_{\dot H^{-\frac 1 2}}+\int_0^t \|g(u(s))\|_{H^1} ds\\
    &\lesssim \|u_0\|_{\dot H^{-\frac 1 2}}+\int_0^t \|u(s)\|_{H^{-1}} ds\\
    &\lesssim \|u_0\|_{\dot H^{-\frac 1 2}}+\int_0^t \|u(s)\|_{\dot H^{-\frac 1 2}} ds
\end{align*}
for all $t\geq 0$. An application of Gr\"onwall's inequality yields \eqref{H-1/2 conservation}.

As far as uniqueness is concerned, we will only show this in $H^{-1}$. For any two solutions $u, v \in H^{-1}$ to \eqref{eq:Hk}, by Duhamel's formula and estimate \eqref{dif1} we get
\begin{align*}
    \|u(t)-v(t)\|_{H^{-1}} \lesssim  \|u(0)-v(0)\|_{H^{-1}}+ \int_0^t  \| u(s)- v(s)\|_{H^{-1}} ds
\end{align*}
so 
$$\|u(t)-v(t)\|_{H^{-1}}\lesssim \|u(0)-v(0)\|_{H^{-1}} e^{ct}\quad\text{or all}\,\,t\geq 0.$$
In particular, we conclude that for each small enough initial data $u_0\in H^{-1}$ there exists unique global solution $u\in H^{-1}$ to \eqref{eq:Hk}.
\end{proof}
 
\begin{remark}
It is evident that in the case of the homogeneous Sobolev space our proof heavily relies on the presence of the derivative on the nonlinearity, as indicated in \eqref{dif2}. The same argument can be used to prove a global well-posedness result and the corresponding estimates for \ref{eq:Hk} in $H^{-s}$ and $\dot H^{-s}$ on the line and on the circle for any $0\leq s\leq 1$. 
\end{remark}

Now we turn our attention to the well-posedness of \eqref{eq:Hkn}. As far as the frequency truncations are concerned, at this moment our sequences of frequency cut-offs $m_n, M_n\in 2^{\mathbb Z}$ only need to satisfy
\begin{align*}
    \lim_{n\to\infty}m_n=0, \qquad \lim_{n\to\infty}M_n=\infty.
\end{align*}
We also consider a sequence $L_n>0$, indicating the lengths of the tori we are working on. We denote these tori by $\mathbb T_n:= \mathbb R/(L_n \mathbb Z)$. We do not impose any hypotheses on $L_n$ for now. It makes sense to focus on initial data in the space
\begin{align*}
    \mathcal{H}_n:= \{f\in H^{-1}(\mathbb T_n): P^{L_n}_{\leq \frac{m_n}{2}}f= P^{L_n}_{>2M_n}f=0\}.
\end{align*}
Solutions to \eqref{eq:Hkn} with initial data in $\mathcal H_n$ stay in $\mathcal{H}_n$.

\begin{theorem} [Well-posedness for \ref{eq:Hkn}]
\label{Hkn wp}
Fix $\kappa \geq 1$. There exists $\delta_0>0$ small enough (independent of $n$ and $\kappa$) and a sequence $T_n>0$ satisfying $\lim_{n\to\infty} T_n=\infty$ so that the following are true for every $n\in\mathbb N$: For every $u_{n,0}\in \mathcal H_n \cap \{f\in \dot H^{-\frac 12}(\mathbb T_n): \kappa^{-\frac 12} \|f\|_{\dot H^{-\frac 12}(\mathbb T_n)}\leq \delta_0\}\subset B_{\delta_0, \kappa}^{L_n}$ there exists a unique solution $u_n\in C([0,T_n])H^{-1}(\mathbb T_n)$ to the the equation \eqref{eq:Hkn}. For each such initial data $u_{n,0}$ the solution $u_n(t)\in \mathcal H_n$ obeys 
\begin{align}\label{H-1 trunc}
  \|u_n(t)\|_{H^{-1}(\mathbb T_{n})}\lesssim \|u_{n,0}\|_{\dot H^{-\frac 12}(\mathbb T_{n})}
\end{align}
and  
\begin{align} \label{H-1/2 trunc}
  \|u_n(t)\|_{\dot H^{-\frac 1 2}(\mathbb T_{n})}\lesssim \|u_{n,0}\|_{\dot H^{-\frac 1 2}(\mathbb T_{n})}e^{ct}
\end{align}
for all $t\in[0,T_n]$. The implicit constants here do not depend on $n$.

\end{theorem}

\begin{proof}
An argument similar to the one in the proof of Theorem \ref{Hk wp} readily shows that \eqref{eq:Hkn} is locally well-posed in $H^{-1}(\mathbb T_n)$. In order to upgrade the existence of solution to larger times it suffices to prove the estimate \eqref{H-1 trunc}. This will ensure both that the nonlinearity makes sense and that we can extend our solution up to time $T_n$.

Let $v_n$ be the unique global solution to \eqref{eq:Hk} with small enough initial data $v_n(0)=u_n(0)=u_{n,0} \in \mathcal{H}_n$ which is guaranteed to exist by Theorem \ref{Hk wp}. By Duhamel,
\begin{align*}
    \|u_n(t)-& v_n(t)\|_{\dot H^{-\frac 12}(\mathbb T_n)}\\
    &\lesssim \int_0^t \left\|P^{L_n}_{m_n<\dots\leq M_n}g'(P^{L_n}_{m_n<...\leq M_n} u_n(s))- g'(v_n(s))\right\|_{\dot H^{-\frac 12}(\mathbb T_n)} ds
\end{align*}
and for each $s$ triangle inequality and the diffeomorphism property of $g$ allow us to estimate
\begin{align*}
    \|P^{L_n}_{m_n<\dots\leq M_n} & g'(P^{L_n}_{m_n<...\leq M_n} u_n(s))- g'(v_n(s))\|_{\dot H^{-\frac 12}(\mathbb T_n)}\\
    &\lesssim   \|P^{L_n}_{\leq m_n} v_n(s)\|_{H^{-1}(\mathbb T_n)}+ \|P^{L_n}_{> M_n} v_n(s)\|_{H^{-1}(\mathbb T_n)}\\
    &  + \|P^{L_n}_{\leq m_n}g'(v_n(s))\|_{\dot H^{-\frac 12}(\mathbb T_n)}+ \|P^{L_n}_{> M_n}g'(v_n(s))\|_{\dot H^{-\frac 12}(\mathbb T_n)}\\
    &+ \|u_n(s)- v_n(s)\|_{\dot H^{-\frac 12}(\mathbb T_n)}.
\end{align*}
Bernstein inequalities, the diffeomorphism property, \eqref{H-1/2 conservation}, and finally Gr\"onwall's inequality yield
\begin{align*}
    \|u_n(t)- v_n(t)\|_{\dot H^{-\frac 12}(\mathbb T_n)}\lesssim (m_n^\frac 12 +M_n^{-\frac 12})e^{ct}\|u_{n,0}\|_{\dot H^{-\frac 12}(\mathbb T_n)}.
\end{align*}
Choosing times $T_n>0$ so that $(m_n^\frac 12 +M_n^{-\frac 12})e^{cT_n}\ll 1 $ and $T_n\to\infty$ as $n\to \infty$, and using \eqref{H-1 conservation}, we obtain \eqref{H-1 trunc} and \eqref{H-1/2 trunc}. Starting with $\delta_0>0$ sufficiently small will ensure that the nonlinearity will make sense up to time $T_n$.

Uniqueness can be shown as in the proof of Theorem \ref{Hk wp}; the presence of the frequency truncation does not affect the argument.
\end{proof}

From now on, we will refer by $\delta_0$ to a constant smaller than the minimum of the two small positive constants obtained in Theorems \ref{Hk wp} and \ref{Hkn wp}. Consequently, $\delta_0>0$ will be a small enough parameter to ensure well-posedness up to time $T_n$ for \eqref{eq:Hkn}, global well-posedness for \eqref{eq:Hk} and that the $H^{-1}$ norm of the solutions given by both Theorems \ref{Hk wp} and \ref{Hkn wp} remains bounded by $1$ and $\delta$ (as in Propositions \ref{diffeo R} and \ref{diffeo T}) in the time interval of existence.

The next question we need to address is whether the frequency truncated models \eqref{eq:Hkn} provide a good approximation for the \eqref{eq:Hk} flow on the torus $\mathbb T_n$.

\begin{theorem}
\label{Hkn to Hk}
Fix $\kappa\geq 1$. Let $u_n\in C([0,T_n])H^{-1}(\mathbb T_n)$ be the solution to \ref{eq:Hkn} with initial data $u_n(0)=u_{n,0}\in \mathcal H_n$, $\|u_{n,0}\|_{\dot H^{-\frac 12}(\mathbb T_n)}\leq \kappa^{\frac 12} \delta_0$, and $v_n\in C([0,T_n])H^{-1}(\mathbb T_n)$ be a solution to
\begin{equation} \label{eq:Hke}
    \frac{d}{dt} v_n = 4\kappa^2 v_n'+16\kappa^5 g'(v_n)+e_n
\end{equation}
with initial data $v_n(0)=v_{n,0}\in H^{-1}(\mathbb T_n)$.
If $\|e_n(t)\|_{H^{-1}(\T_n)}<\varepsilon$ for all $t\in [0,T_n]$, then
$$\|u_n(t)-v_n(t)\|_{H^{-1}(\T_n)} \lesssim \big( \|u_{n,0}-v_{n,0}\|_{H^{-1}(\T_n)}+ (m_n^{\frac 1 2}+M_n^{-\frac 1 2})\|u_{n,0}\|_{\dot H^{-\frac 1 2}(\T_n)}+\varepsilon  \big) e^{ct}$$
for all $t\in [0,T_n]$.
\end{theorem}

\begin{proof}
By Duhamel's formula, for all $t\in[0,T_n]$
\begin{align*}
    \|u_n(t)-v_n(t) & \|_{H^{-1}(\mathbb T_n)}\lesssim \|u_{n,0}-v_{n,0}\|_{H^{-1}(\mathbb T_n)}+\int_0^t \|e_n(s)\|_{H^{-1}(\mathbb T_n)} ds \\
    &+\int_0^t \|P^{L_n}_{m_n<\dots\leq M_n}g'(P^{L_n}_{m_n<\dots\leq M_n}u_n(s))-g'(v_n(s))\|_{H^{-1}(\mathbb T_n)} ds  .
\end{align*}
We only need to work on the last term. For the sake of conciseness, here we denote $P^{L_n}_{m_n<\dots\leq M_n}$ by $\mathcal{P}$.

Working similarly as in the proof of Theorem \ref{Hkn wp},
\begin{align*}
    \|\mathcal P g'(\mathcal P u_n(s))- & g'(v_n(s))\|_{H^{-1}(\T_n)} \lesssim   \|u_n(s)-v_n(s)\|_{H^{-1}(\T_n)}\\
    &+\|P^{L_n}_{\leq m_n} g'(\mathcal P u_n(s))\|_{H^{-1}(\T_n)} +  \|P^{L_n}_{> M_n} g'(\mathcal P u_n(s))\|_{H^{-1}(\T_n)} \\
    &+ \|P^{L_n}_{\leq m_n} u_n(s)\|_{H^{-1}(\T_n)} + \|P^{L_n}_{>M_n} u_n(s)\|_{H^{-1}(\T_n)}.
\end{align*}
By Bernstein inequalities, the diffeomorphism property, and \eqref{H-1/2 trunc}, we get that
\begin{align*}
    \|u_n(t)-v_n(t)\|_{H^{-1}(\mathbb T_n)}\lesssim & \|u_{n,0}-v_{n,0}\|_{H^{-1}(\mathbb T_n)} \\
    &+ \left(\varepsilon+ \left(m_n^{\frac 1 2}+M_n^{-\frac 1 2}\right)\|u_{n,0}\|_{\dot H^{-\frac 1 2}(\T_n)}\right)e^{ct}\\
    &+ \int_0^t \|u_n(s)-v_n(s)\|_{H^{-1}(\mathbb T_n)}ds,
\end{align*}
so an application of Gr\"onwall's inequality finishes the proof.
\end{proof}

\section{From the Line to the Circle}\label{sec;periodization}
Having settled on the finite-dimensional systems that we will use to approximate the $H_\kappa$ flow, the next step is to determine the suitable parameters (that is, the appropriate ball and cylinder) for which we should apply Theorem \ref{gromov}. In this section we develop the tools that will enable us to pass our fixed parameters from the line setting to the circle and ultimately allow us to take advantage of Gromov's result.

\begin{definition}
\label{def per}
Let $f$ be a compactly supported function on $\R$. For $L>0$ large so that the support of $f$ is contained in an interval of length $L$, we define the \textit{$L$-periodization} of $f$
\begin{align*}
    \mathring f_L(x):= \sum_{j\in\mathbb{Z}} f(x+jL).
\end{align*}
\end{definition}

This is an $L$-periodic function that agrees with $f$ on its support. One readily sees that $\mathring f_L$ inherits the smoothness of $f$. However, it is non-trivial to show that it belongs to the corresponding fractional homogeneous Sobolev spaces.  

\begin{lemma}
\label{periodization}
Let $f\in C_c^\infty(\R)$, $L_0>0$ so that the support of $f$ is contained in an interval of length $L_0$.\\
1. For all $L>L_0$ and for all integers $k\geq 0$, $\mathring f_L\in \dot H^k(\T_L)$  and
\begin{align*}
    \|\mathring f_L\|_{\dot H^{k} (\mathbb T_L)}=\|f\|_{\dot H^{k} (\mathbb R)}.
\end{align*}
2. For all $L>L_0$ and for all integers $k< 0$, if $f\in\dot H^k(\R)$ then $\mathring f_L\in \dot H^k(\T_L)$  and
\begin{align*}
    \|\mathring f_L\|_{\dot H^{k} (\mathbb T_L)}=\|f\|_{\dot H^{k} (\mathbb R)}.
\end{align*}
3. If $f\in \dot H^{s}(\R)$ for some $s\in \R$, then $\mathring f_L \in \dot H^{s} (\mathbb T_L)$ for $L$ sufficiently large and
\begin{align*}
    \lim_{L\to\infty} \|\mathring f_L\|_{\dot H^{s} (\mathbb T_L)}=\|f\|_{\dot H^{s} (\mathbb R)}.
\end{align*}
\end{lemma}

\begin{proof}
The first statement is straightforward for $k=0$, due to the locality of the $L^2$-norm. For $k\in\mathbb N$, 
$$(\mathring f_L)^{(k)}(x)= \sum_{j\in \mathbb Z} f^{(k)}(x+jL)=\mathring {(f^{(k)})}_L(x),$$
hence $\|(\mathring f_L)^{(k)}\|_{L^2(\T_L)}=\|\mathring {(f^{(k)})}_L\|_{L^2(\T_L)}=\|f^{(k)}\|_{L^2(\R)}$. 

For integers $k<0$ we proceed inductively. First of all, we already argued about $k=0$. Suppose that the statement is true for some $k\leq 0$. Then for $f\in \dot H^{k-1}(\R)\cap C_c^\infty(\R)$ supported on $[a, a+L_0]$ we consider
\begin{align*}
    F(x):=\int_a^x f(y) dy.
\end{align*}
By the definition of the negative order homogeneous Sobolev spaces, $\int f=0$, therefore $F$ is also supported on $[a, a+L_0]$. We can also see that $F\in \dot H^{k}(\R)\cap C_c^\infty(\R)$, $F'=f$. The inductive hypothesis ensures that $\mathring F_L\in \dot H^k(\T_L)$ so arguing as before
$$\mathring f_L= \mathring{(F')}_L=(\mathring F_L)'\in \dot H^{k-1}(\T_L)$$
and
\begin{align*}
    \|\mathring f_L\|_{\dot H^{k-1} (\mathbb T_L)}=\|(\mathring F_L)'\|_{\dot H^{k-1} (\mathbb T_L)}=4\pi^2\|\mathring F_L\|_{\dot H^{k} (\mathbb T_L)}=4\pi^2 \|F\|_{\dot H^{k} (\mathbb R)}=\|f\|_{\dot H^{k-1} (\mathbb R)}.
\end{align*}

For the last statement, we observe that
\begin{align*}
    \|\mathring f_L\|_{\dot H^{s} (\mathbb R)}^2&= L \sum_{k\in L^{-1}\mathbb{Z}} |k|^{2s} \left| \widehat{\mathring f_L} (k)\right|^2 \\
    &= L \sum_{k\in L^{-1}\mathbb{Z}} |k|^{2s} \left| \frac {1}{L}\hat{f}(k)\right|^2\\
    &=  \frac{1}{L} \sum_{k\in L^{-1}\mathbb{Z}} |k|^{2s} \left|\hat{f} (k)\right|^2
\end{align*}
which converges to $\int |\xi|^{2s} |\hat{f} (\xi)|^2 d\xi=\|f\|_{\dot H^{s} (\mathbb R)}^2$ as $L\to\infty$, given that $f \in \dot H^s(\mathbb R) \cap C_c^\infty(\R)$.
\end{proof}

\begin{lemma}
\label{norm3}
Let $f\in C_c^\infty(\R)$ supported on an interval of length $L_0$. For given sequences $m_n, M_n\in 2^{\mathbb Z}$ and $L_n>L_0$ such that $\lim_{n\to\infty}m_n=0$, $\lim_{n\to\infty}M_n=\infty$, and $\lim_{n\to\infty}L_n=\infty$, consider 
\begin{align*}
    z_n= \mathring f_{L_n}, \qquad \zeta_n=P^{L_n}_{m_n<\dots\leq M_n} z_n, \qquad \lambda_n= \frac{1}{\|\zeta_n\|_{\dot H^{\frac 12}(\mathbb T_n)}}\zeta_n.
\end{align*}
\begin{enumerate}
\item If $f\in \dot H^{-\frac 1 2}(\mathbb R)$, then for $n$ sufficiently large $z_n, \zeta_n, \lambda_n \in \dot H^{-\frac 1 2} (\mathbb T_n)$ and 
\begin{align*}
    \lim_{n\to \infty}\|\zeta_n-z_n\|_{L^2(\mathbb T_n)}=0.
\end{align*}
\item If $\|f\|_{\dot H^{\frac 1 2}(\R)}=1$, then for $n$ sufficiently large $z_n,\zeta_n, \lambda_n \in \dot H^{\frac 1 2} (\mathbb T_n)\cap \dot H^{1} (\mathbb T_n)$ and
\begin{align*}
    \lim_{n\to \infty}\|\lambda_n-z_n\|_{\dot H^{\frac 1 2}(\mathbb T_n)}=0.
\end{align*}
\end{enumerate}
\end{lemma}

\begin{proof}

By Lemma \ref{periodization}, $z_n, \zeta_n, \lambda_n$ are in the appropriate homogeneous Sobolev spaces for large $n$. Moreover, if $f\in \dot H^{-\frac 12}(\mathbb R)$, for $n$ sufficiently large
\begin{align*}
    \|\zeta_n-z_n\|_{L^2(\mathbb T_n)} &\leq  \left\|P^{L_n}_{\leq m_n} z_n\right\|_{L^2(\mathbb T_n)}+\left\|P^{L_n}_{> M_n} z_n\right\|_{L^2(\mathbb T_n)} \\
    &\lesssim m_n^{\frac 1 2} \|z_n\|_{\dot H^{-\frac 1 2}(\mathbb T_n)}+ M_n^{-\frac 1 2} \|z_n\|_{\dot H^{\frac 1 2}(\mathbb T_n)}\\
    &\lesssim m_n^{\frac 1 2} \|f\|_{\dot H^{-\frac 1 2}(\mathbb R)}+ M_n^{-\frac 1 2} \|f\|_{\dot H^{\frac 1 2}(\mathbb R)}.
\end{align*}
A similar argument gives
\begin{align*}
    \left\|\zeta_n -z_n\right\|_{\dot H^{\frac 1 2} (\mathbb T_n)} & \lesssim m_n^{\frac 1 2} \|f\|_{L^2 (\mathbb R)}+ M_n^{-\frac 1 2}\|f\|_{\dot H^{1} (\mathbb R)},\\
    \left\|\zeta_n -z_n\right\|_{\dot H^{1} (\mathbb T_n)} & \lesssim m_n \|f\|_{L^2 (\mathbb R)}+ M_n^{-1}\|f\|_{\dot H^{2} (\mathbb R)}
\end{align*}
for $n$ large, without the additional assumptions of (1) or (2) on $f$. Therefore 
$$\lim_{n\to\infty}\left\|\zeta_n -z_n\right\|_{L^2 (\mathbb T_n)}=\lim_{n\to\infty}\left\|\zeta_n -z_n\right\|_{\dot H^{\frac 1 2} (\mathbb T_n)}=\lim_{n\to\infty}\left\|\zeta_n -z_n\right\|_{\dot H^{1} (\mathbb T_n)}=0.$$
In particular, if we also have that $\|f\|_{\dot H^{\frac 1 2}(\R)}=1$,
\begin{align*}
    \lim_{n\to\infty} \left\|\zeta_n\right\|_{\dot H^{\frac 1 2} (\mathbb T_n)}= \lim_{n\to\infty} \|z_n\|_{\dot H^{\frac 1 2} (\mathbb T_n)}=\|f\|_{\dot H^{\frac 1 2} (\mathbb R)}=1.
\end{align*}
Under this assumption we also get that
\begin{align*}
    \|\lambda_n-z_n\|_{\dot H^{\frac 1 2}(\mathbb T_n)}&\leq \left|1-\left\|\zeta_n\right\|_{\dot H^{\frac 1 2} (\mathbb T_n)}^{-1}\right| \left\|\zeta_n\right\|_{\dot H^{\frac 1 2} (\mathbb T_n)}+ \left\|\zeta_n -z_n\right\|_{\dot H^{\frac 1 2} (\mathbb T_n)}
\end{align*}
converges to $0$ as $n\to\infty$.
\end{proof}

\section{Finite Dimensional Approximation: From the Circle to the Line} \label{sec;fda}

The previous section dealt with the problem of determining the appropriate parameters for which we wish to apply Gromov's theorem to the finite dimensional Hamiltonian systems \eqref{eq:Hkn}. Now we need to prescribe a way to extract a `limit' of the witnesses we obtained. A first step towards this direction is, given a sequence of periodic initial data (whose periods go to infinity) that satisfy the same $\dot H^{-\frac 12}$ bound, to construct a sequence of compactly supported initial data (with ever larger supports) whose \ref{eq:Hk} evolution approximates the one of the original sequence under the \ref{eq:Hkn} flow. This is achieved via the `cutting' and `unwrapping' process we advertised earlier.

Let $\kappa\leq 1$ be fixed.
We are given $T>0$, $0<A<\frac {\delta_0}{4}$ and consider $m_n, M_n, N_n, L_n$ be sequences such that $m_n\to 0$ and $M_n, N_n, L_n \to \infty$ as $n\to \infty$,
\begin{align*}
    m_n^{-2}M_n \ll N_n, \qquad N_n^5\ll L_n.
\end{align*}
We are working on the torus $\mathbb T_n=\mathbb R/L_n \mathbb Z$ with a sequence $u_{n,0}\in \dot H^{-\frac 1 2}(\mathbb T_n)$ such that $u_{n,0}\in \mathcal H_n$ and $\|u_{n,0}\|_{\dot H^{-\frac 1 2}(\mathbb T_n)}\leq A$.

\subsection{`Cutting' on the circle}

For every $n$ large enough, essentially we want to divide the interval $[-\frac{L_n}{2}, \frac{L_n}{2}]$ into $N_n$ subintervals of length $\frac {L_n}{N_n}$ and look for one such subinterval where the $\dot H^{-\frac 1 2}$ norm of $u_{n,0}$ is comparatively small. To avoid some of the problems caused by the non-local nature of the $\dot H^{-\frac 1 2}$ norm, we work with a partition of unity instead.

Let $\phi_n$ be a radial smooth cutoff such that $\phi_n=1$ on $[- \frac {L_n}{4N_n}, \frac {L_n}{4N_n}]$, $\phi_n=0$ outside $[-\frac{3L_n}{4N_n},\frac{3L_n}{4N_n}]$, and $0\leq \phi_n(x)\leq 1$ for $\frac {L_n}{4N_n}\leq x \leq \frac {3L_n}{4N_n}$ with $\phi_n(x)+\phi_n(x-\frac {L_n}{N_n})=1$. We consider the smooth cutoffs
\begin{align*}
    \phi^k_n(x):= \phi_n\left(x+\frac{L_n}{2}-\frac{3L_n}{4N_n}-k \frac {L_n}{N_n}\right), \quad 0\leq k\leq N_n-1
\end{align*}
and their $L_n$-periodizations 
\begin{align*}
    \mathring \phi^k_n(x):= \sum_{j\in\mathbb Z}\phi^k_n(x+j L_n), \quad 0\leq k\leq N_n-1.
\end{align*}
Note that by the definition of $\phi_n$
\begin{align}\label{partition}
    \sum_{k=0}^{N_n-1}\mathring \phi^k_n =1.
\end{align}


\begin{lemma}
\label{bump}
There exists constant $C>0$ (independent of $n$) such that the following are true for every $n$:
\begin{enumerate}
    \item At least $\frac {9} {10} N_n$ elements $k$ in $\{0,\dots, N_n-1\}$ satisfy
\begin{align*}
    \|\mathring\phi^k_n u_{n,0}\|_{\dot H^{-\frac 1 2}(\T_n)} <C \|u_{n,0}\|_{\dot H^{-\frac 1 2}(\T_n)} M_n^{\frac 12} m_n^{-\frac 1 2} N_n^{-\frac 12}.
\end{align*}
\item At least $\frac {9} {10} N_n$ elements $k$ in $\{0,\dots, N_n-1\}$ satisfy
\begin{align*}
    \|\mathring\phi^k_n u_{n,0}\|_{\dot H^{-1}(\T_n)} <C \|u_{n,0}\|_{\dot H^{-\frac 1 2}(\T_n)} M_n^{\frac 12} m_n^{-1} N_n^{-\frac 12}.
\end{align*}
\end{enumerate}
\end{lemma}

\begin{proof}
Firstly, we note that by the definition of $\mathring\phi^k_n$
\begin{align} \label{L2}
    \sum_{k=0}^{N_n-1} \left\|\mathring\phi^k_n u_{n,0}\right\|_{L^2(\T_n)}^2= \int_{\T_n} \sum_{k=0}^{N_n-1} \left(\mathring\phi^k_n(x)\right)^2 \left|u_{n,0}(x)\right|^2 dx \leq\|u_{n,0}\|_{L^2(\T_n)}^2.
\end{align}

Splitting $\mathring\phi^k_n u_{n,0}$ in frequencies we get
\begin{align*}
    \mathring\phi^k_n u_{n,0}= P^{L_n}_{>\frac{m_n}{8}}\left(\mathring\phi^k_n u_{n,0}\right)&+ P^{L_n}_{\leq \frac{m_n}{8}}\left[ \left(P^{L_n}_{\leq \frac{m_n}{8}}\mathring\phi^k_n\right) u_{n,0}\right]+  P^{L_n}_{\leq \frac{m_n}{8}}\left[ \left(P^{L_n}_{>8 M_n}\mathring\phi^k_n\right) u_{n,0}\right] \\
    &+  P^{L_n}_{\leq \frac{m_n}{8}}\left[ \left(P^{L_n}_{\frac{m_n}{8}<\dots \leq 8M_n}\mathring\phi^k_n\right) u_{n,0}\right].
\end{align*}
Since $u_{n,0}$ is supported on $\{\xi:\frac{m_n}{2}\leq |\xi|\leq 4M_n\}$ and $P^{L_n}_{\leq \frac{m_n}{8}}\mathring\phi^k_n$, $P^{L_n}_{>8 M_n}\mathring\phi^k_n$ are supported on $\{\xi:|\xi|\leq \frac{m_n}{4}\}$ and $\{\xi: |\xi|> 8M_n\}$ respectively, we have that $ \left(P^{L_n}_{\leq \frac{m_n}{8}}\mathring\phi^k_n\right) u_{n,0}$ and $\left(P^{L_n}_{>8 M_n}\mathring\phi^k_n\right) u_{n,0}$ must be supported on $\{\xi:|\xi|\geq \frac{m_n}{4} \}$, hence 
\begin{align*}
     P^{L_n}_{\leq \frac{m_n}{8}}\left[ \left(P^{L_n}_{\leq \frac{m_n}{8}}\mathring\phi^k_n\right) u_{n,0}\right]=  P^{L_n}_{\leq \frac{m_n}{8}}\left[ \left(P^{L_n}_{>8 M_n}\mathring\phi^k_n\right) u_{n,0}\right]=0.
\end{align*}
Then
\begin{align}
\label{eq:sum}
    \sum_{k=0}^{N_n-1} \left\|\mathring\phi^k_n u_{n,0}\right\|_{\dot H^{-\frac 1 2}(\T_n)}^2\leq & 2\sum_{k=0}^{N_n-1} \left\|P^{L_n}_{>\frac{m_n}{8}}\left(\mathring\phi^k_n u_{n,0}\right)\right\|_{\dot H^{-\frac 1 2}(\T_n)}^2\notag\\
    &+ 2\sum_{k=0}^{N_n-1} \left\|P^{L_n}_{\leq \frac{m_n}{8}}\left[ \left(P^{L_n}_{\frac{m_n}{8}<\dots \leq 8M_n}\mathring\phi^k_n\right) u_{n,0}\right]\right\|_{\dot H^{-\frac 1 2}(\T_n)}^2
\end{align}
and
\begin{align}
\label{eq:sum2}
    \sum_{k=0}^{N_n-1} \left\|\mathring\phi^k_n u_{n,0}\right\|_{\dot H^{-1}(\T_n)}^2\leq & 2\sum_{k=0}^{N_n-1} \left\|P^{L_n}_{>\frac{m_n}{8}}\left(\mathring\phi^k_n u_{n,0}\right)\right\|_{\dot H^{-1}(\T_n)}^2\notag\\
    &+ 2\sum_{k=0}^{N_n-1} \left\|P^{L_n}_{\leq \frac{m_n}{8}}\left[ \left(P^{L_n}_{\frac{m_n}{8}<\dots \leq 8M_n}\mathring\phi^k_n\right) u_{n,0}\right]\right\|_{\dot H^{-1}(\T_n)}^2.
\end{align}

For the first term in the right-hand side of both inequalities we observe that by Bernstein and \eqref{L2}
\begin{align} \label{high -1/2}
    \sum_{k=0}^{N_n-1}  \left\|P^{L_n}_{>\frac{m_n}{8}}\left(\mathring\phi^k_n u_{n,0}\right)\right\|_{\dot H^{-\frac 1 2}(\T_n)}^2  &\lesssim m_n^{-1}\sum_{k=0}^{N_n-1} \left\|\mathring\phi^k_n u_{n,0}\right\|_{L^2(\T_n)}^2 \notag\\
    &\leq m_n^{-1}\left\|u_{n,0}\right\|_{L^2(\T_n)}^2 = m_n^{-1}\left\|P^{L_n}_{\leq 2M_n}u_{n,0}\right\|_{L^2(\T_n)}^2\notag\\
    &\lesssim m_n^{-1} M_n\left\|u_{n,0}\right\|_{\dot H^{-\frac 1 2}(\T_n)}^2
\end{align}
and similarly
\begin{align} \label{high -1}
    \sum_{k=0}^{N_n-1}  \left\|P^{L_n}_{>\frac{m_n}{8}}\left(\mathring\phi^k_n u_{n,0}\right)\right\|_{\dot H^{-1}(\T_n)}^2 \lesssim m_n^{-2} M_n\left\|u_{n,0}\right\|_{\dot H^{-\frac 1 2}(\T_n)}^2.
\end{align}

For the second term, taking advantage of the fact that $\mathring\phi_n^k$ is a translation of $\mathring \phi_n^0$ for every $k$ and by Cauchy-Schwarz inequality,
\begin{align*}
    &\sum_{k=0}^{N_n-1} \left\|P^{L_n}_{\leq \frac{m_n}{8}}\left[ \left(P^{L_n}_{\frac{m_n}{8}<\dots \leq 8M_n}\mathring\phi^k_n\right) u_{n,0}\right]\right\|_{\dot H^{-\frac 1 2}(\T_n)}^2\\
    &\leq \sum_{k=0}^{N_n-1} L_n \sum_{l\in \frac{1}{L_n}\mathbb Z, 0<|l|\leq \frac{m_n}{4}} |l|^{-1} \left| \sum_{m\in \frac{1}{L_n}\mathbb Z} \widehat{\left(P^{L_n}_{\frac{m_n}{8}<\dots \leq 8M_n}\mathring\phi^k_n\right)}(m) \hat u_{n,0}(l-m) \right|^2\\
    & \leq L_n N_n \sum_{l\in \mathcal A_n^1} |l|^{-1}  \sum_{m,j\in \mathcal A_n^2, j-m\in \frac{N_n}{L_n}\mathbb Z} |\hat {\mathring\phi}_n^0(m)||\hat {\mathring\phi}_n^0(j)| |\hat u_{n,0}(l-m)||\hat u_{n,0}(l-j)| \\
    &\leq \sum_{s\in\mathcal B_n} L_n N_n \sum_{l\in \mathcal A_n^1} |l|^{-1}  \sum_{m, m+s\in \mathcal A_n^2} |\hat {\mathring\phi}_n^0(m)||\hat {\mathring\phi}_n^0(m+s)| |\hat u_{n,0}(l-m)||\hat u_{n,0}(l-m-s)|\\
    &\leq \sum_{s\in\mathcal B_n} L_n N_n \sum_{l\in \mathcal A_n^1} |l|^{-1}  \sum_{m\in \mathcal A_n^2} |\hat {\mathring\phi}_n^0(m)|^2 |\hat u_{n,0}(l-m)|^2\\
    &\lesssim M_n L_n^2 \sum_{l\in \mathcal A_n^1} |l|^{-1}  \sum_{m\in \mathcal A_n^2} |\hat {\mathring\phi}_n^0(m)|^2 |\hat u_{n,0}(l-m)|^2\\
    &= M_n L_n^2 \sum_{(k,m)\in \mathcal C_n} |k+m|^{-1}  |\hat {\mathring\phi}_n^0(m)|^2 |\hat u_{n,0}(k)|^2
\end{align*}
and
\begin{align*}
    &\sum_{k=0}^{N_n-1} \left\|P^{L_n}_{\leq \frac{m_n}{8}}\left[ \left(P^{L_n}_{\frac{m_n}{8}<\dots \leq 8M_n}\mathring\phi^k_n\right) u_{n,0}\right]\right\|_{\dot H^{-1}(\T_n)}^2\\
    &\lesssim M_n L_n^2 \sum_{l\in \mathcal A_n^1} |l|^{-2}  \sum_{m\in \mathcal A_n^2} |\hat {\mathring\phi}_n^0(m)|^2 |\hat u_{n,0}(l-m)|^2\\
    &= M_n L_n^2 \sum_{(k,m)\in \mathcal C_n} |k+m|^{-2}  |\hat {\mathring\phi}_n^0(m)|^2 |\hat u_{n,0}(k)|^2
\end{align*}
where $\mathcal A_n^1:=\{l\in\frac{1}{L_n}\mathbb Z: 0<|l|\leq \frac{m_n}{4}\}$, $\mathcal A_n^2:=\{m\in \frac{1}{L_n}\mathbb Z: \frac{m_n}{8}\leq|m|\leq 16M_n\}$, $\mathcal B_n:=\{s\in\frac{N_n}{L_n}\mathbb Z:  |s|\leq 32 M_n\}$, $\mathcal C_n:=\{(k,m)\in \frac{1}{L_n}\mathbb Z^2:\frac{m_n}{2}\leq|k|\leq 4M_n, \frac{m_n}{8}\leq|m|\leq 16M_n, 0<|k+m|\leq \frac{m_n}{4} \}$.
Taking advantage of the bounds that all $(k,m)\in\mathcal C_n$ must satisfy as well as the fact that $u_{n,0}\in\dot H^{-\frac 1 2}(\T_n)$, $\mathring\phi_n^0\in\dot H^2(\T_n)$ and $\mathring\phi_n^0\in\dot H^3(\T_n)$ with $\|\mathring\phi_n^0\|_{\dot H^2(\T_n)}\lesssim \left(\frac{N_n}{L_n}\right)^{\frac 3 2}$, $\|\mathring\phi_n^0\|_{\dot H^3(\T_n)}\lesssim \left(\frac{N_n}{L_n}\right)^{\frac 5 2}$,
\begin{align} \label{low -1/2}
    \sum_{k=0}^{N_n-1} & \left\|P^{L_n}_{\leq \frac{m_n}{8}}\left[ \left(P^{L_n}_{\frac{m_n}{8}<\dots \leq 8M_n}\mathring\phi^k_n\right) u_{n,0}\right]\right\|_{\dot H^{-\frac 1 2}(\T_n)}^2\notag\\
    &\lesssim M_n L_n^2 \sum_{k\in \frac{1}{L_n}\mathbb Z}\sum_{m\in \frac{1}{L_n}\mathbb Z} L_n (M_n |k|^{-1})(m_n^{-4} |m|^4) |\hat {\mathring\phi}_n^0(m)|^2 |\hat u_{n,0}(k)|^2\notag\\
    &= M_n^2 m_n^{-4} L_n \left(L_n\sum_{k\in \frac{1}{L_n}\mathbb Z}|k|^{-1} |\hat u_{n,0}(k)|^2 \right) \left(L_n \sum_{m\in \frac{1}{L_n}\mathbb Z}  |m|^4 |\hat {\mathring\phi}_n^0(m)|^2\right) \notag\\
    & =M_n^2 m_n^{-4} L_n \|u_{n,0}\|_{\dot H^{-\frac 1 2}(\T_n)}^2 \|\mathring\phi_n^0\|_{\dot H^2(\T_n)}^2\notag\\
    &\lesssim M_n^2 m_n^{-4}N_n^3 L_n^{-2} \|u_{n,0}\|_{\dot H^{-\frac 1 2}(\T_n)}^2
\end{align}
and similarly
\begin{align} \label{low -1}
    \sum_{k=0}^{N_n-1} & \left\|P^{L_n}_{\leq \frac{m_n}{8}}\left[ \left(P^{L_n}_{\frac{m_n}{8}<\dots \leq 8M_n}\mathring\phi^k_n\right) u_{n,0}\right]\right\|_{\dot H^{-1}(\T_n)}^2\notag\\
    &\lesssim M_n^2 m_n^{-6}N_n^5 L_n^{-3} \|u_{n,0}\|_{\dot H^{-\frac 1 2}(\T_n)}^2
\end{align}
Combining estimates \eqref{high -1/2} and \eqref{low -1/2} with \eqref{eq:sum} and estimates \eqref{high -1} and \eqref{low -1} with \eqref{eq:sum2}, and assuming that $n$ is sufficiently large so that $L_n\gg m_n^{-2} M_n N_n^{2}$,
\begin{align*}
    \sum_{k=0}^{N_n-1} \left\|\mathring\phi^k_n u_{n,0}\right\|_{\dot H^{-\frac 1 2}(\T_n)}^2 \leq C m_n^{-1}M_n \|u_{n,0}\|_{\dot H^{-\frac 1 2}(\T_n)}^2
\end{align*}
and
\begin{align*}
    \sum_{k=0}^{N_n-1} \left\|\mathring\phi^k_n u_{n,0}\right\|_{\dot H^{-1}(\T_n)}^2 \leq C m_n^{-2}M_n \|u_{n,0}\|_{\dot H^{-\frac 1 2}(\T_n)}^2
\end{align*}
where the constant $C$ does not depend on $n$. This implies that for at least $\frac{9}{10}N_n$ of the integers $k\in \{0,\dots, N_n-1\}$
\begin{align*}
    \left\|\mathring\phi^k_n u_{n,0}\right\|_{\dot H^{-\frac 1 2}(\T_n)} \leq \sqrt{10C} m_n^{-\frac 1 2} M_n^{\frac 1 2}  N_n^{-\frac 1 2}\|u_{n,0}\|_{\dot H^{-\frac 1 2}(\T_n)}
\end{align*}
and similarly for at least $\frac{9}{10}N_n$ of the integers $k\in \{0,\dots, N_n-1\}$
\begin{align*}
    \left\|\mathring\phi^k_n u_{n,0}\right\|_{\dot H^{-1}(\T_n)} \leq \sqrt{10C} m_n^{-1} M_n^{\frac 1 2}  N_n^{-\frac 1 2}\|u_{n,0}\|_{\dot H^{-\frac 1 2}(\T_n)}.
\end{align*}
\end{proof}


Now we consider
\begin{align*}
    S_1&:=\left\{0\leq k\leq N_n-1: \left\|\mathring\phi^k_n u_{n,0}\right\|_{\dot H^{-\frac 1 2}(\T_n)}<C m_n^{-\frac 1 2} M_n^{\frac 1 2}  N_n^{-\frac 1 2}\|u_{n,0}\|_{\dot H^{-\frac 1 2}(\T_n)}\right\},\\
    S_2&:=\left\{0\leq k\leq N_n-1: \left\|\mathring\phi^k_n u_{n,0}\right\|_{\dot H^{-1}(\T_n)}<C m_n^{-1} M_n^{\frac 1 2}  N_n^{-\frac 1 2}\|u_{n,0}\|_{\dot H^{-\frac 1 2}(\T_n)}\right\},\\
    S_3&:=\{k\in S_1\cap S_2: \dist(0,\supp (\phi_n^k))>10 \tfrac{L_n}{N_n}\},\\
    S_4&:= \{k\in S_3: k+1\in S_3\}, 
\end{align*}
where the constant $C$ in the definition of $S_1$ and $S_2$ is inherited from the Lemma above.

In the light of Lemma \ref{bump}, $S_1$ and $S_2$ consist of at least $\frac{9}{10}N_n$ elements, therefore $S_1\cap S_2$ consists of at least $\frac{8}{10}N_n$ elements. By restricting ourselves to bumps $\phi_n^k$ that are supported at least $\frac{10 L_n}{N_n}$ away from $0$, we have to remove at most 30 elements, so we conclude that $S_3$ has at least $\frac{7}{10} N_n$ elements (assuming $n$ is large enough). However, we cannot conclude that $\mathring\phi^k_n u_{n,0}\in\dot H^{-\frac 1 2}(\T_n)$ for $k\in S_3$; we also need $\int\mathring\phi^k_n u_{n,0}=0$, which is not necessarily true. 

To remedy that we work as follows: If there exists one $k_0\in S_3$ so that $\int\mathring\phi^{k_0}_n u_{n,0}=0$, we choose 
\begin{align}\label{case1}
    \varphi_n= \phi^{k_0}_n.
\end{align} 
If not, we look at $S_4$. Since $S_3$ contains at least $\frac{7}{10}$ of the integers between $0$ and $N_n-1$, it must contain at least two consecutive integers, in other words there exists $k_1\in S_4$. If $0<|\int\mathring\phi^{k_1}_n u_{n,0}|\leq |\int\mathring\phi^{k_1+1}_n u_{n,0}|$, we chose 
\begin{align}\label{case2}
    \varphi_n:= \phi_n^{k_1}- \frac {\int\mathring\phi^{k_1}_n u_{n,0}}{\int\mathring\phi^{k_1+1}_n u_{n,0}}\phi_n^{k_1+1},
\end{align}
otherwise we choose
\begin{align}\label{case3}
    \varphi_n:= \phi_n^{k_1+1}- \frac {\int\mathring\phi^{k_1+1}_n u_{n,0}}{\int\mathring\phi^{k_1}_n u_{n,0}} \phi_n^{k_1}.
\end{align}
In the following we will denote the $L_n$-periodization of $\varphi_n$ by $\mathring{\varphi_n}$. One can see that $\int\mathring\varphi_n u_{n,0}= 0$,
\begin{align} \label{small -1/2}
    \left\|\mathring\varphi_n u_{n,0}\right\|_{\dot H^{-\frac 1 2}(\T_n)}&\leq \left\|\mathring\phi^{k_1}_n u_{n,0}\right\|_{\dot H^{-\frac 1 2}(\T_n)}+ \left\|\mathring\phi^{k_1+1}_n u_{n,0}\right\|_{\dot H^{-\frac 1 2}(\T_n)}\notag\\
    &\lesssim m_n^{-\frac 1 2} M_n^{\frac 1 2}  N_n^{-\frac 1 2}\|u_{n,0}\|_{\dot H^{-\frac 1 2}(\T_n)}
\end{align}
and
\begin{align} \label{small -1}
    \left\|\mathring\varphi_n u_{n,0}\right\|_{\dot H^{-1}(\T_n)}\lesssim m_n^{-1} M_n^{\frac 1 2}  N_n^{-\frac 1 2}\|u_{n,0}\|_{\dot H^{-\frac 1 2}(\T_n)}.
\end{align}


\subsection{`Unwrapping' on the line}

Now we consider the following smooth compactly supported functions. We start by defining
\begin{align*}
    \chi_n^1(x):= \sum_{k=0}^{N_n-1}\phi_n^k(x).
\end{align*}
Then we take $\chi_n^2$ to be a translation of $\chi_n^1$, defined by
\begin{align*}
\chi_n^2(x):=
\begin{cases}
    \chi_n^1(x-k_0 \tfrac{L_n}{N_n})\qquad &\text{if}\,\,\varphi_n\,\,\text{is given by}\,\,\eqref{case1},\\
    \chi_n^1(x-k_1 \tfrac{L_n}{N_n})\qquad &\text{if}\,\,\varphi_n\,\,\text{is given by}\,\,\eqref{case2},\\
    \chi_n^1(x-(k_1+2) \tfrac{L_n}{N_n})\qquad &\text{if}\,\,\varphi_n\,\,\text{is given by}\,\,\eqref{case3},\\
    \end{cases}
\end{align*}
and
\begin{align*}
    \chi_n^0(x):= 
    \begin{cases}
        (\chi_n^2-\varphi_n)(x)\qquad & \text{if}\,\,0\,\,\text{is on the right of the support of}\,\,\varphi_n,\\
        (\chi_n^2-\varphi_n)(x+L_n)\qquad & \text{if}\,\,0\,\,\text{is on the left of the support of}\,\,\varphi_n.
    \end{cases}
\end{align*}
Observe that by construction the support of $\chi_n^0$ is connected and contains $0$. By \eqref{partition}, their $L_n$-periodizations satisfy $\mathring \chi_n^0+\mathring \varphi_n=1$. We take the initial data 
\begin{align*}
    q_{n,0}:= \chi_n^0 u_{n,0}\in C_c^\infty(\mathbb R),
\end{align*}
which is supported inside an interval of length $L_n$, and its $L_n$-periodization
$$\mathring q_{n,0}= \mathring \chi_n^0 u_{n,0}.$$
Then
\begin{align*}
    \int_{\R} q_{n,0} =\int_{\T_n} \mathring\chi_n^0 u_{n,0}= \int_{\T_n} u_{n,0} - \int_{\T_n} \mathring\varphi_n u_{n,0}=0.
\end{align*}
By Lemma \ref{periodization}, the definition of $q_{n,0}$, and \eqref{small -1/2}, it is clear that for $n$ sufficiently large and assuming that $N_n\ll m_n^{-1} M_n$
\begin{align} \label{eq:H-1 bound}
    \|q_{n,0}\|_{H^{-1}(\mathbb R)}=  \|\mathring q_{n,0}\|_{H^{-1}(\mathbb T_n)}\leq \|u_{n,0}\|_{H^{-1}(\mathbb T_n)} + \|\mathring \varphi_n u_{n,0}\|_{H^{-1}(\mathbb T_n)}\leq 2 A<\delta_0.
\end{align}
Moreover, using \eqref{small -1} and the fact that $0\leq \varphi_n\leq 1$ we get that
\begin{align} \label{eq:dot -1 bound}
    \|q_{n,0}\|_{\dot H^{-1}(\mathbb R)}& \leq \|u_{n,0}\|_{\dot H^{-1}(\mathbb T_n)} + \|\mathring \varphi_n u_{n,0}\|_{\dot H^{-1}(\mathbb T_n)}\notag\\
    &\lesssim m_n^{-\frac 12}  \|u_{n,0}\|_{\dot H^{-\frac 1 2}(\mathbb T_n)}+ m_n^{-1}M_n^{\frac 12} N_n^{-\frac 12} \|u_{n,0}\|_{\dot H^{-\frac 1 2}(\mathbb T_n)}\notag\\
    &\lesssim m_n^{-\frac 12}\|u_{n,0}\|_{\dot H^{-\frac 1 2}(\mathbb T_n)}
\end{align}
and
\begin{align} \label{eq:L2 bound}
    \|q_{n,0}\|_{L^2(\mathbb R)}& \leq \|u_{n,0}\|_{L^2(\mathbb T_n)} + \|\mathring \varphi_n u_{n,0}\|_{L^2(\mathbb T_n)}\leq 2\|u_{n,0}\|_{L^2(\mathbb T_n)}\notag\\
    &\lesssim M_n^{\frac 12}  \|u_{n,0}\|_{\dot H^{-\frac 1 2}(\mathbb T_n)}.
\end{align}
Estimates \eqref{eq:dot -1 bound} and \eqref{eq:L2 bound} imply
\begin{align*}
    \|q_{n,0}\|_{\dot H^{-\frac 1 2}(\mathbb R)} ^2&\leq \|q_{n,0}\|_{\dot H^{-1}(\mathbb R)} \|q_{n,0}\|_{L^2(\mathbb R)}\lesssim m_n^{-\frac 12} M_n^{\frac 12} A^2
\end{align*}
so, indeed, $q_{n,0}\in \dot H^{-\frac 12}(\mathbb R)$. Our next goal is to obtain a uniform bound.


\begin{lemma}\label{H-1/2 bound} For $n$ sufficiently large $q_{n,0}\in \dot H^{-\frac 1 2}(\mathbb R)$ and
$$\|q_{n,0}\|_{\dot H^{-\frac 1 2}(\mathbb R)} \lesssim A,$$
where the implicit constant does not depend on $n$.
\end{lemma}

\begin{proof}

We split $q_{n,0}= P_{\leq \frac {m_n}{2}} q_{n,0}+ P_{\frac {m_n}{2}<\dots\leq 2 M_n} q_{n,0}+P_{> 2M_n} q_{n,0}$. 

For the low frequency part, Bernstein estimates and \eqref{eq:dot -1 bound} give us
\begin{align*}
    \|P_{\leq \frac {m_n}{2}} q_{n,0}\|_{\dot H^{-\frac 1 2}(\mathbb R)}\lesssim m_n^{\frac 1 2}\|q_{n,0}\|_{\dot H^{-1}(\mathbb R)}\lesssim A
\end{align*}
for $n$ sufficiently large.

Bernstein estimates and \eqref{eq:L2 bound} also allow us to control the high frequency part:
\begin{align*}
    \|P_{> 2M_n} q_{n,0}\|_{\dot H^{-\frac 1 2}(\mathbb R)}\lesssim M_n^{-
    \frac 1 2} \| q_{n,0}\|_{L^2(\mathbb R)}\lesssim A.
\end{align*}

Now we turn our attention to the middle frequency part. To this end, we will compare $\left\| P_{\frac {m_n}{2}<\dots\leq 2 M_n} q_{n,0} \right\|_{\dot H^{-\frac 1 2}(\mathbb R)}$ and $\left\| P^{L_n}_{\frac {m_n}{2}<\dots\leq 2 M_n} \mathring q_{n,0} \right\|_{\dot H^{-\frac 1 2}(\mathbb T_n)}$.
We observe that
\begin{align*}
    \left\| P_{\frac {m_n}{2}<\dots\leq 2 M_n} q_{n,0} \right\|_{\dot H^{-\frac 1 2}(\mathbb R)}^2& =\int_{\mathbb R} \int_{\mathbb R} K_n(x-y) q_{n,0}(x) q_{n,0}(y) dx dy,\\
    \left\| P^{L_n}_{\frac {m_n}{2}<\dots\leq 2 M_n} \mathring q_{n,0} \right\|_{\dot H^{-\frac 1 2}(\mathbb T_n)}^2&= \int_{\mathbb R}\int_{\mathbb R} \frac{1}{L_n} K_n^{L_n}(x-y) q_{n,0}(x) q_{n,0}(y) dx dy,
\end{align*}
where $K_n$ and $K_n^{L_n}$ are the inverse Fourier transform (in $\mathbb R$ and $\mathbb T_n$ respectively) of $\mathcal K_n$,
$$\mathcal K_n(\xi):= |\xi|^{-1} \mu_n^2(\xi)= |\xi|^{-1} \left(m(\tfrac{\xi}{2M_n})- m(\tfrac{2 \xi}{m_n}) \right)^2.$$
Note that $\mu_n$, and consequently $\mathcal K_n$, are supported on $\{\xi\in \mathbb R: \frac {m_n}{2}\leq |\xi|\leq 4 M_n\}$. Thus $\mathcal K_n\in C_c^\infty(\mathbb R)$ with
\begin{align*}
    \|\mathcal K_n\|_{L^\infty(\mathbb R)}\lesssim m_n^{-1},\qquad \|\mathcal K_n ^{(\alpha)} \|_{L^\infty(\mathbb R)}\lesssim_{\alpha} m_n^{-(\alpha+1)}\quad\text{for all}\,\,\alpha\in\mathbb N.
\end{align*}
We conclude that $K_n$ satisfies
\begin{align*}
    \left| K_n(x) \right|\lesssim M_n m_n^{-1} , \quad \left| K_n(x) \right|\lesssim M_n m_n^{-4} |x| ^{-3}\quad\text{for all}\,\, x\in \mathbb R.
\end{align*}
By Cauchy-Schwarz and Minkowski inequalities in combination with the above estimate and \eqref{eq:L2 bound}, we get that
\begin{align}
\label{eq:largeR}
    \left| \iint_{\{\dist(x-y,L_n\mathbb Z )>A_n\}}  K_n(x-y) q_{n,0}(x) q_{n,0}(y) dx dy  \right| &\lesssim  \| K_n\|_{L^1(\{|x|>A_n\})}    \|q_{n,0}\|_{L^2(\mathbb R)}^2\notag\\
    &\lesssim M_n m_n^{-4} A_n^{-2} \|q_{n,0}\|_{L^2(\mathbb R)}^2\notag\\
    & \lesssim M_n^2 m_n^{-4} A_n^{-2} A^2.
\end{align}
On the other hand, a similar argument allows us to estimate the inverse Fourier transform of $(\mathcal K_n(k))_{k\in\frac{1}{L_n} \mathbb Z}$:
\begin{align*}
    K_n^{L_n}(x)&= \sum_{k\in\frac{1}{L_n} \mathbb Z} \mathcal K_n(k) e^{2\pi i xk}\\
    & =  \sum_{k\in\frac{1}{L_n} \mathbb Z} \mathcal K_n(k) \frac{e^{2\pi i x(k+\frac{3}{L_n})}- 3 e^{2\pi i x(k+\frac{2}{L_n})} +3 e^{2\pi i x(k+\frac{1}{L_n})}- e^{2\pi i xk}}{(e^{2\pi i x \frac{1}{L_n}}-1)^3}\\
    & = \sum_{k\in\frac{1}{L_n} \mathbb Z} \frac{ \mathcal K_n(k-\tfrac{3}{L_n}) -3 \mathcal K_n(k-\tfrac{2}{L_n}) +3 \mathcal K_n(k-\tfrac{1}{L_n})- \mathcal K_n(k)}{(e^{2\pi i x \frac{1}{L_n}}-1)^3} e^{2\pi i xk}.
\end{align*}
Since
\begin{align*}
    \left| e^{2\pi i x \frac{1}{L_n}}-1 \right| \gtrsim \frac{|x|}{L_n} \quad\text{for}\quad |x|<\frac{L_n}{2}
\end{align*}
and
\begin{align*}
    \left| \mathcal K_n(k-\tfrac{3}{L_n}) -3 \mathcal K_n(k-\tfrac{2}{L_n}) +3 \mathcal K_n(k-\tfrac{1}{L_n})- \mathcal K_n(k) \right| \lesssim \frac{\|\mathcal K_n^{(3)}\|_{L^\infty}}{L_n^{3} }\lesssim  \frac {m_n^{-4}}{L_n^{3}},
\end{align*}
\begin{align*}
    \left| K_n^{L_n}(x) \right|\lesssim L_n M_n m_n^{-4} |x|^{-3}\quad\text{for}\,\,|x|<\frac{L_n}{2}.
\end{align*}
We also have that
\begin{align*}
    \left| K_n^{L_n}(x) \right|\lesssim L_n M_n \|\mathcal K_n\|_{L^\infty}\lesssim L_n M_n m_n^{-1}\quad\text{for all}\,\,x.
\end{align*}
Working similarly as in \eqref{eq:largeR}, the estimates for $K_n^{L_n}$ we just obtained allow us to get
\begin{align}
\label{eq:largeT}
    \left| \iint_{\{\dist(x-y,L_n \mathbb Z)>A_n\}} \frac{1}{L_n} K_n^{L_n}(x-y) q_{n,0}(x) q_{n,0}(y) dx dy \right| \lesssim M_n^2 m_n^{-4} A_n^{-2} A^2.
\end{align}
In the region $\{\dist(x-y,L_n\mathbb Z)\leq A_n\}$ the kernels are not small; however, we will show that they are close to each other, as expected. Indeed, for $x$ such that $\dist(x,L_n\mathbb Z)\leq A_n$
\begin{align*}
    \left| K_n(x)-  \frac {1}{L_n}K_n^{L_n}(x) \right|&= \left| \int_{\mathbb R} \mathcal K_n(\xi) e^{2 \pi i x \xi}  d\xi  - \frac{1}{L_n}\sum_{k\in\frac{1}{L_n} \mathbb Z} \mathcal K_n(k) e^{2\pi i xk} \right| \\
    &= \left| \sum_{k\in\frac{1}{L_n}\mathbb Z} \left[ \int_{I_k(\tfrac{1}{L_n})} \mathcal K_n(\xi) e^{2 \pi i x \xi}  d\xi  - \frac{1}{L_n}  \mathcal K_n(k) e^{2\pi i xk} \right] \right|\\
    &\lesssim \sum_{k\in\frac{1}{L_n}\mathbb Z} L_n^{-2} \|\left( \mathcal K_n(\xi) e^{2 \pi i x \xi} \right)''\|_{L^\infty(I_k(\tfrac{1}{L_n}))}\\
    &\lesssim L_n^{-1} M_n m_n^{-1} A_n^2
\end{align*}
where $I_k(\tfrac{1}{L_n})=\left[k-\tfrac{1}{2L_n},k+\tfrac{1}{2L_n}\right)$. Then the same argument as before yields
\begin{align}
\label{eq:R-T}
    \left| \iint_{\{\dist(x-y, L_n\mathbb Z)\leq A_n\}} \left[K_n(x-y)-  \frac {1}{L_n}K_n^{L_n}(x-y)\right]  q_{n,0}(x) q_{n,0}(y) dx dy \right| \lesssim    \notag\\
    \lesssim L_n^{-1} M_n^2 m_n^{-1} A_n^3 A^2.
\end{align}
By \eqref{eq:largeR}, \eqref{eq:largeT} and \eqref{eq:R-T}, 
\begin{align*}
    \left\| P_{\frac {m_n}{2}<\dots\leq 2 M_n} q_{n,0} \right\|_{\dot H^{-\frac 1 2}(\mathbb R)}^2-
    \left\| P^{L_n}_{\frac {m_n}{2}<\dots\leq 2 M_n} \mathring q_{n,0} \right\|_{\dot H^{-\frac 1 2}(\mathbb T_n)}^2\\
    \lesssim M_n^2 m_n^{-4} (A_n^{-2}+ L_n^{-1} A_n^3)A^2,
\end{align*}
therefore, for $n$ sufficiently large and given that $M_n  m_n^{-2} \ll A_n\ll L_n^{\frac 12} M_n^{-1} m_n^{\frac 1 2}$,
\begin{align*}
    \left\| P_{\frac {m_n}{2}<\dots\leq 2 M_n} q_{n,0} \right\|_{\dot H^{-\frac 1 2}(\mathbb R)}\lesssim A.
\end{align*}

\end{proof}

\subsection{Local behavior}

Having established a uniform bound for the sequence $q_{n,0}$ in $\dot H^{-\frac 12}(\mathbb R)$, Theorem \ref{Hk wp} provides us with solutions $q_n$ to \ref{eq:Hk} with initial data $q_{n,0}$. 
Furthermore,
\begin{align*}
    \|q_n(t)\|_{\dot H^{-\frac 1 2}(\R)} \lesssim \|q_{n,0}\|_{\dot H^{-\frac 1 2}(\R)} \lesssim A\quad\text{for all}\,\,t\in[0,T].
\end{align*}

By construction the new sequence of initial data $q_{n,0}$ is close to $u_{n,0}$ in the sense that
\begin{align*}
    \|u_{n,0}-\mathring q_{n,0}\|_{\dot H^{-\frac 12}(\mathbb T_n)}\to 0\qquad\text{as}\,\, n\to\infty.
\end{align*}
Next, we will show that $q_n(t)$ stays close to the solution $u_n(t)$ to \eqref{eq:Hkn} with initial data $u_{n,0}$ on the time interval $[0,T]$. This is understood in the sense that $q_n(t)$ locally, around the support of $q_{n,0}$, behaves similarly to $u_n(t)$. Below we formulate a more precise statement and study this local behavior.

Assuming that $\chi_n^0$ is supported on the interval $[a_n, b_n]$, we consider a smooth compactly supported function $\chi_n^\ast$ such that
\begin{align*}
    \chi_n^\ast(x)=
    \begin{cases}
        1,\qquad &x\in [a_n-\tfrac{L_n}{20N_n}, b_n +\tfrac{L_n}{20N_n}]\\
        0,\qquad &x\notin [a_n-\tfrac{L_n}{10N_n}, b_n +\tfrac{L_n}{10N_n}],
    \end{cases}
\end{align*}
as well as its $L_n$-periodization
\begin{align*}
    \mathring \chi_n^\ast(x)=\sum_{k\in \mathbb Z} \chi_n^\ast(x+k L_n).
\end{align*}
We also define
\begin{align*}
    f_n(t)&:= \chi_n^\ast q_n(t),\\
    \mathring f_n(t,x)&= \sum_{k\in \mathbb Z} f_n(t, x+k L_n).
\end{align*}
Note that by the definition $\chi_n^0 (1-\chi_n^\ast)=0$, so $f_n(0)=q_{n,0}$, $u_{n,0}-\mathring f_n(0)= u_{n,0}\mathring \varphi_n$.

\begin{lemma}
\label{locality}
Let $\chi$ be smooth with $\chi'\in C_c^\infty(\R)$, $q\in H^{-1}(\R)$ with $\|q\|_{H^{-1}}<\min(\delta,1)$. Then 
\begin{align*}
    \|\chi g'(q)\|_{L^2(\R)} \lesssim  \|\chi q\|_{H^{-1}(\R)}+\|\chi'\|_{L^2 (\R)} + \|\chi'\|_{L^\infty(\R)}. 
\end{align*}
\end{lemma}

\begin{proof}
We will exploit the series expansion of $g$ \eqref{series}. First of all, one can easily verify that
\begin{align}
\label{commutator}
    [\chi, R_0]= \sqrt{R_0} A(\chi) \sqrt{R_0}\quad\text{with}\quad A(\chi)=\sqrt{R_0} \big(-2\partial_x \chi'(x)+\chi''(x)\big) \sqrt{R_0}.
\end{align}
In the following we denote $r(T) :=\sqrt{R_0}T\sqrt{R_0}$. Then for any $T, S$ we get
\begin{align}\label{move cutoff}
    r(T \chi)r(S)&= \sqrt{R_0} T \chi R_0 S \sqrt{R_0}\\
    &= \sqrt{R_0} T  R_0 \chi S \sqrt{R_0} + \sqrt{R_0} T [\chi, R_0] S \sqrt{R_0}\\
    &= r(T) r(\chi S)+ r(T) A(\chi) r(S).
\end{align}

Then for any $f\in H^{-1}(\R)$ with $\|f\|_{H^{-1}(\R)}=1$
\begin{align*}
    \int \chi(x) & \left[g(q)(x)-\frac{1}{2\kappa}\right] f(x) dx =\\
    &= \sum_{l\geq 1} (-1)^l \left[\tr \{r(f) r(\chi q) r(q)^{l-1} \} +\tr\{ r(f)A(\chi) r(q)^l\}\right],
\end{align*}
so by \eqref{op norm}
\begin{align*}
    \left|\int \chi(x)  \left[g(q)(x)-\frac{1}{2\kappa}\right] f(x) dx\right| \lesssim  \|\chi q\|_{H^{-1}(\R)} +\|\chi\|_{\dot H^1 (\R)}.
\end{align*}
This implies that 
\begin{align*}
    \left\|\chi \left[g(q)-(2\kappa)^{-1}\right]\right\|_{H^1(\R)}\lesssim  \|\chi q\|_{H^{-1}(\R)} +\|\chi'\|_{L^2 (\R)},
\end{align*}
thus
\begin{align*}
    \left\|\chi g'(q)\right\|_{L^2(\R)}&\lesssim  \left\|\chi \left[g(q)-(2\kappa)^{-1}\right]\right\|_{\dot H^1(\R)} + \left\|\chi' \left[g(q)-(2\kappa)^{-1}\right]\right\|_{L^2(\R)} \\
    &\lesssim \|\chi q\|_{H^{-1}(\R)} +\|\chi'\|_{L^2 (\R)} + \|\chi'\|_{L^\infty(\R)}.
\end{align*}
\end{proof}

\begin{theorem}
\label{finite speed}
Let $\chi$ be smooth with $\chi'\in C_c^\infty(\R)$, $q(t)$ solution to \ref{eq:Hk} with initial data $q(0)\in H^{-1}(\R)$, $\|q(0)\|_{H^{-1}}<\delta_0$. Then for all $t\in [0,T]$
\begin{align*}
    \|\chi q(t)\|_{H^{-1}(\R)}\lesssim \|\chi(x-4 \kappa^2t)q(0,x)\|_{H^{-1}(\R)}+ \|\chi\|_{\dot H^{1}(\R)}+\|\chi'\|_{L^\infty(\R)}.
\end{align*}
In particular, if $q(0)$ is compactly supported with $\supp(q(0))\subset [a,b]$ and $\chi=0$ on $[a-10\kappa^2 T,b+10\kappa^2 T ]$, then for all $t\in [0,T]$
\begin{align*}
    \|\chi q(t)\|_{H^{-1}(\R)}\lesssim \|\chi\|_{\dot H^{1}(\R)} +\|\chi'\|_{L^\infty(\R)}.
\end{align*}
The implicit constants do not depend on $t$.
\end{theorem}

\begin{proof}
Let $t\in [0,T]$. We consider 
\begin{align*}
    F(s):= \|\chi e^{4(t-s)\kappa^2 \partial_x}q(s)\|_{H^{-1}(\R)}.
\end{align*}
By the Duhamel formula,
\begin{align*}
    \|\chi q(t)\|_{H^{-1}(\R)}\lesssim \|\chi e^{4t\kappa^2 \partial_x}q(0)\|_{H^{-1}(\R)}+\int_0^t \|\chi e^{4(t-s)\kappa^2 \partial_x}g'(q(s))\|_{H^{-1}(\R)} ds,
\end{align*}
or equivalently
\begin{align*}
    F(t)\lesssim F(0)+\int_0^t \|\chi e^{4(t-s)\kappa^2 \partial_x}g'(q(s))\|_{H^{-1}(\R)} ds.
\end{align*}
We observe that for any $f$
\begin{align}
    \label{linear}
    \left(e^{-4(t-s)\kappa^2 \partial_x}f\right)(x)= f\left(x-4(t-s)\kappa^2 \right).
\end{align}
In particular, we have that for any $f_1, f_2$
\begin{align*}
    \left|f_1 \left(e^{4(t-s)\kappa^2 \partial_x}f_2\right)\right|=\left|\left(e^{-4(t-s)\kappa^2 \partial_x}f_1\right) f_2\right|.
\end{align*}
Therefore by Lemma \ref{locality}
\begin{align*}
    \left\|\chi \left(e^{4(t-s)\kappa^2 \partial_x}g'(q(s))\right)\right\|_{H^{-1}}&=\left\|\left(e^{-4(t-s)\kappa^2 \partial_x}\chi\right) g'(q(s))\right\|_{H^{-1}}\\
    &\lesssim \left\|\left(e^{-4(t-s)\kappa^2 \partial_x}\chi\right) q(s)\right\|_{H^{-1}}+ \left\|\chi' \right\|_{L^2}+ \left\|\chi' \right\|_{L^\infty}\\
    &=\left\|\chi \left(e^{4(t-s)\kappa^2 \partial_x}q(s)\right)\right\|_{H^{-1}}+ \left\|\chi' \right\|_{L^2}+ \left\|\chi' \right\|_{L^\infty}\\
    &=F(s)+ \left\|\chi' \right\|_{L^2}+ \left\|\chi' \right\|_{L^\infty}.
\end{align*}
All in all,
\begin{align*}
    F(t)\lesssim F(0)+ \left\|\chi' \right\|_{L^2}+ \left\|\chi' \right\|_{L^\infty}+\int_0^t F(s)ds
\end{align*}
so Gr\"onwall's inequality ensures that $F(t)\lesssim F(0)+ \left\|\chi' \right\|_{L^2}+ \left\|\chi' \right\|_{L^\infty}$ or, equivalently,
\begin{align*}
    \|\chi q(t)\|_{H^{-1}(\R)}\lesssim \left\|\left(e^{-4t\kappa^2 \partial_x}\chi\right) q(0)\right\|_{H^{-1}(\R)}+ \left\|\chi' \right\|_{L^2}+ \left\|\chi' \right\|_{L^\infty}.
\end{align*}
Note that the implicit constants are allowed to depend on $T$ but not $t$.

If $q(0)$ is compactly supported with $\supp(q(0))\subset [a,b]$ and $\chi(x)=0$ for all $x\in[a-10\kappa^2 T,b+10\kappa^2 T ]$, then $\left(e^{-4t\kappa^2 \partial_x}\chi\right)(x)=\chi\left(x-4t\kappa^2 \right)=0$ for all $x\in[a-4\kappa^2 T,b+4\kappa^2 T ]$ and we get that for all $t\in [0,T]$
\begin{align*}
    \|\chi q(t)\|_{H^{-1}(\R)}\lesssim \|\chi\|_{\dot H^{1}(\R)} +\|\chi'\|_{L^\infty(\R)}.
\end{align*}
\end{proof}

In our setting, Theorem \ref{finite speed} implies that for $n$ sufficiently large so that $\kappa^2 T\ll \frac{L_n}{N_n}$
\begin{align}
\label{eq:finite speed cor}
    \|(1-\chi_n^\ast) q_n(t)\|_{H^{-1}(\R)}\lesssim N_n ^{\frac 1 2} L_n^{-\frac 1 2} \quad\text{for all}\quad t\in [0,T].
\end{align}
Since $\frac{N_n}{L_n}\to 0$ as $n\to \infty$, $f_n(t)-q_n(t)$ converges to  $0$ weakly in $H^{-1}(\R)$ uniformly for all $t\in [0,T]$.

Note that $\mathring f_n$ solves 
\begin{equation*} 
    \begin{cases}
    \frac{d}{dt} \mathring f_n = 4\kappa^2 \mathring f_n'+16\kappa^5 g'(\mathring f_n)+\mathring e_n\\
     \mathring f_n(0)= \mathring \chi_n^0 u_{n,0}
    \end{cases}
\end{equation*}
with 
\begin{align*}
    e_n= 16\kappa^5 \big(g'(q_n) -g'(\chi_n^\ast q_n) -(1-\chi_n^\ast)g'(q_n) \big) -4\kappa^2 (\chi_n^\ast) 'q_n.
\end{align*}
Using the diffeomorphism property of $g$, Lemma \ref{locality}, and \eqref{eq:finite speed cor}, we get that for all $t\in [0,T]$ the error term $\|\mathring e_n(t)\|_{H^{-1}(\mathbb T_n)}$ is bounded (up to a constant) by
\begin{align*}
    \|(1-\chi_n^\ast)g'(q_n)\|&_{H^{-1}(\mathbb R)}+ \|g'(q_n) -g'(\chi_n^\ast q_n)\|_{H^{-1}(\mathbb R)} + \|(\chi_n^\ast) 'q_n\|_{H^{-1}(\mathbb R)}\\
    & \lesssim N_n ^{\frac 1 2} L_n^{-\frac 1 2} + \|(1-\chi_n^\ast)q_n\|_{H^{-1}(\mathbb R)}+ \|(\chi_n^\ast) '\|_{H^{1}(\mathbb R)} \|q_n\|_{H^{-1}(\mathbb R)}\\
    &\lesssim N_n ^{\frac 1 2} L_n^{-\frac 1 2}(1+ A).
\end{align*}
Now let's consider $u_n$ the solution (up to time $T_n>T$) to \eqref{eq:Hkn} with initial data $u_{n,0}$. By Theorem \ref{Hkn to Hk}, for all $t\in[0,T]$
\begin{align*}
    \|u_n(t)-\mathring f_n(t)\|_{H^{-1}(\T_n)}&\lesssim \|\mathring \varphi_n u_{n,0}\|_{H^{-1}(\T_n)}+ (m_n^{\frac 12} + M_n^{-\frac 12})A + N_n ^{\frac 1 2} L_n^{-\frac 1 2} (1+A)\\
    &\lesssim \left(m_n^{\frac 12} + M_n^{-\frac 12} + N_n ^{\frac 1 2} L_n^{-\frac 1 2}+ m_n ^{-\frac 12} M_n^{\frac 12} N_n ^{-\frac 1 2} \right) (1+A)
\end{align*}
which approaches $0$ as $n\to \infty$. 

The following theorem records the results of this section:

\begin{theorem}[Finite Dimensional PDE Approximation]
\label{fda}
Fix $\kappa\geq1$, $T>0$, $0<A<\frac{\delta_0}{4}$ and let $m_n\to 0$, $M_n, N_n, L_n\to\infty$ as above. Assume $u_{n,0}\in \dot H^{-\frac 12}(\mathbb T_n)$ with $\|u_{n,0}\|_{\dot H^{-\frac 12}(\mathbb T_n)}\leq A$ and $P^{L_n}_{\leq m_n}u_{n,0}=0=P^{L_n}_{>M_n}u_{n,0}$. Let $u_n$ be the solution to \eqref{eq:Hkn} with initial data $u_n(0)=u_{n,0}$. Then there exist $q_{n,0}\in \dot H^{-\frac 12}(\mathbb R)$ compactly supported on an interval of length $L_n$ containing $0$, satisfying $\|q_{n,0}\|_{\dot H^{-\frac 12}(\mathbb R)}\lesssim A$, so that
\begin{align*}
    \lim_{n\to\infty}\|u_{n,0}-\mathring q_{n,0}\|_{\dot H^{-\frac 12}(\mathbb T_n)}= \lim_{n\to\infty}\|u_{n,0}-\mathring q_{n,0}\|_{\dot H^{-1}(\mathbb T_n)}=0.
\end{align*}
Moreover, there exist bump functions $\chi_n^\ast\in C_c^\infty(\mathbb{R})$ satisfying $\chi_n^\ast=1$ on the support of $q_{n,0}$, for which the following is true for the solutions $q_n$ to \eqref{eq:Hk} with initial data $q_n(0)=q_{n,0}$:
$$\lim_{n\to\infty}\left\|u_n(t)-\mathring{\left[\chi_n^\ast q_n(t)\right]}_{L_n}\right\|_{H^{-1}(\T_n)}=0\quad\text{for all}\quad t\in[0,T].$$
\end{theorem}

\section{Weak well-posedness for the $H_\kappa$ flow} \label{sec;weak wp}

The process carried out in the previous section leaves us with a sequence of initial data that now lives in the same space: $\dot H^{-\frac 12}(\mathbb R)$. Moreover, as we saw it is uniformly bounded in this space. Unfortunately, the only conclusion we can draw from this is the existence of a weak limit. Naturally, the following question arises: If we have weak convergence of the initial data, what does this imply for the behavior of the solutions to \eqref{eq:Hk} at later times in the weak topology? The key to answering this question is equicontinuity in time.

\begin{lemma}
\label{Hk eqcts}
The \ref{eq:Hk} flow is equicontinuous in time in $H^{-1}(\R)$, i.e., for every $T>0$, $Q\subset B_{\delta_0}$ equicontinuous in $H^{-1}$ and $\varepsilon>0$ there exists $\delta>0$ such that for all $q$ that solve \eqref{eq:Hk} with initial data $q(0)$ in $Q$, we have
$$\|q(t)-q(s)\|_{H^{-1}(\R)}<\varepsilon$$
for all $t,s\in[0,T]$ with $|t-s|<\delta$.
\end{lemma}

\begin{proof}
We fix $T>0$, $Q\subset B_{\delta_0}$ equicontinuous in $H^{-1}(\mathbb R)$, and $\varepsilon>0$. In the following we will allow the implicit constants to depend on $T$ and $Q$. Let $t,s\in[0,T]$. Without loss of generality we may assume that $s<t$.  First of all, by the Duhamel formula
\begin{align}\label{a}
    \|q(t)-q(s)\|_{H^{-1}(\R)}\lesssim &\left\|e^{4t\kappa^2 \partial_x}q(0)-e^{4s\kappa^2 \partial_x}q(0)\right\|_{H^{-1}(\R)} \notag\\
    &+\int_0^s\left \|e^{4(t-\tau)\kappa^2 \partial_x} g'(q(\tau))-e^{4(s-\tau)\kappa^2 \partial_x} g'(q(\tau))\right\|_{H^{-1}(\R)} d\tau \notag\\
    &+ \int_s^t \left\|g'(q(\tau))\right\|_{H^{-1}(\R)} d\tau.
\end{align}
Starting from the last term, the diffeomorphism property and the estimates of Theorem \ref{Hk wp} give
\begin{align*}
    \int_s^t \|g'(q(\tau))\|_{H^{-1}(\R)}d\tau \lesssim \int_s^t \|q(\tau)\|_{H^{-1}(\R)}d\tau\lesssim \int_s^t\|q(0)\|_{H^{-1}(\R)} e^{c\tau} d\tau \lesssim |t-s|.
\end{align*}

Next, we observe that for any bounded equicontinuous subset $F\subset H^{-1}(\mathbb R)$, \eqref{linear} implies that
\begin{align}
    \label{eq} 
    \left\|e^{4t\kappa^2 \partial_x}f-e^{4s\kappa^2 \partial_x}f\right\|_{H^{-1}(\R)}&= \left\|f(x+4t\kappa^2) -f(x+4s\kappa^2 )\right\|_{H^{-1}(\R)}\notag\\
    &= \left\|f(x+4(t-s)\kappa^2) -f(x)\right\|_{H^{-1}(\R)}\to 0
\end{align}
as $|t-s|\to 0$, uniformly for all $f\in F$, by the definition of equicontinuity.

We can now apply this estimate for the first term of (\ref{a}) to get $\delta_Q>0$ so that for all $q(0)\in Q$ and all $t,s\in[0,T]$ with $|t-s|<\delta_Q$
\begin{align*}
    \left\|e^{4t\kappa^2 \partial_x}q(0)-e^{4s\kappa^2 \partial_x}q(0)\right\|_{H^{-1}(\R)}<\frac{\varepsilon}{3}.
\end{align*}
Moreover, the set 
\begin{align*}
    Q'=\{g'(q(t)): t\in [0,T],\, q\,\,\text{solution to \eqref{eq:Hk} with initial data}\,\,q(0)\in Q\}
\end{align*}
is bounded in $\dot H^{-\frac 12}(\mathbb R)$, so it is equicontinuous as well, with its equicontinuity properties depending on $Q$. Then \eqref{eq} yields for the second term of (\ref{a}) that there exists $\delta_{Q'}=\delta_Q'>0$
so that 
\begin{align*}
    \int_0^s \left\|e^{4(t-\tau)\kappa^2 \partial_x} g'(q(\tau))-e^{4(s-\tau)\kappa^2 \partial_x} g'(q(\tau))\right\|_{H^{-1}(\R)} d\tau <\frac{\varepsilon}{3}
\end{align*}
for all $q(0)\in Q$ and $t,s\in[0,T]$ with $|t-s|<\delta_{Q}'$.

Combining all the above, we conclude that we can find $\delta>0$ depending only on $\varepsilon$, $T$ and $Q$ so that for all $q(0)\in Q$ and $t,s\in[0,T]$ with $|t-s|<\delta$
\begin{align*}
  \|q(t)-q(s)\|_{H^{-1}(\R)}<\varepsilon.
\end{align*}
\end{proof}

\begin{theorem}
\label{Hk weak}
Fix $T>0$ and $Q\subset B_{\delta_0}$ equicontinuous in $H^{-1}$. Let $\{q_n(0)\}$ be a sequence in $Q$ and $q_n$ the solution to \eqref{eq:Hk} with initial data $q_n(0)$. Then:\\
1. Passing to a subsequence, $q_{n}(t)$ converges weakly in $H^{-1}(\R)$ to some $q(t)\in H^{-1}(\R)$ uniformly for all $t\in[0,T]$.\\
2. For the same subsequence, $g(q_{n}(t))-\frac{1}{2\kappa}$ converges weakly in $H^{1}(\R)$ to $g(q(t))-\frac{1}{2\kappa}$ uniformly for all $t\in[0,T]$.\\
3. $q$ is a solution to \eqref{eq:Hk} on $[0,T]$.
\end{theorem}

\begin{proof}
In the following, we allow the implicit constants to depend on $T$ and $Q$.\\
1. Since $\|q_n(0)\|_{H^{-1}(\R)}\leq \delta_0$ for all $n$, by \eqref{H-1 conservation}
\begin{align*}
    \|q_n(t)\|_{ H^{-1}(\R)}\lesssim \delta_0\quad\text{for all}\quad t\in[0,T], \,n\in \mathbb N.
\end{align*}
We consider $[0,T]\cap \mathbb Q=\{t_j:j\in \mathbb N\}$. Then, for each $j\in\mathbb N$, $\|q_n(t_j)\|_{H^{-1}(\R)}$ is bounded so we can find a subsequence that converges weakly in $H^{-1}(\R)$ to some $q(t_j)\in H^{-1}(\R)$. By a diagonal argument, we obtain a subsequence $\{q_{n_k}\}$ such that $q_{n_k}(t_j)$ converges weakly in $H^{-1}(\R)$ to $q(t_j)$ for all $j\in\mathbb N$. For the sake of convenience, we relabel the subsequence $q_n$. 

We now define for every $t\in[0,T]$ $q(t)$ to be the $H^{-1}$ weak limit of $q(\tau_m)$ for $\{\tau_m\}\subset [0,T]\cap \mathbb Q$ such that $\tau_m\to t$. First of all, we need to establish that this is well-defined. For every such sequence $\{\tau_m\}$ there exists a subsequence $\{\tau_{m_k}\}$ for which $q(\tau_{m_k})$ converges weakly in $H^{-1}$, since
\begin{align*}
    \|q(\tau)\|_{H^{-1}(\R)} \lesssim R\quad\text{for all}\quad \tau\in[0,T]\cap \mathbb{Q}.
\end{align*}
Moreover, for any two sequences $\{\tau_m\}, \{\tau_m'\}$ in $[0,T]\cap\mathbb{Q}$ such that $\tau_m\to t$, $\tau_m'\to t$ and $\{q(\tau_m)\}$, $\{q(\tau_m')\}$ converge weakly in $H^{-1}(\R)$, we get that for any $\phi\in H^{1}(\R)$ with $\|\phi\|_{H^{1}(\R)}=1$
\begin{align}
\label{b}
    |\langle q(\tau_m)-q(\tau_m'), \phi \rangle|\leq  |\langle q(\tau_m)-q_n(\tau_m), \phi \rangle|+ &|\langle q_n(\tau_m)-q_n(\tau_m'), \phi \rangle| \notag\\
    &+ |\langle q_n(\tau_m')-q(\tau_m'), \phi \rangle|
\end{align}
for all $m, n\in \mathbb N$. Let $\varepsilon>0$. For all $n\in\mathbb N$, by Lemma \ref{Hk eqcts}
\begin{align*} 
    |\langle q_n(\tau_m)-q_n(\tau_m'), \phi \rangle|\leq \|q_n(\tau_m)-q_n(\tau_m')\|_{H^{-1}(\R)} <\varepsilon
\end{align*}
for $m$ large (independently of $n$) so that $|\tau_m-\tau_m'|$ is sufficiently small. Then, for fixed $m$, the first and last terms of (\ref{b}) are small for $n$ sufficiently large by the definition of $q$ at rational times. Therefore, for $m$ sufficiently large we get that
\begin{align*}
    |\langle q(\tau_m)-q(\tau_m'), \phi \rangle|< 3\varepsilon,
\end{align*}
so the $H^{-1}$ weak limits of $q(\tau_m)$ and $q(\tau_m')$ are the same. The above two points ensure that for every $t\in[0,T]$, for any sequence $\{\tau_m\}$ in $[0,T]\cap\mathbb{Q}$ such that $\tau_m\to t$, 
the sequence $q(\tau_m)$ converges weakly in $H^{-1}(\R)$ and the limit does not depend on the choice of $\{\tau_m\}$. 

Next, we will show that $q_n(t)$ converges weakly in $H^{-1}(\R)$ to $q(t)$ uniformly for all $t\in[0,T]$. Let $\varepsilon>0$ and fix $\phi\in H^{1}(\R)$ with $\|\phi\|_{H^{1}(\R)}=1$. By definition $q$ is uniformly continuous in time on $[0,T]$ (with respect to the weak $H^{-1}$ topology) and by Lemma \ref{Hk eqcts} $\{q_n\}$ is equicontinuous in time on the same interval, so there exists $\delta>0$ such that for all $t,s\in [0,T]$ satisfying $|t-s|<\delta$
\begin{align}
\label{eq:wwp1}
    |\langle q(t)-q(s), \phi \rangle|<\varepsilon
\end{align}
and
\begin{align}
\label{eq:wwp2}
    |\langle q_n(t)-q_n(s), \phi \rangle|<\varepsilon\quad\text{for all}\quad n\in\mathbb{N}.
\end{align}
The interval $[0,T]$ is compact, so we can find $\tau_1,\dots,\tau_M\in [0,T]\cap\mathbb{Q}$ so that $[0,T]\subset \bigcup_{m=1}^M (\tau_m-\delta, \tau_m+\delta)$. We can also find $N\in \mathbb{N}$ large enough so that for all $1\leq m\leq M$  
\begin{align}
\label{eq:wwp3}
    |\langle q_n(\tau_m)-q(\tau_m), \phi \rangle|<\varepsilon\quad\text{for all}\quad  n>N.
\end{align}
Let $t\in [0,T]$ and $\tau_{m_0}\in\{\tau_1,\dots,\tau_M\}$ such that $|t-\tau_m|<\delta$. Then by (\ref{eq:wwp1}), (\ref{eq:wwp2}), and (\ref{eq:wwp3}), for all $n>N$
\begin{align*}
    |\langle q_n(t)-q(t), \phi \rangle|&\leq |\langle q_n(t)-q_n(\tau_{m_0}), \phi \rangle|+ |\langle q_n(\tau_{m_0})-q(\tau_{m_0}), \phi \rangle|\\
    &+ |\langle q(\tau_{m_0})-q(t), \phi \rangle|\\
    & <3\varepsilon.
\end{align*}
Hence $q_n(t)$ converges weakly in $H^{-1}(\R)$ to $q(t)$ uniformly for all $t\in [0,T]$.

2. Fix $\phi\in C_c^\infty(\R)\cap H^{-1}(\R)$ with $\|\phi\|_{H^{-1}(\mathbb R)}=1$, $M>0$ so that $\supp(\phi)\subset \{x:|x|\leq M\}$, and let $\varepsilon>0$.  We want to show that for $n$ sufficiently large 
\begin{align*}
    \left|\int \phi(x)[g(q_n(t))-g(q(t))] dx\right| <\varepsilon\quad\text{for all}\quad t\in[0,T].
\end{align*}
Let $\psi\in C_c^\infty(\R)$ be a smooth cutoff such that $\psi=1$ on $\{|x|\leq M\}$ and $\psi=0$ on $\{|x|>M+\varepsilon^{-2}\}$ in order to ensure that $\|\psi'\|_{L^2}+\|\psi''\|_{L^2}\lesssim \varepsilon$. 

From the definition of $g$ \eqref{series}, using the notation $r(T)= \sqrt{R_0} T \sqrt{R_0}$ as in section \ref{sec;fda}, and by applying repeatedly the identity \eqref{move cutoff} we have
\begin{align*}
    &\int \phi(x)[g(q_n(t))-g(q(t))] dx =\int \phi(x) \psi(x)[g(q_n(t))-g(q(t))]  dx\\
    &=\sum_{l\geq 1} (-1)^l \tr\left\{ \sqrt{R_0} \phi\psi \sqrt{R_0}(\sqrt{R_0} q_n \sqrt{R_0})^l \right\}\\
    &- \sum_{l\geq 1} (-1)^l \tr\left\{ \sqrt{R_0} \phi\psi \sqrt{R_0}(\sqrt{R_0} q \sqrt{R_0})^l \right\}\\
    &=\sum_{l\geq 1} (-1)^l \left[\tr\left\{ r(\phi\psi) r(q_n )^l \right\}-  \tr\left\{ r(\phi\psi) r(q)^l \right\}\right]\\
    &=\sum_{l\geq 1}(-1)^l \left[\tr\left\{ r(\phi)r(\psi q_n) r(q_n)^{l-1} \right\} - \tr\left\{ r(\phi)r(\psi q) r(q)^{l-1} \right\}\right]\\
    &+ \sum_{l\geq 1}(-1)^l \left[ \tr\left\{ r(\phi) A(\psi) r(q_n)^{l} \right\} - \tr\left\{ r(\phi) A(\psi) r(q)^{l} \right\}\right]\\
    &= \sum_{k\geq 0} \sum_{m\geq 0} (-1)^{k+m+1} \tr\left\{ r(\phi)r( q )^{k} r(\psi (q_n-q) )r( q_n )^{m} \right\}\\
    &+ \sum_{l\geq 0}\sum_{k\geq 0} \sum_{m\geq 0} (-1)^{l+k+m+1} \tr\left\{ r( \phi)r( q )^{l} A(\psi) r( q )^{k} r( q_n-q) r( q_n)^{m} \right\}.
\end{align*}
Then \eqref{op norm} allows us to estimate
\begin{align*}
    &\left|\int  \phi(x)[g(q_n(t))  -g(q(t))] dx\right|\\
    &\lesssim\sum_{k\geq 0} \sum_{m\geq 0} \|\phi\|_{H^{-1}}\|q\|_{H^{-1}}^{k} \|\psi (q_n-q)\|_{H^{-1}}\| q_n \|_{H^{-1}}^{m} \\
    &+ \sum_{l\geq 0}\sum_{k\geq 0} \sum_{m\geq 0} \|\phi\|_{H^{-1}}\|q\|_{H^{-1}}^{l} \varepsilon (\|q\|_{H^{-1}})^{k} \|q_n-q\|_{H^{-1}}\| q_n \|_{H^{-1}}^{m}\\
    &\lesssim \varepsilon
\end{align*}
for $n$ sufficiently large so that $\|\psi (q_n(t)-q(t))\|_{H^{-1}}<\varepsilon$ for all $t\in [0,T]$.

3. We start with the equation given by the Duhamel formula for $q_n$:
\begin{align}
\label{eq:d}
    q_n(t)=e^{4t\kappa^2 \partial_x}q_n(0)+\int_0^t 16\kappa^5 e^{4(t-s)\kappa^2 \partial_x} g'(q_n(s)) ds.
\end{align}

Fix $\phi\in C_c^\infty(\R)$ with $\|\phi\|_{H^{1}(\mathbb R)}=1$, $M>0$ so that $\supp(\phi)\subset \{x:|x|\leq M\}$. By \eqref{linear} for every $t\in[-T,T]$
\begin{align*}
    \supp(e^{4t\kappa^2 \partial_x} \phi)\subset \{x:|x|\leq M+ 4\kappa^2 T\}.
\end{align*}
We consider a smooth cutoff $\psi\in C_c^\infty$ supported on $\{x:|x|\leq M+ 5\kappa^2 T\}$ and equal to $1$ inside $\{x:|x|\leq M+ 4\kappa^2 T\}$. Then, for any sequence $\{f_n\}$ in $H^{-1}(\R)$ that converges weakly in $H^{-1}(\R)$ to $f\in H^{-1}(\mathbb R)$, we get that 
\begin{align*}
    \langle e^{4t\kappa^2 \partial_x} (f_n-f), \phi\rangle&= \langle  f_n-f, e^{-4t\kappa^2 \partial_x} \phi\rangle\\
    &=\langle  (f_n-f) \psi, e^{-4t\kappa^2 \partial_x} \phi\rangle
\end{align*}
which converges to $0$ uniformly for all $t\in[0,T]$. This ensures that 
\begin{align}
\label{eq:wl1}
    e^{4t\kappa^2 \partial_x}q_n(0) \rightharpoonup e^{4t\kappa^2 \partial_x}q(0)\quad\text{weakly in}\,\, H^{-1}.
\end{align}
For the second term of (\ref{eq:d}) we observe that, since $$q_n(s)\rightharpoonup q(s)\quad\text{weakly in} \,\, H^{-1}\quad\text{uniformly for all}\,\, s\in [0,T],$$
by part (2) 
$$g(q_n(s))-\frac{1}{2\kappa}\rightharpoonup g(q(s))-\frac{1}{2\kappa}\quad\text{weakly in} \,\, H^{1}\quad\text{uniformly for all}\,\, s\in [0,T],$$
so
$$g'(q_n(s))\rightharpoonup g'(q(s))\quad\text{weakly in} \,\, H^{-1}\quad\text{uniformly for all}\,\, s\in [0,T].$$
Now, taking advantage of \eqref{eq:wl1} again,
$$e^{4t\kappa^2 \partial_x}g'(q_n(s))\rightharpoonup e^{4t\kappa^2 \partial_x}g'(q(s))\quad\text{weakly in} \,\, H^{-1}\quad\text{uniformly for all}\,\, s,t\in [0,T].$$
This means that for $G_n(t;s):=e^{4(t-s)\kappa^2 \partial_x}[g'(q_n(s))-g'(q(s))]$, for any given $\varepsilon>0$
\begin{align*}
    \left|\langle G_n(t;s), \phi\rangle\right|<\varepsilon \quad\text{for all}\,\,s,t\in [0,T]
\end{align*}
for $n$ sufficiently large (independently of $s,t$), hence
\begin{align*}
    \left|\langle \int_0^t G_n(t;s) ds, \phi\rangle\right|\leq \int_0^t\left|\langle G_n(t;s), \phi\rangle\right|ds <\varepsilon T\quad\text{for all}\,\,t\in [0,T].
\end{align*}
We conclude that
\begin{align}
\label{eq:wl2}
    \int_0^t  e^{4(t-s)\kappa^2 \partial_x} g'(q_n(s)) ds \rightharpoonup \int_0^t  e^{4(t-s)\kappa^2 \partial_x} g'(q(s)) ds\quad\text{weakly in}\,\, H^{-1}
\end{align}
uniformly for all $t\in[0,T]$.

Returning to (\ref{eq:d}), we already know that the left hand side converges weakly in $H^{-1}$ to $q(t)$ for all $t\in[0,T]$, and by (\ref{eq:wl1}) and (\ref{eq:wl2}) the right hand side converges weakly in $H^{-1}$ to $e^{4t\kappa^2 \partial_x}q(0)+\int_0^t 16\kappa^5 e^{4(t-s)\kappa^2 \partial_x} g'(q(s)) ds$ for all $t\in[0,T]$, so
\begin{align*}
    q(t)=e^{4t\kappa^2 \partial_x}q(0)+\int_0^t 16\kappa^5 e^{4(t-s)\kappa^2 \partial_x} g'(q(s)) ds\quad\text{for all}\,\, t\in[0,T],
\end{align*}
proving that $q$ is a solution to \ref{eq:Hk}, as desired.
\end{proof}

\section{Symplectic non-squeezing for the \ref{eq:Hk} flow on the line} \label{sec;sns Hk}

We are finally in a position to state and prove a symplectic non-squeezing theorem for \eqref{eq:Hk}. Note that it can only encompass small initial data; after all, the nonlinearity only makes sense for sufficiently small solutions.

\begin{theorem}
\label{sns Hk R}
Fix $\kappa\geq 1$. Let $z\in \dot H^{-\frac 1 2}(\mathbb R)$ with $\|z\|_{\dot H^{-\frac 12}(\mathbb R)}\leq \frac{\delta_0}{10}$, $l\in \dot H^{\frac 1 2} (\mathbb R)$ with $\|l\|_{\dot H^{\frac 1 2}(\mathbb R)}=1$, $\alpha\in \mathbb C$, $0<r<R<\frac{\delta_0}{10}$, and $T>0$. Then there exists $q_{\kappa,0} \in \{ q\in \dot H^{-\frac 1 2}(\mathbb R) : \|q-z\|_{\dot H^{-\frac 1 2}(\mathbb R)}< R\}$ such that the solution $q_\kappa$ to \ref{eq:Hk} with initial data $q_\kappa(0)=q_{\kappa,0}$ satisfies 
$$|\langle l, q_\kappa(T)\rangle -\alpha|>r.$$
\end{theorem} 

\begin{proof}
We fix $0<\delta<\frac{R-r}{100}<\frac{\delta_0}{1000}$. Let $A:=\|z\|_{\dot H^{-\frac 1 2} (\mathbb R)}+R<\frac{\delta_0}{5}$. Let $m_n, M_n, N_n, L_n$ be sequences such that $m_n\to 0$ and $M_n, N_n, L_n \to \infty$ as $n\to \infty$, as in Section \ref{sec;fda}. First of all, we can find $\tilde z\in \dot H^{-\frac 1 2}\cap C_c ^\infty (\R)$ and $\tilde l\in C_c^\infty (\R)$ such that
\begin{align*}
    \|z-\tilde z\|_{\dot H^{-\frac 1 2} (\mathbb R)}<\delta,\quad \|l-\tilde l\|_{\dot H^{\frac 1 2} (\mathbb R)}<\tilde \delta \quad\text{and} \quad \|\tilde l\|_{\dot H^{\frac 1 2} (\mathbb R)}=1,
\end{align*}
where $\tilde \delta>0$ is a small parameter to be chosen later.

We will only consider $n$ sufficiently large so that $\supp(\tilde z), \supp (\tilde l)\subset I_n:= [-\frac {L_n}{2}, \frac{L_n}{2}]$. We define 
\begin{align*}
    z_n(x)&:=\mathring{\tilde z}_{L_n}(x)=\sum_{j\in \mathbb Z} \tilde z(x+j L_n ),\\
    l_n(x)&:=\mathring{\tilde l}_{L_n}(x)=\sum_{j\in \mathbb Z} \tilde l(x+j L_n ).
\end{align*}
By Lemma \ref{periodization}, for all $n$ we have that $z_n, l_n\in \dot H^k(\T_n)$ for all integers $k\geq 0$ with
\begin{align}
\label{eq:norm1}
    \|z_n\|_{\dot H^{k} (\mathbb T_n)}=\|\tilde z\|_{\dot H^{k} (\mathbb R)}, \qquad \|l_n\|_{\dot H^{k} (\mathbb T_n)}=\|\tilde l\|_{\dot H^{k} (\mathbb R)}.
\end{align}
In addition, for $n$ sufficiently large, $z_n \in \dot H^{-\frac 1 2} (\mathbb T_n)$ and $l_n \in \dot H^{\frac 1 2} (\mathbb T_n)$ with
\begin{align}
\label{eq:norm2}
    \lim_{n\to\infty} \|z_n\|_{\dot H^{-\frac 1 2} (\mathbb T_n)}=\|\tilde z\|_{\dot H^{-\frac 1 2} (\mathbb R)},\qquad \lim_{n\to\infty} \|l_n\|_{\dot H^{\frac 1 2} (\mathbb T_n)}=\|\tilde l\|_{\dot H^{\frac 1 2} (\mathbb R)}=1.
\end{align} 
From now on, we also require $n$ to be sufficiently large so that $z_n \in \dot H^{-\frac 1 2} (\mathbb T_n)$, $l_n \in \dot H^{\frac 1 2} (\mathbb T_n)$ and 
\begin{align*}
    \|z_n\|_{\dot H^{-\frac 1 2} (\mathbb T_n)}&\leq \|z\|_{\dot H^{-\frac 1 2} (\mathbb R)}+\delta,\\
    \|l_n\|_{\dot H^{\frac 1 2} (\mathbb T_n)}&\leq 2.
\end{align*}

We also define
\begin{align*}
    \zeta_n&:=P^{L_n}_{m_n<\dots\leq M_n} z_n,\\
    \lambda_n&:= \frac{1}{\|P^{L_n}_{m_n<\dots\leq M_n} l_n\|_{\dot H^{\frac 1 2} (\mathbb T_n)}} P^{L_n}_{m_n<\dots\leq M_n} l_n.
\end{align*}
By Lemma \ref{norm3}, $\zeta_n\in \dot H^{-\frac 1 2}(\mathbb T_n)$, $\lambda_n\in \dot H^{\frac 1 2} (\mathbb T_n)$ with $\|\lambda_n\|_{\dot H^{\frac 1 2}(\T_n)}=1$, and
\begin{align}\label{0}
    \lim_{n\to\infty}\|\zeta_n-z_n\|_{L^2 (\mathbb T_n)}&=0,\\
    \lim_{n\to\infty}\|\lambda_n-l_n\|_{\dot H^{\frac 1 2} (\mathbb T_n)}&=0.
\end{align}
In addition,
\begin{align*}
    \|\zeta_n\|_{\dot H^{-\frac 1 2}(\mathbb T_n)}\leq \|z_n\|_{\dot H^{-\frac 1 2}(\mathbb T_n)}\leq \|z\|_{\dot H^{-\frac 1 2}(\mathbb R)}+ \delta\leq \frac{\delta_0}{5}.
\end{align*}

For every $n$ sufficiently large, we consider the equation \eqref{eq:Hkn} on $\T_n$ with initial data in $\mathcal H_n$. Theorem \ref{Hkn wp} guarantees the existence of a unique solution up to time $T_n$ for any initial data in $\mathcal H_n\cap\{f\in\dot H^{-\frac 1 2}(\T_n):\kappa^{-\frac 12}\|f\|_{\dot H^{-\frac 12}(\mathbb T_n)}\leq \delta_0\}$. We can apply Gromov's result for these finite dimensional Hamiltonian systems for $n$ large enough so that $T<T_n$ with the new parameters $\zeta_n$ and $\lambda_n$; we get witnesses $u_n$ that solve \eqref{eq:Hkn} with initial data $u_n(0)=u_{n,0}\in\mathcal H_n$ and satisfy
\begin{align}\label{1}
    \|u_{n,0}-\zeta_n\|_{\dot H^{-\frac 1 2}(\T_n)}<R-10\delta, \qquad |\langle \lambda_n, u_n(T) \rangle -\alpha| >r+10\delta.
\end{align}
It is clear that
\begin{align*}
    \|u_{n,0}\|_{\dot H^{-\frac 1 2}(\T_n)}\leq A.
\end{align*}
Moreover, by Theorem \ref{Hkn wp}, we know that for all $n$, $u_n(t)\in \dot H^{-\frac 1 2}(\T_n)$ and 
\begin{align*}
    \|u_n(t)\|_{\dot H^{-\frac 1 2}(\T_n)}\lesssim \|u_{n,0}\|_{\dot H^{-\frac 1 2}(\T_n)}\quad\text{for all}\quad t\in[0,T],
\end{align*}
hence
\begin{align}\label{2}
    \|u_n(t)\|_{\dot H^{-\frac 1 2}(\T_n)}\lesssim A\quad\text{for all}\quad t\in[0,T].
\end{align}
In addition, for every $t\in [0,T]$, the Fourier transform of $u_n(t)$ is supported in $\{\xi:\frac 1 2 m_n\leq |\xi|\leq 4 M_n\}$. 

Next, we consider $q_{n,0}:=\chi_n^0 u_{n,0}$, provided by Theorem \ref{fda} and constructed as described in section \ref{sec;fda}. We have already seen in Lemma \ref{H-1/2 bound} that $\|q_{n,0}\|_{\dot H^{-\frac 12}(\mathbb R)}\lesssim A$ uniformly for all large $n$. Thus, passing to a subsequence, $q_{n,0}$ converges weakly in $\dot H^{-\frac 1 2}(\R)$ to some $q_0\in \dot H^{-\frac 1 2}(\R)$. Also $q_{n,0}$ converges weakly in $H^{-1}(\R)$ to $q_0\in H^{-1}(\R)$. In fact, we can show that
\begin{align}
\label{eq:ball}
\|q_{0}-\tilde z\|_{\dot H^{-\frac 1 2}(\R)}<R-4\delta.    
\end{align}
Let $\phi \in \dot H^{\frac 12}\cap C_c^\infty(\mathbb R)$ with $\|\phi\|_{\dot H^{\frac 12}(\mathbb R)}=1$. We only consider $n$ large enough so that $\supp \phi \subset[-\frac{L_n}{2}, \frac{L_n}{2}]$. We observe that for any $n$
\begin{align*}
    \left| \int_{\mathbb R}(q_0-\tilde z)\phi \right|&\leq \left| \int_{\mathbb R}(q_0-q_{n,0})\phi \right| + \left| \int_{\mathbb R}(q_{n,0}-u_{n,0})\phi \right|  + \left| \int_{\mathbb R}(\zeta_n- z_n)\phi \right|\\
    &+ \left| \int_{\mathbb R}(z_n-\tilde z)\phi \right| + \left| \int_{\mathbb R}(u_{n,0}-\zeta_n)\phi \right|.
\end{align*}
By working on each term separately, we verify that the first four terms are bounded by $\delta$ for $n$ large enough, depending on $\phi$. The last term is bounded by $R-9\delta$ provided $n$ is sufficiently large, again depending on $\phi$, thus completing the proof of \eqref{eq:ball}.

This also implies that $\|q_0-z\|_{\dot H^{-\frac 12}(\mathbb R)}<R$, thus $\|q_0\|_{\dot H^{-\frac 12}(\mathbb R)}\leq A $. Additionally, by construction we get \eqref{eq:H-1 bound}, so $\|q_{n,0}\|_{H^{-1}(\mathbb R)}<\frac{\delta_0}{2}$. Theorem \ref{Hk wp} then guarantees the existence of a unique global solution $q_n$ to \ref{eq:Hk} with initial data $q_{n,0}$ and $q$ solution to \ref{eq:Hk} with initial data $q_{0}$. Furthermore, 
\begin{align*}
    \|q_n(T)\|_{\dot H^{-\frac 1 2}(\R)} \lesssim \|q_{n,0}\|_{\dot H^{-\frac 1 2}(\R)} \lesssim A.
\end{align*}
Then $q_{n}(T)$ has a subsequence that converges weakly in $\dot H^{-\frac 1 2}(\R)$.
By Theorem \ref{Hk weak}, passing to a further subsequence, $q_{n}(T)$ converges weakly in $H^{-1}(\R)$ to $q(T)$. Uniqueness of limit asserts that $q(T)\in \dot H^{-\frac 12}(\mathbb R)$ and a subsequence of $q_{n}(T)$ converges weakly in $\dot H^{-\frac 1 2}(\R)$ to $q(T)$. By \eqref{eq:finite speed cor} we can also see that, passing to a subsequence, $f_n(T)= \chi_n^\ast q_n(T)$ converges weakly in $\dot H^{-\frac 1 2}(\R)$ to $q(T)$.

Furthermore, for $n$ sufficiently large, by \eqref{1}, \eqref{2}, \eqref{0}
\begin{align*}
    |\langle l_n, u_n(T) \rangle-\alpha| &\geq |\langle \lambda_n, u_n(T) \rangle-\alpha| - |\langle \lambda_n-l_n, u_n(T) \rangle|\\
    &>r+10 \delta - \|\lambda_n-l_n\|_{\dot H^{\frac 1 2}(\T_n)} \|u_n(T)\|_{\dot H^{-\frac 1 2}(\T_n)}\\
    &>r+9\delta
\end{align*}
and 
\begin{align*}
    |\langle \tilde l, q(T) \rangle-\alpha| &\geq |\langle l_n, u_n(T) \rangle-\alpha| -|\langle l_n, u_n(T)-\mathring f_n (T)\rangle|- |\langle \tilde l, f_n (T)-q(T)\rangle|\\
    &>r+9\delta -\|\tilde l\|_{H^{1}(\R)}\|u_n(T)-\mathring f_n (T)\|_{H^{-1}(\T_n)}- \delta\\
    &> r+7 \delta
\end{align*}
for $n$ sufficiently large, depending on $\tilde l$.

Finally, by \eqref{H-1/2 conservation}
\begin{align*}
    |\langle l-\tilde l, q(T)\rangle| \leq \|l-\tilde l\|_{\dot H^{\frac 1 2}(\R)} \|q(T)\|_{\dot H^{-\frac 1 2}(\R)}\lesssim \tilde \delta A.
\end{align*}
Making sure we have chosen $\tilde \delta$ small enough so that $|\langle l-\tilde l, q(T)\rangle|<\delta$, we conclude that
\begin{align*}
    |\langle l, q(T) \rangle-\alpha|>r.
\end{align*}

\end{proof}

\section{Symplectic non-squeezing for the \ref{KdV} flow on the line}  \label{sec;sns KdV}

The goal of this section is to prove Theorem \ref{sns KdV R}. This is accomplished by exploiting the fact that the $H_\kappa$ flows are a good approximation to the KdV flow, an idea introduced in \cite{KV18}. The key property that this approximation result relies on is equicontinuity. We have already defined what it means for a set to be equicontinuous in $H^{-1}(\mathbb R)$ and have spotlighted a certain class of bounded subsets of $H^{-1}$ that have this property. Let us recall some further results about equicontinuous subsets and the flows in question from \cite{KV18}, including the precise statement of the crucial approximation of \eqref{KdV} by \eqref{eq:Hk}. 

\begin{lemma}\label{Q2}
If $Q$ is a set of Schwartz functions that is equicontinuous in $H^{-1}(\mathbb R)$, then so is 
\begin{align*}
    Q^\ast:= \{e^{J\nabla(t H_{KdV}+s H_\kappa)}q: q\in Q, t,s\in\mathbb R, \kappa\geq 1\}.
\end{align*}
\end{lemma}

\begin{lemma}\label{Q3}
If $Q^\ast$ is equicontinuous in $H^{-1}(\mathbb R)$, then
\begin{align}
    \label{main}
    \lim_{\kappa\to\infty}\sup_{q\in Q^\ast} \sup_{|t|\leq T} \|e^{tJ\nabla(H_{KdV}- H_\kappa)}q-q \|_{H^{-1}(\mathbb R)}=0.
\end{align}
\end{lemma}

\begin{remark}
In \cite{KV18} the authors show \eqref{main} for a special set $Q$, namely where the elements of $Q$ form a sequence of Schwartz functions that converges in $H^{-1}(\mathbb R)$, as part of the proof of Theorem 5.1. However, the only property of $Q$ that is needed in the proof is the equicontinuity of $Q^\ast$ in $H^{-1}(\mathbb R)$.
\end{remark}

Lemma \ref{Q3} was the decisive result in the proof of the following well-posedness theorem for \eqref{KdV} in \cite{KV18}.

\begin{theorem}\label{KdV wp}
The \eqref{KdV} equation is globally well-posed in $H^{-1}(\mathbb R)$ in the following sense: The solution map extends (uniquely) from Schwartz space to a jointly continuous map
\begin{align*}
    \Phi:\mathbb R\times H^{-1}(\mathbb R)\to H^{-1}(\mathbb R).
\end{align*}
In particular, $\Phi$ has the group property: $\Phi(t + s) = \Phi(t) \circ \Phi(s)$. Moreover, each orbit $\{\Phi(t, q) : t \in \mathbb R\}$ is bounded and equicontinuous in $H^{-1}(\mathbb R)$. Concretely, 
\begin{align*}
    \sup_{t}\|q(t)\|_{H^{-1}(\mathbb R)}\lesssim  \|q(0)\|_{H^{-1}(\mathbb R)}+ \|q(0)\|_{H^{-1}(\mathbb R)}^3.
\end{align*}
If in addition $q(0)\in H^{-\frac 12}(\mathbb R)$, then
\begin{align*}
     \sup_{t}\|q(t)\|_{H^{-\frac 12}(\mathbb R)}\lesssim  \|q(0)\|_{H^{-\frac 12}(\mathbb R)}+ \|q(0)\|_{H^{-\frac 12}(\mathbb R)}^3.
\end{align*}
\end{theorem}

The results above combined with our work on Section \ref{sec;weak wp} give us access to the following theorem on the preservation of weak limits under the KdV flow. As we have explained earlier, working in the weak topology is one of the essential difficulties in the infinite volume regime. Apart from being a interesting result on its own, Theorem \ref{KdV weak} will be the last ingredient for the proof of Theorem \ref{sns KdV R}. 

\begin{theorem}
\label{KdV weak}
Fix $T>0$ and $Q$ a bounded equicontinuous subset of $H^{-1}$. Let $\{q_n(0)\}$ be a sequence of Schwartz functions in $Q$ and $q_n$ the solution to \eqref{KdV} with initial data $q_n(0)$. Suppose that
\begin{align*}
  q_{n}(0)\rightharpoonup q(0)\qquad \text{as}\,\, n\to \infty\,\, \text{weakly in}\,\, H^{-1}(\mathbb R).  
\end{align*}
Then, if $q$ is the solution to \eqref{KdV} with initial data $q(0)$, passing to a subsequence, 
\begin{align*}
  q_{n}(T)\rightharpoonup q(T)\qquad \text{as}\,\, n\to \infty\,\, \text{weakly in}\,\, H^{-1}(\mathbb R).  
\end{align*}
\end{theorem}

\begin{proof}
Suppose $Q\subset \{f\in H^{-1}(\mathbb R):\|f\|_{H^{-1}(\mathbb R)}<A\}$. First, we consider the case when $0<A<\frac{\delta_0}{2}$. This additional assumption ensures that the elements of the sequence $q_{n}(0)$ give rise to unique global solutions under the $H_\kappa$ flow.

Since $\|q_n(0)\|_{H^{-1}(\mathbb R)}<A$ for all $n$, $q(0)\in H^{-1}(\mathbb R)$ and $\|q(0)\|_{H^{-1}(\mathbb R)}\leq A$. Moreover, we can find Schwartz functions $f_n(0)\in H^{-1}(\mathbb R)$ so that $F=\{f_n(0)\}\subset \{f\in H^{-1}(\mathbb R):\|f\|_{H^{-1}(\mathbb R)}<2A\}$ is equicontinuous in $H^{-1}(\mathbb R)$ and $f_n(0)\to q(0)$ in $H^{-1}(\mathbb R)$.
Then the set $Q\cup F$ is equicontinuous in $H^{-1}(\mathbb R)$. Lemma \ref{Q2} and Lemma \ref{Q3} ensure that \eqref{main} holds. 
\begin{align*}
    q_n(T)- q(T)= & \left[e^{TJ\nabla H_{KdV}} q_n(0)-e^{TJ\nabla H_{\kappa}} q_n(0)\right]\\
    &+ \left[e^{TJ\nabla H_{\kappa}} q_n(0)- e^{TJ\nabla H_{\kappa}} q(0)\right]\\
    &+ \left[e^{TJ\nabla H_{\kappa}} q(0)- e^{TJ\nabla H_{\kappa}} f_n(0)\right]\\
    &+ \left[e^{TJ\nabla H_{\kappa}} f_n(0)- e^{TJ\nabla H_{KdV}} f_n(0)\right]\\
    &+ \left[e^{TJ\nabla H_{KdV}} f_n(0)- e^{TJ\nabla H_{KdV}} q(0)\right].
\end{align*}
By \eqref{main}, the first and fourth terms converge to $0$ in $H^{-1}(\mathbb R)$ as $\kappa\to\infty$ for all $n$. Having fixed $\kappa$ sufficiently large, passing to a subsequence the second and third terms converge to $0$ weakly in $H^{-1}(\mathbb R)$ as $n\to\infty$ due to Theorem \ref{Hk weak}. Finally, the last term converges to $0$ in $H^{-1}(\mathbb R)$ by Theorem \ref{KdV wp}. This completes the proof of this special case.

As far as general bounds $A$ are concerned, it can be reduced to the special case by rescaling. Indeed, for $\lambda \in 2^{\mathbb Z}$ we can consider 
\begin{align*}
    q_{n}^{\lambda}(t,x)=\lambda^2 q_n(\lambda^3 t, \lambda x),\quad n\in \mathbb N
\end{align*}
which also solve \eqref{KdV}. In particular, one can observe that for $N\in 2^{\mathbb N}$
\begin{align*}
    \left\|P_{> N}q_n^\lambda(0)\right\|_{H^{-1}(\mathbb R)} & \sim   \left\|P_{> N}q_n^\lambda(0)\right\|_{\dot H^{-1}(\mathbb R)} = \lambda^{\frac 12} \left\|P_{> \frac{N}{\lambda}}q_n(0)\right\|_{\dot H^{-1}(\mathbb R)}\\
    &\sim \lambda^{\frac 12} \left\|P_{> \frac{N}{\lambda}}q_n(0)\right\|_{ H^{-1}(\mathbb R)}
\end{align*}
and
\begin{align*}
    \left\|P_{\leq N}q_n^\lambda(0)\right\|_{H^{-1}(\mathbb R)} & \leq \left\|P_{\leq N}q_n^\lambda(0)\right\|_{L^2(\mathbb R)}= \lambda^{\frac 32} \left\|P_{\leq \frac{N}{\lambda}}q_n(0)\right\|_{L^2(\mathbb R)}\\
    &\lesssim \lambda^{\frac 12} N \left\|P_{\leq \frac{N}{\lambda}}q_n(0)\right\|_{H^{-1}(\mathbb R)}.
\end{align*}
The implication of these observations is twofold: on the one hand, for $N=1$ they ensure that 
\begin{align*}
\|q_n^\lambda(0)\|_{H^{-1}(\mathbb R)}\lesssim\lambda^{\frac 12} \|q_n(0)\|_{H^{-1}(\mathbb R)}\lesssim \lambda^{\frac 12} A
\end{align*}
hence $\|q_n^\lambda(0)\|_{H^{-1}(\mathbb R)}<\frac{\delta_0}{2}$ if $\lambda$ is chosen appropriately small; one the other hand, for fixed $\lambda$ we have that $\left\|P_{> N}q_n^\lambda(0)\right\|_{H^{-1}(\mathbb R)}\to 0$ as $N\to \infty$ uniformly for all $n$, thanks to the equicontinuity of $Q$. We conclude that the rescaled initial data fall under the special case of the Theorem.


Since $ q_{n}^\lambda(0)\rightharpoonup q^\lambda(0)$ as $ n\to \infty$ weakly in $ H^{-1}(\mathbb R)$, the special case of the theorem implies that $ q_{n}^\lambda(\lambda^{-3}T)\rightharpoonup q^\lambda(\lambda^{-3}T)$ as $ n\to \infty$ weakly in $ H^{-1}(\mathbb R)$. This in turn shows that $q_{n}(T)\rightharpoonup q(T)$ as $ n\to \infty$ weakly in $ H^{-1}(\mathbb R)$. 
\end{proof}

\begin{proof} [Proof of Theorem \ref{sns KdV R}]
We are given the fixed parameters $z\in \dot H^{-\frac 1 2}(\mathbb R)$, $l\in H^{\frac 1 2} (\mathbb R)$ with $\|l\|_{\dot H^{\frac 1 2}}=1$, $\alpha\in \mathbb C$, $0<r<R<\infty$, and $T>0$. 

First, we consider the case when $\|z\|_{\dot H^{-\frac 12}(\mathbb R)}<\frac{\delta_0}{10}$ and $R<\frac{\delta_0}{10}$.

We fix $0<\delta< \frac{R-r}{100}$ and take $\tilde l\in C_c^\infty(\mathbb R)$ with $\|\tilde l\|_{\dot H^{\frac 12}(\mathbb R)}=1$ such that
\begin{align*}
     \|l-\tilde l\|_{ H^{\frac 12}(\mathbb R)}<\tilde \delta
\end{align*}
with $\tilde\delta>0$ a small parameter to be chosen later. By Theorem \ref{sns Hk R}, for each $\kappa\geq 1$ we can find $q_\kappa$ solution to the \ref{eq:Hk} flow with initial data $q_{\kappa,0}\in \dot H^{-\frac 12}(\mathbb R)$ satisfying
\begin{align*}
    \|q_{\kappa,0}- z\|_{\dot H^{-\frac 12}(\mathbb R)}<R-10\delta,\qquad \left|\langle \tilde l, q_\kappa(T) \rangle-\alpha \right|>r+10 \delta.
\end{align*}
For each $\kappa\geq 1$ we can find $\tilde q_{\kappa,0}\in \dot H^{-\frac 12}(\mathbb R)\cap C_c^\infty(\mathbb R)$ so that
\begin{align*}
    \|q_{\kappa,0}-\tilde q_{\kappa,0}\|_{\dot H^{-\frac 12}(\mathbb R)}<e^{-\kappa^{10}}.
\end{align*}
We denote by $\tilde q_\kappa$ the solution to \eqref{eq:Hk} with initial data $\tilde q_{\kappa,0}$. Since $\|q_{\kappa,0}\|_{\dot H^{-\frac 12}(\mathbb R)}\leq\|z\|_{\dot H^{-\frac 12}(\mathbb R)}+ R-9\delta$ uniformly in $\kappa$, there exists $q_0 \in \dot H^{-\frac 12}(\mathbb R)$ so that, passing to a subsequence, $q_{\kappa,0}\rightharpoonup q_0$ weakly in $\dot H^{-\frac 12}$ and $H^{-1}$. In particular,
$$\|q_{0}- z\|_{\dot H^{-\frac 12}(\mathbb R)}<R,$$
as for every $\phi\in \dot H^{\frac 12}\cap C_c^\infty(\mathbb R)$ with $\|\phi\|_{\dot H^{\frac 12}(\mathbb R)}=1$
\begin{align*}
    \left| \int_{\mathbb R}(q_0-z)\phi \right|&\leq \left| \int_{\mathbb R}(q_{\kappa,0}- z)\phi \right| + \left| \int_{\mathbb R}(q_{\kappa,0}-q_0)\phi \right|\\
    &\leq \|q_{\kappa,0}- z\|_{\dot H^{-\frac 12}(\mathbb R)} +\delta\\
    &\leq R-9\delta
\end{align*}
for $\kappa$ sufficiently large, depending on $\phi$. What is more, $q_0$ is also the weak $\dot H^{-\frac 12}$ and $H^{-1}$ limit of (a subsequence of) $\tilde q_{\kappa,0}$ as $\kappa\to\infty$.

We consider the sets
\begin{align*}
    Q&= \{\tilde q_{\kappa,0}:\kappa\geq 1\},\\
    Q^\ast&= \{e^{J\nabla(t H_{KdV}+s H_\kappa)}q: q\in Q, t,s\in\mathbb R, \kappa\geq 1\}.
\end{align*}
By Corollary \ref{Qc}, Lemma \ref{Q2} and Lemma \ref{Q3}, the boundedness of $Q$ in $\dot H^{-\frac 12}(\mathbb R)$ ensures that \eqref{main} holds for the solutions we may need it. 

Let $q$ be the solution to \eqref{KdV} with initial data $q_0$. We observe that
\begin{align*}
    q(T)-\tilde q_\kappa(T)= & \left[e^{TJ\nabla H_{KdV}}q_0- e^{TJ\nabla H_{KdV}}\tilde q_{\kappa,0}\right]\\
    &+ \left[e^{TJ\nabla H_{KdV}}\tilde q_{\kappa,0}- e^{TJ\nabla H_{\kappa}}\tilde q_{\kappa,0}\right].
\end{align*}
Passing to a subsequence, one readily sees that by \eqref{main} the second term converges to $0$ strongly in $H^{-1}(\mathbb R)$ as $\kappa\to\infty$ and Theorem \ref{KdV weak} affirms that the first term converges to $0$ weakly in $H^{-1}(\mathbb R)$ as $\kappa\to\infty$, therefore
\begin{align*}
     \tilde q_{\kappa}(T)\rightharpoonup  q(T) \qquad \text{weakly in} \,\, H^{-1}(\mathbb R) \,\,\text{as}\,\,\kappa\to\infty.
\end{align*}
Since
\begin{align*}
     \|q_{\kappa}(t)-\tilde q_{\kappa}(t)\|_{H^{-1}(\mathbb R)}\lesssim  e^{- \kappa^{10}}+ \kappa^5 \int_0^t \|q_{\kappa}(s)-\tilde q_{\kappa}(s)\|_{H^{-1}(\mathbb R)} ds
\end{align*}
for all $t\geq 0$, Gr\"onwall's inequality yields that
\begin{align*}
     \|q_{\kappa}(T)-\tilde q_{\kappa}(T)\|_{H^{-1}(\mathbb R)}\lesssim  e^{- \kappa^{10}+ \kappa^5}
\end{align*}
so
\begin{align*}
      q_{\kappa}(T)\rightharpoonup  q(T) \qquad \text{weakly in} \,\, H^{-1}(\mathbb R) \,\,\text{as}\,\,\kappa\to\infty.
\end{align*}
As a consequence, for $\kappa$ large enough (depending on $\tilde \delta$)
\begin{align*}
    \left| \langle l, q(T) \rangle -\alpha\right|&\geq \left| \langle \tilde  l, q_\kappa(T) \rangle -\alpha\right|-\left| \langle \tilde l, q_\kappa(T) -q(T) \rangle\right|- \left| \langle l-\tilde l , q(T) \rangle\right|\\
    &> r+10 \delta - \delta - C\left(\|q(0)\|_{ H^{-\frac 12}(\mathbb R)}\right) \tilde \delta\\
    &>r
\end{align*}
by making sure we have chosen $\tilde \delta\ll \left[C\left(\|z\|_{\dot H^{-\frac 12}(\mathbb R)}+R\right)\right]^{-1} \delta$.

The general case follows by scaling. For $\lambda>0$, we consider the new parameters 
\begin{align*}
    z_\lambda(x)&:= \lambda^2 z(\lambda x)\in \dot H^{-\frac 12}(\mathbb R)\qquad\text{with}\quad\|z_\lambda\|_{\dot H^{-\frac 12}(\mathbb R)}=\lambda \|z\|_{\dot H^{-\frac 12}(\mathbb R)},\\
    l_\lambda(x)&:= l(\lambda x) \in  H^{\frac 12}(\mathbb R)\qquad\text{with}\quad\|l_\lambda\|_{\dot H^{\frac 12}(\mathbb R)}=\|l\|_{\dot H^{\frac 12}(\mathbb R)}=1,\\
    r_\lambda&:=\lambda r,\qquad R_\lambda:=\lambda R, \qquad \alpha_\lambda:=\lambda \alpha,\qquad T_\lambda:= \lambda^{-3}T.
\end{align*}
Assuming that $\lambda$ is sufficiently small, we can use the special case of Theorem \ref{sns KdV R} that we proved above and obtain solution to \eqref{KdV} $q_\lambda $ satisfying
\begin{align*}
    \|q_\lambda(0)-z_\lambda\|_{\dot H^{-\frac 12}(\mathbb R)}<R_\lambda\qquad\text{and}\qquad |\langle l_\lambda, q_\lambda(T_\lambda) \rangle-\alpha_\lambda|>r_\lambda.
\end{align*}
Then, taking $q(t,x):= \tfrac{1}{\lambda^{2}}q_\lambda(\tfrac{1}{\lambda^{3}}t, \tfrac{1}{\lambda}x)$ we get that
\begin{align*}
    \|q(0)-z\|_{\dot H^{-\frac 12}(\mathbb R)}&=\tfrac{1}{\lambda}\|q_\lambda(0)-z_\lambda\|_{\dot H^{-\frac 12}(\mathbb R)}<\tfrac{1}{\lambda}R_\lambda=R,\\
    |\langle l, q(T) \rangle-\alpha| &=\tfrac{1}{\lambda}|\langle l_\lambda, q_\lambda(T_\lambda) \rangle-\alpha_\lambda|>\tfrac{1}{\lambda}r_\lambda=r.
\end{align*}
\end{proof}

\section{Symplectic non-squeezing for the KdV flow on the circle}
\label{sec;sns T}

The methodology developed in the previous sections also allows us to obtain a simpler proof of the already known result that the KdV flow on the circle has the symplectic non-squeezing property. Once again, the key is to show symplectic non-squeezing for the $H_\kappa$ flows via finite dimensional approximation and argue that this property is inherited by the KdV flow.

The finite-volume setting affords us several simplifications in comparison to the line case. The most obvious manifestation of the more favorable compact setting is that it permits employing simpler finite-dimensional Hamiltonian systems to approximate the $H_\kappa$ flow; after all, truncation in space is no longer necessary. Instead, we are going to use the flows induced by the Hamiltonians 
$$H_\kappa^N(q)=-16\kappa^5 \alpha(\kappa;P^1_{\leq N}q)+4\kappa^2 \int_{\T} \frac 1 2 q^2 dx,$$
namely
\begin{equation} \tag{$H_\kappa^{N}$} \label{eq:HkN}
    \begin{cases}
    \frac{d}{dt} q_\kappa^{N} = 4\kappa^2 (q_\kappa^{N})'+ 16\kappa^5 P^1_{\leq N}g'(P^1_{\leq N} q_\kappa^{N})\\
     q_\kappa^{N}(0) \in \dot H^{-\frac 12}(\mathbb T)
    \end{cases}
\end{equation}
for $N\in 2^{\mathbb N}$.

Furthermore, the compactness of the circle allows for a much shorter argument regarding the recovery of witnesses to non-squeezing as limits of bounded sequences. More specifically, in the circle setting we can extract a strong $H^{-1}$ subsequential limit from a sequence bounded in $\dot H^{-\frac 12}$, thus avoiding working in the weak topology. This is in direct contrast to the line regime where much of our efforts were focused on understanding the behavior of weak limits under the $H_\kappa$ and KdV flows. 

Before we start, let us point out that on the circle the homogeneous and inhomogeneous Sobolev norms $\dot H^s$ and $H^s$ for $s<0$ are equivalent.

Once again, the first step is the finite-dimensional approximation of the $H_\kappa$ flows. Of course, the immediate question that arises is whether the equations induced by the Hamiltonians $H_\kappa^N$ are well-posed. One readily sees that they bear a strong resemblance to the systems employed in the treatment of the line problem, \eqref{eq:Hkn}. Looking back at Theorem \ref{Hkn wp}, we did not impose any conditions on the lengths of the tori $L_n$, so the theorem can be applied equally well to systems where the lengths change (like in the infinite-volume regime, $L_n\to\infty$) and to systems on a fixed circle (for instance $L_n=1$). The last difference between \eqref{eq:HkN} and \eqref{eq:Hkn} that we need to address is the absence of a low frequency truncation. Nevertheless, this change can only simplify the argument. 


\begin{theorem} [Well-posedness for \ref{eq:HkN}]
\label{HkN wp}
Fix $\kappa \geq 1$. Let $\mathcal H_N:=\{f\in  H^{-1}(\mathbb T): P^1_{>2N} f=0\}$, $N\in 2^{\mathbb N}$. There exists $\delta_0>0$ small enough (independent of $N$ and $\kappa$) and a sequence $T_N>0$ satisfying $\lim_{N\to\infty} T_N=\infty$ so that the following are true for every $N\in\mathbb N$: For every $u_{N,0}\in \{f\in \mathcal H_N: \kappa^{-\frac 12} \|f\|_{\dot H^{-\frac 12}(\mathbb T)}\leq \delta_0\}\subset B_{\delta_0, \kappa}^{1}$ there exists a unique solution $u_N\in C([0,T_N]) H^{-1}(\mathbb T)$ to the equation \eqref{eq:HkN}. For each such initial data $u_{N,0}$ the solution $u_N(t)\in \mathcal H_N$ obeys 
\begin{align}\label{H-1 N trunc}
  \|u_N(t)\|_{ H^{-1}(\mathbb T)}\lesssim \|u_{N,0}\|_{\dot H^{-\frac 12}(\mathbb T)}
\end{align}
and
\begin{align} \label{H-1/2 N trunc}
  \|u_N(t)\|_{\dot H^{-\frac 12}(\mathbb T)}\lesssim \|u_{N,0}\|_{\dot H^{-\frac 12}(\mathbb T)}e^{ct},
\end{align}
for all $t\in[0,T_N]$. The implicit constants here do not depend on $N$.

\end{theorem}

We begin by proving the following lemma, which will be applied repeatedly in this and the subsequent sections.

\begin{lemma} \label{moving centers}
Let $z_n\in H^{-1}(\T)$, $z\in H^{-1}(\T)$ and  $0<R_n<\infty$, $0<R<\infty$ such that
$$\lim_{n\to\infty} \|z_n-z\|_{ H^{-1}(\mathbb T)}=0,\quad \lim_{n\to\infty}R_n=R.$$
Let $q_n$, $n\in \mathbb N$, be a sequence in $ H^{-1}(\T)$ such that
$$\|q_n-z_n\|_{\dot H^{-\frac 12}(\mathbb T)}<R_n\quad\text{for all}\quad n\in\mathbb N.$$
Then there exists $q\in  H^{-1}(\T)$ such that, passing to a subsequence, $q_n$ converges to $q$ (strongly) in $H^{-1}$. Moreover,
$$\|q-z\|_{\dot H^{-\frac 12}(\mathbb T)}\leq R.$$
\end{lemma}

\begin{proof}
We look at the sequence
$$f_n:=q_n-z_n \in \dot H^{-\frac 1 2}(\mathbb T).$$
This sequence is restricted in a compact subset of $\dot H^{-\frac 12}(\T)$, since we have that $\|f_n\|_{\dot H^{-\frac 1 2}(\mathbb T)}< R_n\leq R+1$ for $n$ sufficiently large. Therefore, there exists some $f\in \dot H^{-\frac 1 2}(\mathbb T)$ such that, passing to a subsequence, $f_n$ converges to $f$ strongly in $H^{-1}(\mathbb T)$ and weakly in $\dot H^{-\frac 1 2}(\mathbb T)$. We then take 
$$q:=f+z\in H^{-1}(\mathbb T).$$ 
Observing that
$$q_n-q= f_n-f +z_n-z,$$
we conclude that it converges to $0$ strongly in $H^{-1}$.
In addition,
$$\|q-z\|_{\dot H^{-\frac 1 2}(\mathbb T)}=\|f\|_{\dot H^{-\frac 1 2}(\mathbb T)}\leq R.$$
\end{proof}

\begin{theorem}
\label{sns Hk T}
Fix $\kappa\geq 1$. Let $z\in \dot H^{-\frac 12}(\mathbb T)$ with $\|z\|_{\dot H^{-\frac 12}(\mathbb T)}<\frac{1}{10}\kappa^{\frac 12}\delta_0$, $l\in \dot H^s (\mathbb T)$ with $\|l\|_{\dot H^{\frac 1 2}}=1$, $\alpha\in \mathbb C$, $0<r<R<\frac{1}{10}\kappa^{\frac 12}\delta_0$, and $T>0$. Then there exists $q_{\kappa,0} \in \{ q\in \dot H^{-\frac 12}(\mathbb T) : \|q-z\|_{\dot H^{-\frac 1 2}}< R\}$ such that the solution $q_\kappa$ to \ref{eq:Hk} with initial data $q_\kappa(0)=q_{\kappa,0}$ satisfies 
$$|\langle l, q_\kappa(T)\rangle -\alpha|>r.$$
\end{theorem}

\begin{proof}
We begin by taking $0<\delta<\frac{1}{5}(R-r)$. Thanks to Theorem \ref{HkN wp}, Gromov's theorem can be applied to the flow induced by the Hamiltonian $H_\kappa^N$ for initial data in $\{q\in \dot H^{-\frac 12}(\mathbb T):q\in \mathcal H_N, \|q\|_{\dot H^{-\frac 12}}< \kappa^{\frac 12}\delta_0\}$, for appropriately adapted parameters as long as $N$ is sufficiently large so that $T_N>T$. As a result, we obtain a sequence of witnesses $q_\kappa^N(0)\in \dot H^{-\frac 12}(\mathbb T)$, $N\in 2^\mathbb N$, with $P_{>2N}q_\kappa^N(0)=0$ such that the solution $q_\kappa^N$ to the $H_\kappa ^N$ flow with initial data $q_\kappa^N(0)$ satisfies
\begin{align*}
    \|q_\kappa^N(0)-P_{\leq N}z\|_{\dot H^{-\frac 1 2} (\mathbb T)}&< R-\delta,\\
    |\langle q_\kappa ^N(T),l_N\rangle -\alpha|&> r+3 \delta,
\end{align*}
where $l_N=\frac{1}{\|P_{\leq N} l\|_{\dot H^{\frac 1 2}(\mathbb T)}}P_{\leq N} l$. 
Note that
$$\|l-l_N\|_{\dot H^{\frac 12}(\mathbb T)}\leq \left|\frac{1}{\|P_{\leq N} l\|_{\dot H^{\frac 1 2}(\mathbb T)}}-1 \right| \|l\|_{\dot H^{\frac 12}(\mathbb T)} + \|P_{> N} l\|_{\dot H^{\frac 12}(\mathbb T)},$$
so 
\begin{align}\label{l limit}
    \lim_{N\to\infty}\|l-l_N\|_{\dot H^{\frac 12}(\mathbb T)}=0.
\end{align}

We intend to use this sequence of witnesses to construct a witness for the $H_\kappa$ flow. By Lemma \ref{moving centers}, passing to a subsequence, $q_\kappa^N(0)$ converges strongly in $H^{-1}$ to some $q_\kappa(0)\in \dot H^{-\frac 12}(\mathbb T)$ satisfying
$$\|q_\kappa(0)-z\|_{\dot H^{-\frac 1 2}(\mathbb T)}<R.$$
For every $N$ sufficiently large we consider the solution $q_\kappa^N\in C([0,T])\dot H^{-{\frac 12}}(\mathbb T)$ to \eqref{eq:HkN} with initial data $q_\kappa^N(0)$ and the solution $q_\kappa\in C([0,T])\dot H^{-{\frac 12}}(\mathbb T)$ to \eqref{eq:Hk} with initial data $q_\kappa(0)$. 
Using Duhamel's formula we get for all $t\in [0,T]$
\begin{align*}
    \|q_\kappa^N(t)- q_\kappa(t)\|_{ H^{-1}(\mathbb T)}
    \lesssim & \|q_\kappa^N(0)-q_\kappa(0)\|_{ H^{-1}(\mathbb T)} \\
    &+ 16\kappa^5 \int_0^t \|P_{\leq N}g'(P_{\leq N} q_\kappa^N(\tau))-g'(q_\kappa(\tau))\|_{ H^{-1}(\mathbb T)} d\tau.
\end{align*}
By Bernstein estimates, the diffeomorphism property of $g$, and estimate \eqref{H-1/2 conservation}, for all $s\in [0,t]$
\begin{align*}
    \|P_{\leq N}g'(P_{\leq N} & q_\kappa^N(\tau))- g'(q_\kappa(\tau))\|_{H^{-1}(\mathbb T)} \\
    &\lesssim N^{-{\frac 12}}\|q_\kappa(0)\|_{\dot H^{-{\frac 12}}(\mathbb T)} + \|q_\kappa^N(\tau)-q_\kappa(\tau)\|_{ H^{-1}(\mathbb T)},
\end{align*}
so by Gr\"onwall we get that for all $t\in[0,T]$
$$\|q_\kappa^N(t)- q_\kappa(t)\|_{ H^{-1}(\mathbb T)}\lesssim \|q_\kappa^N(0)- q_\kappa(0)\|_{ H^{-1}(\mathbb T)}+ N^{-{\frac 12}}  \| q_\kappa(0)\|_{\dot H^{-{\frac 12}}(\mathbb T)}.$$
In particular, this implies that $q_\kappa^N(T)\to q_\kappa(T)$ in $H^{-1 }(\mathbb T)$ as $N\to \infty$. Moreover, due to \eqref{H-1/2 N trunc} we have that 
\begin{align}\label{bound}
    \|q_\kappa^N(T)\|_{\dot H^{-{\frac 12}}(\mathbb T)}\lesssim \|q_\kappa^N(0)\|_{\dot H^{-{\frac 12}}(\mathbb T)}\lesssim \|z\|_{\dot H^{-{\frac 12}}(\mathbb T)}+R ,
\end{align}
therefore, passing to a subsequence, uniqueness of limits asserts that 
\begin{align}\label{weak limit}
    q_\kappa^N(T)\rightharpoonup q_\kappa(T)\qquad\text{weakly in}\,\,\dot H^{-{\frac 12} }(\mathbb T).
\end{align}
Then for $N$ sufficiently large
\begin{align*}
    |\langle q_\kappa(T), l\rangle -\alpha|&\geq |\langle q_\kappa^N(T), l_N\rangle -\alpha|- |\langle q_\kappa^N(T), l- l_N\rangle|- |\langle q_\kappa(T)-q_\kappa^N(T), l\rangle|\\
    &\geq r+3 \delta- 2\delta>r
\end{align*}
by \eqref{bound}, \eqref{l limit}, and \eqref{weak limit}. This concludes the proof of Theorem \ref{sns Hk T}.
\end{proof}

We are now ready to use the result of Theorem \ref{sns Hk T} to prove Theorem \ref{sns KdV T}.

\begin{proof}[Proof of Theorem \ref{sns KdV T}]
Once again, we take $0<\delta<\frac{1}{5}(R-r)$. Let $\tilde l\in  H^{{\frac 12}}(\mathbb T)\cap C_c^\infty(\mathbb T)$ such that $\|\tilde l\|_{\dot H^{\frac 12}(\mathbb T)}=1$ and $\|l-\tilde l\|_{ H^{\frac 12}(\mathbb T)}<\tilde \delta$ for some small parameter $\tilde \delta>0$ to be chosen later. We only consider $\kappa\geq 1$ sufficiently large so that $\|z\|_{\dot H^{-{\frac 12}}(\mathbb T)}<\frac {1}{10}\kappa^{\frac 12} \delta_0$ and $R<\frac {1}{10}\kappa^{\frac 12} \delta_0$. For all such $\kappa$, let $q_\kappa$ denote the witness to symplectic non-squeezing for the flow induced by $H_\kappa$ that we obtain from Theorem \ref{sns Hk T} for the radii $0<r+2\delta<R-\delta <\infty$ and the parameters $z, \tilde l, \alpha, T$. Since
$$\|q_\kappa(0)-z\|_{\dot H^{-\frac 1 2}(\mathbb T)}<R-\delta\qquad\text{for all}\,\, \kappa, $$
Lemma \ref{moving centers} guarantees that there exists $q(0)\in \dot H^{-{\frac 12}}(\mathbb T)$ with 
$$\|q(0)-z\|_{\dot H^{-\frac 1 2}(\mathbb T)}<R$$
such that, passing to a subsequence, 
$q_\kappa(0)$ converges to $q(0)$ strongly in $H^{-1}(\mathbb T)$ and weakly in $\dot H^{-{\frac 12}}(\mathbb T)$ as $\kappa\to\infty$. 

The next step will be to replace the sequence $q_\kappa(0)$ by a sequence of Schwartz initial data that converges to the same limit. For each $\kappa$ we can find a Schwartz function $\tilde q_\kappa(0)\in \dot H^{-{\frac 12}}(\mathbb T)$ such that 
$$\|q_\kappa(0)-\tilde q_\kappa(0)\|_{\dot H^{-{\frac 12}}(\mathbb T)}<e^{-\kappa^{10}}.$$
Note that $\|\tilde q_\kappa(0)\|_{\dot H^{-{\frac 12}}(\mathbb T)}<\|z\|_{\dot H^{-{\frac 12}}(\mathbb T)}+ R$ for $\kappa$ large. Then clearly passing to a subsequence $\tilde q_\kappa(0)$ converges to $q(0)$ strongly in $H^{-1}(\mathbb T)$ and weakly in $\dot H^{-\frac 12}(\mathbb T)$ as $\kappa\to\infty$. 

Working similarly as in the proof of Theorem \ref{sns KdV R}, we observe that the difference of the solutions $\tilde q_\kappa$ to \eqref{eq:Hk} with initial data $\tilde q_\kappa(0)$ and $q$ to (KdV) with initial data $q(0)$ can be written as
\begin{align*}
    q(T)- \tilde q_\kappa(T)= & \left(e^{tJ\nabla H_{KdV}}q(0)- e^{tJ\nabla H_{KdV}} \tilde q_\kappa(0)\right)\\
    &+\left( e^{tJ\nabla H_{KdV}}\tilde q_\kappa(0)- e^{tJ\nabla H_{\kappa}}\tilde q_\kappa(0)\right).
\end{align*}
Up to a subsequence, as $\kappa\to\infty$ the first term converges to $0$ in $H^{-1}$ by Theorem \ref{KdV wp} and the second term converges to $0$ in $H^{-1}$ by an application of \eqref{main} (which holds for the solutions in question due to the boundedness of $\tilde q_\kappa(0)$ in $\dot H^{-\frac 12}(\mathbb T)$), so $\tilde q_\kappa(T)\to q(T)$ in $H^{-1}(\mathbb T)$ as $\kappa\to\infty$. Moreover, by Duhamel's formula the difference between the solutions to the $H_\kappa$ flow with initial data $q_\kappa(0)$ and $\tilde q_\kappa(0)$ at time $t\geq 0$ satisfies
\begin{align*}
    \|\tilde q_\kappa(t)-q_\kappa(t)\|_{ H^{-1}}&\leq \|\tilde q_\kappa(0)- q_\kappa(0)\|_{H^{-1}}+16\kappa^5 \int_0^t\|g'(\tilde q_\kappa(s))- g'( q_\kappa(s))\|_{ H^{-1}} ds\\
    & \lesssim e^{-\kappa^{10}}+ 16 \kappa^5 \int_0^t\|\tilde q_\kappa(s)- q_\kappa(s)\|_{H^{-1}} ds
\end{align*}
with the implicit constant independent of $\kappa$. By Gr\"onwall's inequality we obtain that
$$\|q_\kappa(T)-q_\kappa(T)\|_{ H^{-1}(\mathbb T)}< e^{CT \kappa^5 -\kappa^{10}},$$
which suggests that $q_\kappa(T)\to q(T)$ in $ H^{-1}(\mathbb T)$ as $\kappa\to\infty$. 
Consequently, for $\kappa$ sufficiently large
\begin{align*}
    |\langle q(T), l\rangle -\alpha|&\geq |\langle q_\kappa(T),\tilde l\rangle -\alpha|- |\langle q_\kappa(T)-q(T), \tilde l\rangle|- |\langle q(T),\tilde l-l\rangle|\\
    &> r+2 \delta-\|q_\kappa(T)-q(T)\|_{H^{-1}}\|\tilde l\|_{H^1}- \|q(T)\|_{\dot H^{-\frac 12}}\|l-\tilde l\|_{\dot H^s}\\
    &>r+2 \delta-\|q_\kappa(T)-q(T)\|_{H^{-1}}\|\tilde l\|_{H^1}- C(\|z\|_{\dot H^{-\frac 12}}+R)\tilde \delta\\
    &>r
\end{align*}
provided that $\tilde \delta$ has been chosen appropriately small, thus completing the proof that $q$ is, indeed, a witness to symplectic non-squeezing for (KdV).
\end{proof}

\section{Proof of Theorems \ref{sns KdV R ell} and \ref{sns KdV T ell} }

\begin{proof}[Proof of Theorem \ref{sns KdV T ell}]
We are given $z\in H^{-1}(\mathbb T)$, $l\in H^1(\mathbb T)$ with $\|l\|_{\dot H^{\frac 12}(\mathbb T)}=1$, $\alpha\in \mathbb{C}$, $0<r<R<\infty$, and $T>0$. Let us take a small parameter $0<\delta<\frac{R-r}{5}$ and consider a sequence $z_n\in H^{-1}(\mathbb T)\cap \dot H^{-\frac 12}(\mathbb T)$ such that $\|z-z_n\|_{H^{-1}(\mathbb T)}\to 0$ as $n\to\infty$. Then, by Theorem \ref{sns KdV T}, we get solutions $q_n$ to \eqref{KdV} with initial data $q_n(0)\in \dot H^{-\frac 12}(\mathbb T)$ satisfying
\begin{align*}
    \|q_n(0)-z_n\|_{\dot H^{-\frac 12}(\mathbb T)}<R-\delta,\qquad \left|\langle l, q_n(T) \rangle-\alpha\right|>r+2\delta.
\end{align*}
Lemma \ref{moving centers} provides us with $q(0)\in H^{-1}(\mathbb T)$ with $\|q(0)-z\|_{\dot H^{-\frac 12}(\mathbb T)}<R$ such that, passing to a subsequence, $q_n(0)$ converges to $q(0)$ in $H^{-1}(\mathbb T)$ as $n\to\infty$. The continuity of the KdV flow in $H^{-1}$ then guarantees that $q_n(T)$ converges in $H^{-1}(\mathbb T)$ to the solution $q(T)$ to \eqref{KdV} with initial data $q(0)$, as $n\to\infty$. Therefore,
\begin{align*}
    \left| \langle l, q(T) \rangle -\alpha \right|& \geq  \left| \langle l, q_n(T) \rangle -\alpha \right| -  \left| \langle l, q_n(T)- g(T) \rangle \right|\\
    & > r+2\delta - \|l\|_{H^1(\mathbb T)}\|q_n(T)-q(T)\|_{H^{-1}(\mathbb T)}\\
    &>r.
\end{align*}
\end{proof}

\begin{proof}[Proof of Theorem \ref{sns KdV R ell}]
We follow an argument parallel to the one in the proof of Theorem \ref{sns KdV T ell}. Let us focus only on the aspects that require a modifies treatment.

The first point of divergence is the extraction of a subsequential limit $q(0)$ from $q_n(0)$ by an application of Lemma \ref{moving centers}. Here, we need to use a variant of that Lemma for the line; the same argument can work in the non-compact setting, the sole difference being that it can give us only weak convergence. That way we obtain $q(0)\in H^{-1}(\mathbb R)$ with $\|q(0)-z\|_{\dot H^{-\frac 12}(\mathbb R)}<R$ such that, passing to a subsequence, $q_n(0)$ converges to $q(0)$ weakly in $H^{-1}(\mathbb R)$. 

Next, instead of strong $H^{-1}$ convergence of $q_n(T)$, we can aim for the analogous weak result. The only obstacle on our way towards this is to show that the set $Q=\{q_n(0):n\in \mathbb N\}$ is equicontinuous in $H^{-1}(\mathbb R)$; then we will be able to apply Theorem \ref{KdV weak} and carry out the rest of the argument the same way we did for the circle case. However, one can readily see that the equicontinuity of $Q$ is guaranteed by the fact that $Q\subset Q_1+Q_2$, where
\begin{align*}
    Q_1=\{q_n(0)-z_n:n\in\mathbb N\}, \qquad Q_2=\{z_n:n\in\mathbb N\},
\end{align*}
and Lemma \ref{Qc}.
\end{proof}

\section{Equivalent formulations of symplectic non-squeezing} \label{sec;equivalents}

In the introduction we listed two other equivalent formulations of Gromov's Theorem. In \cite{KVZ16} the corresponding alternative expressions of symplectic non-squeezing are derived for the cubic NLS on $\mathbb R^2$. It is not suprising that the analogous statements are true for the KdV equation, both in the circle and in the line setting. Below we will present the two equivalent formulations of Theorem \ref{sns KdV T} and Theorem \ref{sns KdV T ell}.  The proofs of the respective results on the line can be reconstructed from the following arguments and the proofs of Theorems \ref{sns KdV R} and \ref{sns KdV R ell}.

\begin{corollary}
Fix $z\in  \dot H^{-\frac 12}(\mathbb T)$, $l\in H^{\frac 12} (\mathbb T)$ with $\|l\|_{\dot H^{\frac 1 2}}=1$, $\alpha\in \mathbb C$, $0<R<\infty$, and $T>0$. Then there exists solution $q\in  H^{-\frac 12}(\T)$ to (KdV) such that
\begin{align*}
    \|q(0)-z\|_{\dot H^{-\frac 1 2}(\mathbb T)}\leq R,\quad |\langle l, q(T)\rangle -\alpha|\geq R.
\end{align*}
\end{corollary}

\begin{proof}
Let $q_n\in \dot H^{-\frac 12}(\T)$ be the witnesses to symplectic non-squeezing for \eqref{KdV} given by Theorem \ref{sns KdV T} for the radii $0<R-\frac 1 n<R+\frac 1 n<\infty$, for $n\in \mathbb N$ sufficiently large. We can find a subsequence along which $q_n(0)$ converges to some $q(0)\in \dot H^{-\frac 12}(\mathbb T)$ in $H^{-1}(\T)$. Moreover, for this limit we have that 
$$\|q(0)-z\|_{\dot H^{-\frac 12}(\mathbb T)}\leq R.$$
By the continuity of the solution map for \eqref{KdV}, $q_n(T)$ converges to $q(T)$ in $H^{-1}(\T)$. In addition, the fact that 
$$\|q_n(T)\|_{H^{-\frac 12}(\mathbb T)}<C\left(\|q_n(0)\|_{H^{-\frac 12}(\mathbb T)}\right)<C\left(\|z\|_{\dot H^{-\frac 12}(\mathbb T)}+R\right)$$
and uniqueness of limits ensure that passing to a subsequence $q_n(T)$ converges to $q(T)$ weakly in $\dot H^{-\frac 12}$
Then for any $\varepsilon>0$
\begin{align*}
    |\langle q(T), l\rangle -\alpha|&\geq |\langle  q_n(T), l\rangle -\alpha|- |\langle  q_n(T)-q(T), l\rangle|\\
    &>R-\varepsilon
\end{align*}
for $n$ large enough, so we conclude that 
$$ |\langle q(T), l\rangle -\alpha|\geq R.$$
\end{proof}

\begin{corollary}
Fix $z\in  H^{-1}(\mathbb T)$, $l\in H^1 (\mathbb T)$ with $\|l\|_{\dot H^{\frac 1 2}}=1$, $\alpha\in \mathbb C$, $0<R<\infty$, and $T>0$. Then there exists solution $q\in  H^{-1}(\T)$ to (KdV) such that
\begin{align*}
    \|q(0)-z\|_{\dot H^{-\frac 1 2}(\mathbb T)}\leq R,\quad |\langle l, q(T)\rangle -\alpha|\geq R.
\end{align*}
\end{corollary}

\begin{proof}
Let $q_n\in  H^{-1}(\T)$ be the witnesses to symplectic non-squeezing for \eqref{KdV} that we obtain from Theorem \ref{sns KdV T ell} for the radii $0<R-\frac 1 n<R+\frac 1 n<\infty$, for $n\in \mathbb N$ sufficiently large. Once again, we take advantage of Lemma \ref{moving centers}, which provides us with $q(0)\in  H^{-1}(\T)$ such that $q_n(0)$ converges to $q(0)$ in $H^{-1}(\T)$, and satisfies 
$$\|q(0)-z\|_{\dot H^{-\frac 12}(\mathbb T)}\leq R.$$
By the continuity of the solution map for \eqref{KdV}, $q_n(T)$ converges to $q(T)$ in $H^{-1}(\T)$.
This suggests that for any $\varepsilon>0$
\begin{align*}
    |\langle q(T), l\rangle -\alpha|&\geq |\langle  q_n(T), l\rangle -\alpha|- |\langle  q_n(T)-q(T), l\rangle|\\
    &>R-\varepsilon
\end{align*}
for $n$ large enough, so
$$ |\langle q(T), l\rangle -\alpha|\geq R.$$
\end{proof}

\begin{corollary} \label{area -1/2}
For every $z\in  \dot H^{-\frac 12}(\mathbb T)$, $l\in  H^\frac 12 (\mathbb T)$ with $\|l\|_{\dot H^{\frac 1 2}(\mathbb T)}=1$, $0<R<\infty$ and $T>0$,
$$\area(\{\langle l, q(T)\rangle: q\,\,\text{solves \eqref{KdV}},\,\,q(0)\in \dot H^{-\frac 12}(\mathbb T), \|q(0)-z\|_{\dot H^{-\frac 12}(\mathbb T)}<R\})\geq \pi R^2.$$
\end{corollary}

\begin{proof}
Without loss of generality, we may assume that $\|z\|_{\dot H^{-\frac 12}(\mathbb T)}<\frac{\delta_0}{10}$ and $R<\frac{\delta_0}{10}$; the general case then follows by rescaling. Under these extra assumptions we ensure the well-posedness of \eqref{eq:HkN} on $\{q\in 
\mathcal H_N: \|q-z\|_{\dot H^{-\frac 12}(\mathbb T)}<R\}$ up to time $T_N$, where $T_N\to\infty$ as $N\to\infty$. 

Our goal is to prove a bound for the area of the set
$$\mathcal O:=\{\langle l, q(T)\rangle: q\,\,\text{solves \eqref{KdV}},\,\,q(0)\in \dot H^{-\frac 12}(\mathbb T), \|q(0)-z\|_{\dot H^{-\frac 12}(\mathbb T)}<R\}.$$
We define for $0<r<R$ the sets
$$\mathcal K_r:=\{\langle l, q(T)\rangle: q\,\,\text{solves \eqref{KdV}},\,\,q(0)\in \dot H^{-\frac 12}(\mathbb T), \|q(0)-z\|_{\dot H^{-\frac 12}(\mathbb T)}\leq r\}.$$
For any fixed $0<r<R$,
$$|\langle l, q(T)\rangle|\leq \|l\|_{H^\frac 12(\mathbb T)} \|q(T)\|_{ H^{-\frac 12}(\mathbb T)}< \|l\|_{ H^\frac 12(\mathbb T)} C(\|z\|_{\dot H^{-\frac 12}(\mathbb T)}+r)$$
so each $\mathcal K_r$ is bounded. One can establish that each $\mathcal K_r$ is also closed, hence compact. Note also that
$$\mathcal O=\bigcup_{0<r<R} \mathcal K_r.$$
Suppose that $\area (\mathcal O)<\pi R^2$. Then there exists $\delta >0$ such that
$$\area(\mathcal B_\delta):=\area (\{z\in \mathbb C: \dist(z, \mathcal K_{R-\delta})\leq 4\delta\})<\pi (R-4\delta)^2.$$
We will use the analogous formulation of Gromov's theorem for a sequence of finite dimensional systems that approximate our flow and follow the arguments used in the proofs of Theorems \ref{sns Hk T} and \ref{sns KdV T} to reach a contradiction. First of all, for each $\kappa\geq 1$ and $N\in 2^{\mathbb N}$ we consider the set
$$\mathcal A_\kappa^N :=\{\langle l_N, q(T)\rangle: q\,\,\text{solves \eqref{eq:HkN}},\,\,q(0) \in \dot H^{-\frac 12}(\mathbb T), \|q(0)-z_N\|_{\dot H^{-\frac 1 2}(\mathbb T)}<R-4\delta\}$$
where $z_N = P_{\leq N}z \in \dot H^{-\frac 1 2}(\mathbb T)$, $l_N=\frac{1}{\|P_{\leq N} l\|_{\dot H^{\frac 12}(\mathbb T)}} P_{\leq N} l \in \dot H^{\frac 1 2}(\mathbb T)$. By the equivalent formulation of Gromov's theorem stated in the introduction, we get that
$$\area{(\mathcal A_\kappa^N)}\geq\pi (R-4\delta)^2.$$
If all $q$ that solve $H_\kappa^N$ with initial data $q(0) \in \dot H^{-\frac 12}(\mathbb T)$, $\|q(0)-z_N\|_{\dot H^{-\frac 1 2}(\mathbb T)}<R-4\delta$ satisfied that $\dist(\langle l_N, q(T)\rangle, \mathcal K_{R-\delta})\leq 4\delta$, then we would get that $\mathcal A_\kappa^N\subset \mathcal B_\delta$, which would in turn yield
$$\area(\mathcal A_\kappa^N)\leq \area (\mathcal B_\delta)<\pi (R-4\delta)^2,$$
resulting in a contradiction. This asserts that, for each $\kappa\geq 1$ and $N\in 2^{\mathbb N}$, there exists $q_\kappa^N$ that solves $(H_\kappa^N)$ with $q_\kappa^N(0)\in \dot H^{-\frac 12}(\mathbb T)$,
$$\|q_\kappa^N(0)-P_{\leq N}z\|_{\dot H^{-\frac 1 2}(\mathbb T)}<R-4\delta, \quad   \dist(\langle l_N, q_\kappa^N(T)\rangle, \mathcal K_{R-\delta})> 4\delta.$$
Using Lemma \ref{moving centers} as in the proof of Theorem \ref{sns Hk T}, we obtain $q_\kappa(0)\in  H^{-1}(\T)$ such that $q_\kappa^N(0)$ converges to $q_\kappa(0)$ in $ H^{-1}(\mathbb T)$ as $N\to\infty$, which also satisfies 
$$\|q_\kappa(0)-z\|_{\dot H^{-\frac 1 2}}<R-3\delta.$$
Moreover, an argument similar to the one in the proof of Theorem \ref{sns Hk T} gives us for the solution $q_\kappa$ to \eqref{eq:Hk} with initial data $q_\kappa(0)$
$$\dist(\langle l, q_\kappa(T)\rangle, \mathcal K_{R-\delta})> 3\delta.$$
Next, we use Lemma \ref{moving centers} again as in the proof of Theorem \ref{sns KdV T} to get $q(0)\in \dot H^{-\frac 12}(\T)$ such that $q_\kappa(0)$ converges to $q(0)$ in $ H^{-1}(\mathbb T)$ as $\kappa\to\infty$, with
$$\|q(0)-z\|_{\dot H^{-\frac 1 2}}<R-2\delta.$$
Imitating the proof of Theorem \ref{sns KdV T} once again, we can see that $q_\kappa(T)$ converges in $H^{-1}(\T)$ to the solution $q(T)$ to \eqref{KdV} with initial data $q(0)$ at time $T$, so
$$\dist(\langle l, q(T)\rangle, \mathcal K_{R-\delta})> 2\delta.$$
This results in a contradiction, since we found $z=\langle l, q(T)\rangle\in \mathcal K_{R-\delta}$ so that $\dist(z, \mathcal K_{R-\delta})> 2\delta$. We conclude that
$$\area(\mathcal O)\geq \pi R^2.$$

\end{proof}

\begin{corollary}
For every $z\in  H^{-1}(\mathbb T)$, $l\in  H^1 (\mathbb T)$ with $\|l\|_{\dot H^{\frac 1 2}}=1$, $0<R<\infty$ and $T>0$,
$$\area(\{\langle l, q(T)\rangle: q\,\,\text{solves \eqref{KdV}},\,\,q(0)\in \dot H^{-s}(\mathbb T), \|q(0)-z\|_{\dot H^{-\frac 12}(\mathbb T)}<R\})\geq \pi R^2.$$
\end{corollary}

\begin{proof}
Once again, we consider
$$\mathcal O:=\{\langle l, q(T)\rangle: q\,\,\text{solves \eqref{KdV}},\,\,q(0)\in  H^{-1}(\mathbb T), \|q(0)-z\|_{\dot H^{-\frac 12}(\mathbb T)}<R\}$$
and for $0<r<R$
$$\mathcal K_r:=\{\langle l, q(T)\rangle: q\,\,\text{solves \eqref{KdV}},\,\,q(0)\in H^{-1}(\mathbb T), \|q(0)-z\|_{\dot H^{-\frac 12}(\mathbb T)}\leq r\}.$$
Assuming that $\area (\mathcal O)<\pi R^2$, we find $\delta >0$ such that
$$\area(\mathcal B_\delta):=\area (\{z\in \mathbb C: \dist(z, K_{R-\delta})\leq 4\delta\})<\pi (R-4\delta)^2.$$
We can find a sequence $z_n\in H^{-1}(\mathbb T)\cap \dot H^{-\frac 12}(\mathbb T)$ such that $z_n \to z$ in $H^{-1}(\mathbb T)$. Arguing as before and applying Corollary \ref{area -1/2} for the parameters $z_n$, we obtain for each $n$ some solution to \eqref{KdV} $q_n$ that satisfies
$$\|q_n(0)-z_n\|_{\dot H^{-\frac 1 2}}<R-4\delta, \quad   \dist(\langle l, q_n(T)\rangle, \mathcal K_{R-\delta})> 4\delta.$$
Lemma \ref{moving centers} provides us with $q(0)\in H^{-1}(\T)$ with $\|q(0)-z\|_{\dot H^{-\frac 1 2}}<R-3\delta$ such that passing to a subsequence $q_n(0)$ converges to $q(0)$ in $H^{-1}(\mathbb T)$ as $n\to\infty$.
Then $q_n(T)$ converges to $q(T)$ in $H^{-1}(\mathbb T)$, yielding that
$$\dist(\langle l, q(T)\rangle, \mathcal K_{R-\delta})> 3\delta.$$
Since $\langle l, q(T)\rangle\in \mathcal K_{R-\delta}$, this is a contradiction. We conclude that
$$\area(\mathcal O)\geq \pi R^2.$$
\end{proof}

\bibliographystyle{amsplain}

\end{document}